\documentclass{ucrthesis-1chair}

\usepackage[mathscr]{eucal}
\usepackage{amsfonts}
\usepackage{amsmath}
\usepackage{amsthm}
\usepackage{amssymb}
\usepackage{latexsym}
\usepackage{svg}
\include{svg.dtx}
\usepackage{graphicx}
\usepackage{color}
\usepackage[noadjust]{cite}
\usepackage{epsfig}
\usepackage{tikz-cd}
\usepackage{float}
\allowdisplaybreaks

\newtheorem{theorem}{Theorem}[section]
\newtheorem{proposition}[theorem]{Proposition}
\newtheorem{corollary}[theorem]{Corollary}
\newtheorem{lemma}[theorem]{Lemma}
\theoremstyle{definition}
\newtheorem{definition}[theorem]{Definition}
\newtheorem{remark}[theorem]{Remark}

\newtheorem{example}[theorem]{Example}
\newtheorem{conjecture}[theorem]{Conjecture}

\let\nc\newcommand

\nc{\la}{\label}

\nc{\End}{{\rm{End}}}
\nc{\Hom}{{\rm{Hom}}}
\newcommand{\id}{{\rm{id}}}

\newcommand{\cn}{\mathrm{cn}}
\newcommand{\cl}{\mathrm{cl}}

\newcommand{\im}{{\rm{Im}}}

\newcommand{{\scrc}}{\mathscr{C}}
\newcommand{{\sfc}}{\mathsf{C}}
\newcommand{{\sfr}}{\mathsf{R}}
\newcommand{{\sfskein}}{\mathsf{Skein}}
\newcommand{{\sft}}{\mathsf{T}}
\newcommand{{\sfd}}{\mathsf{D}}
\newcommand{{\sfh}}{\mathsf{H}}
\newcommand{{\cs}}{\mathcal{S}}
\newcommand{{\cc}}{\mathcal{C}}
\newcommand{{\ct}}{\mathcal{T}}
\newcommand{{\ca}}{\mathcal{A}}
\newcommand{{\ck}} {\mathcal{K}}
\newcommand{{\cd}} {\mathcal{D}}
\newcommand{{\ch}} {\mathcal{H}}
\newcommand{\ci}{\mathcal{I}}
\newcommand{\xx}{{\mathbf{x}}}
\newcommand{\yy}{{\mathbf{y}}}
\newcommand{\zz}{{\mathbf{z}}}
\newcommand{\aab}{{\mathbf{a}}}
\newcommand{\bb}{{\mathbf{b}}}
\newcommand{\SL}{\mathrm{SL}}

\newcommand{\GL}{\mathrm{GL}}
\newcommand{\N}{\mathbb{N}}
\newcommand{{\Z}}{\mathbb{Z}}
\newcommand{\Q}{\mathbb{Q}}
\newcommand{\C}{\mathbb{C}}

\newcommand{\fa}{\mathfrak{a}}
\newcommand{\fb}{\mathfrak{b}}

\newcommand{\pic}[2][3]{{\,\vcenter{\hbox{\includegraphics[width= #1 cm]{#2}}}}}

\usepackage[alphabetic]{amsrefs}

\usepackage{microtype}				

\title{Dubrovnik Skein Theory and Power Sum Elements}
\author{Alexander T. Pokorny}
\degreemonth{June}			
\degreeyear{2021}
\degreesemester{Spring}			
\degree{Doctor of Philosophy}
\chair{Dr. Peter Samuelson}
\othermembers{%
	Dr. Jacob Greenstein \\
	Dr. Stefano Vidussi
}
\numberofmembers{2}					
\field{Mathematics}
\campus{Riverside}

\begin{document}

\maketitle


\cleardoublepage


\begin{frontmatter}
	\setcounter{secnumdepth}{3}
	\setcounter{tocdepth}{3}





	\begin{abstract}
In this work, we extend some results from the Kauffman bracket and HOMFLYPT skein theories to the Kauffman (Dubrovnik) skein theory. A definition is given for ``power sum" type elements $\widetilde{P}_k$ in the Dubrovnik skein algebra of the annulus $\mathcal{D}(A)$. These elements generalize the Chebyshev polynomials often used when studying Kauffman bracket skein algebras. Threadings of the $\widetilde{P}_k$ are used as generators in a presentation of the Dubrovnik skein algebra of the torus $\mathcal{D}(T^2)$, where they are shown to satisfy simple relations. This description of $\mathcal{D}(T^2)$ is used to describe the natural action of this algebra on the skein module of the solid torus. We give evidence that the universal character rings for the orthogonal and symplectic Lie groups correspond to the skein algebra $\mathcal{D}(A)$ such that the Schur functions of type either B, C or D correspond to annular closures $\widetilde{Q}_\lambda$ of minimal idempotents of the Birman-Murakami-Wenzl algebras $BMW_n$. We also record some miscellaneous applications of the $\widetilde{P}_k$, such as commutation relations for the annular closures of BMW symmetrizers $\widetilde{Q}_{(n)}$ and an expression of central elements of $BMW_n$ in terms of Jucys-Murphy elements.
\end{abstract}
	\cleardoublepage

	\tableofcontents
	\cleardoublepage





\end{frontmatter}


\chapter{Introduction}

\section{History and Development of Skein Theory}

Often in mathematics, the simplest objects are the most difficult to describe. At the very least, this is certainly true for mathematical knots. For over a century, mathematicians and others been puzzled by questions contained in what may be called knot theory. One of the guiding themes in knot theory is the search for new knot invariants, mechanisms defined combinatorially or otherwise, which partition the set of knots into equivalence classes so that each class is distinguishable from the others. 

The latter part of the twentieth century saw a surge of interest in this area with the discovery of the Jones polynomial \cite{Jon85}, a knot invariant which may be calculated by applications of a simple, recursive, diagrammatic formula on knot diagrams called the Kauffman bracket \cite{Kau87}. This caused much excitement for many different reasons. Firstly, applications of the Jones polynomial were used to prove some of the Tait conjectures \cite{Kau87, Thi88} which were century-old open problems in knot theory. Separate from this, connections to physics were discovered where a knot may be regarded as the orbit of a charged particle in three-dimensional spacetime. It was realized that the Jones polynomial of a knot is, in a certain physical model (SU(2) Chern-Simons theory), a quantum averaging of the Wilson operator, which is a ``probability amplitude" that the particle traveled along the orbit \cite{Wit89}. The Jones polynomial is also intimately connected to representation theory, where it may be calculated as the quantum trace of an endomorpism, considered as a braid, of tensor powers of the two-dimensional representation of the Drinfeld-Jimbo quantum group $U_q(\mathfrak{sl}_2)$. There is also a connection to hyperbolic geometry, which has been formulated via the relatively long-standing (generalized) volume conjecture which states that there is a relationship between the complete hyperbolic volume of the complement of a knot in the $3$-sphere and its (colored) Jones polynomial \cite{Hik07}. The conjecture has been verified for certain special cases \cite{KT00}. 

The Jones polynomial has been categorified via Khovanov homology \cite{Kho00} where a link diagram gives rise to a chain complex whose homology is isotopy invariant and whose Euler characteristic is the Jones polynomial of the link. Khovanov homology is important for physics because it allows for a knot to be interpreted as a physical object in four-dimensional spacetime \cite{Wit12}. It is equally important for mathematics as categorification spawns additional structure which can be used to prove things. For example, Khovanov homology can detect the unknot \cite{KM11}. Whether or not the Jones polynomial can detect the unknot remains open.

A key fact about the diagrammatic formulas, generally called a skein relations, which defines the Jones polynomial is that they are defined locally at the crossings of the diagrams. Consequently, one may use the skein relations to define invariants of knots in more general $3$-manifolds $M$. The result is a family of modules, called skein modules, each of which are the set of linear combinations of links in $M$ modulo the skein relations. The direction of this sort of generalization is orthogonal to the direction of categorification, although there is effort going into combining the two ideas \cite{APS04}. The skein module construction is functorial with respect to smooth orientation-preserving embeddings of $3$-manifolds, so the skein relations provide a sort of algebraic topology of smooth embedded submanifolds called a skein theory. Furthermore, the skein module of a trivial interval-bundle over an oriented surface may be equipped with a canonical algebra structure, hence they are called skein algebras. These algebras are of particular interest due to their connections with other algebraic objects arising from low-dimensional topology, geometry, and representation theory. Notably, the skein algebra of a surface is a deformation quantization of the coordinate ring of the $SL_2(\C)$-character variety of the surface with respect to a certain Poisson algebra structure \cite{BFK99}. More recently with the advent of cluster algebras, the skein algebra of a marked surface was shown by to be related to certain quantum cluster algebras of the marked surface if the markings yield triangulations \cite{Mul16}.

Shortly after Jones' discovery of the Jones polynomial, mathematicians quickly defined more general knot polynomials. Where the Jones polynomial may be seen as corresponding to $\mathfrak{sl}_2$, the HOMFLYPT polynomial \cite{FHLMOY85} \cite{PT88} corresponds to $\mathfrak{sl}_n$ or $\mathfrak{gl}_n$ for any $n$, and the Dubrovnik polynomial \cite{Kau90} corresponds to $\mathfrak{sp}_{2n}$ or $\mathfrak{so}_n$ for any $n$. Similar to the Jones polynomial, each of these invariants may be defined and computed using similar recursive linear formulas on knot diagrams. These types of formulas, called skein relations, were first considered by Alexander \cite{Ale28} and popularized by Conway \cite{Con70} in the context of the Alexander polynomial (which corresponds to the Lie superalgebra $\mathfrak{gl}(1|1)$ as above, but we will not discuss this here). The search for knot homologies to extend ideas of Khovanov is an ongoing project among the categorification community. 

The HOMFLYPT and Dubrovnik skein relations each give rise to a skein theory in the same way that the skein relations of the Jones polynomial does. Many of the connections between the Jones polynomial and other areas of mathematics either generalize or are expected to generalize to these other skein theories. For example, the HOMFLYPT and Dubrovnik polynomials provide relationships to physics, just with respect to different Chern-Simons theories. As for categorification, there is presently a large effort to define and understand knot homologies with respect to different knot polynomials \cite{KR08, Web17}. Separate from this, HOMFLYPT and Dubrovnik skein algebras of a surface may be understood as quantizations of the symmetric algebras of the oriented and unoriented Goldman Lie algebras of the surface respectively \cite{Tur91}. 

One difficulty in working with skein algebras is that they are difficult to describe explicitly using generators and relations. In the Kauffman bracket skein theory, the algebras are relatively small making them easier to work with. One typically uses a basis of built from certain Chebyshev polynomials of simple curves due to the surprisingly simple relations they satisfy. The HOMFLYPT skein algebras are much larger since the corresponding skein relation doesn't allow one to eliminate crossings in diagrams like the Kauffman bracket does. Still, the Chebyshev elements admit HOMFLYPT generalizations, called power sum elements \cite{Mor02b}, which satisfy equally simple relations. The name ``power sum element" stems from the fact that the HOMFLYPT skein algebra of the annulus corresponds nicely to the ring of symmetric functions \cite{Luk05}. Under this correspondence, the power sum element in the skein algebra is a power sum symmetric function. 

The Chebyshev elements and power sum elements were used to give beautiful, compact presentations of the respective skein algebras of the torus \cite{FG00, MS17}. In the former case, this shows that the Kauffman bracket skein algebra of the torus is isomorphic to a subalgebra of a specialization of the $\mathfrak{sl}_2$ DAHA (see \cite{MS19} for a survey of this relationship). In the case, Morton and Samuelson exhibit a surprising relationship to the elliptic Hall algebra. This is an algebra which comes from considering extensions in the category of coherent sheaves over a smooth elliptic curve over a finite field \cite{BS12}. Although this relationship is mysterious, the authors suggest this relationship is a shadow of a special case of mirror symmetry. At the very least, this gives motivation to further understand and develop the theory of skein algebras.

\section{Thesis Outline}

In this thesis, we attempt to generalize various facts known about the Kauffman bracket and HOMFLYPT skein theories to the Dubrovnik skein theory. We organize the paper as follows. In Chapter 2, we define precisely what we mean by skein theories and properties thereof. Much of this chapter is a summary of current literature, although the exposition seems to be somewhat unique in that we describe the Kauffman bracket, HOMFLYPT, and Dubrovnik skein theories all in parallel. In any case, it is intended to be a good starting point for someone new to skein theory. We also define a natural transformation from the Dubrovnik skein theory to the Kauffman bracket skein theory. 

The aim of Chapter 3 is to discuss the Dubrovnik skein algebra of the torus, which was first discussed by Morton, Samuelson, and the author \cite{MPS19}. First we define Dubrovnik versions of the HOMFLYPT power sum elements and show that they satisfy very simple relations which describe how a power sum element moves past a single strand. We then use these elements to give a presentation of the skein algebra of the torus. Using the natural transformation above, we show that the Dubrovnik power sum elements project to the Chebyshev elements. We also describe an injective algebra homomorphism from the Dubrovnik skein algebra of the torus to the HOMFLYPT skein algebra of the torus. There is a $\Z/2\Z$-action on the HOMFLYPT skein algebra of the torus, and the image of this homomorphism is the subalgebra invariant with respect to this action. Next, we give a description of the natural action of the Dubrovnik skein algebra of the torus on the skein module of the solid torus. 

In Chapter 4, we compile a few other results that aren't related to the skein algebra of the torus. First, we give a formula which describes both HOMFLYPT and Dubrovnik versions of how other commonly used skein elements, the symmetrizer closures, commute past a single strand. Comparing these formulas with the analagous relations satisfied by the power sum elements really highlight how the simple the latter relations really are. Next, for each type X among B, C, and D, we set up an algebra homomorphism from the ring of symmetric functions to the Dubrovnik skein algebra of the annulus. We conjecture that the Schur functions of type X defined in \cite{KT87} correspond to the basis of the Dubrovnik skein algebra of the annulus defined \cite{LZ02} under this homomorphism. We prove the conjecture holds for Schur functions indexed by partitions of length at most two. Finally, we give formulas for wrapping Dubrovnik power sum elements around braids in the Birman-Murakami-Wenzl (BMW) algebras as meridians. We show how wrapping the identity braid by such a meridian, which is contained in the center of the BMW algebra, admits a simple description in terms of Jucys-Murphy elements. A similar argument shows how the minimal idempotents constructed in \cite{BB01} are eigenvalues of the linear maps defined by the meridian operations. The eigenvalues are given and are all distinct, showing that the eigenspaces are $1$-dimensional, generalizing a result of \cite{LZ02}.

\chapter{Skein Theory} \label{chapter2}

\section{Foundations and General Notions} \label{sec:foundations}
In this work, we will need to discuss a few different variants of skein modules. For this reason, it will be useful to first describe some general framework of skein theory so that each of these variants will be a special case. Unless otherwise stated, we will assume $M$ is an oriented $3$-manifold with boundary $\partial M$ (possibly empty), $\Sigma$ is an oriented surface, $I$ is the real interval $[0,1]$, and $R$ is a commutative and unital ring. 

\begin{definition}
    Let $T_1, T_2: X \to M$ be smooth embeddings of a smooth manifold $X$ into $M$. A \textbf{smooth ambient isotopy} $H: T_1 \Rightarrow T_2$ is a smooth homotopy of diffeomorphisms $H_t: M \to M$ such that $H_0=\id_M$ and $H_1 \circ T_1 = T_2$. Furthermore, we demand that the boundary $\partial M$ is fixed by the homotopy. 
\end{definition}

The relation 
\[
T_1 \sim T_2 \textrm{ if and only if there exists a smooth ambient isotopy } H: T_1 \Rightarrow T_2
\]
is an equivalence relation. The smoothness requirement of $H$ is important when considering knots. Without it, all knots would be equivalent to the unknot by contracting all of the complexity of the knot to a point. 

\begin{definition}
Let $N$ be a finite set of points contained in the boundary $\partial M$. An \textbf{$N$-tangle} in $M$ (or just \textit{tangle} for short) is the smooth ambient isotopy class of a smooth embedding 
\[
T: \underbrace{\underset{j \in J}{\bigsqcup}\, S^1}_{:=L} \sqcup \underbrace{\underset{k \in K}{\bigsqcup}\, I}_{:=B} \to M
\]
for some finite sets $J$ and $K$ such that
\begin{enumerate}
\item the image of $L$ lies in the interior of $M$,
\item the image of the interior of $B$ lies in the interior of $M$,
\item the image of the boundary of $B$ equals $N$.
\end{enumerate}
If $B$ is empty, then the result is called a \textbf{link} in $M$. Similarly, if $L$ is empty, then it's called a \textbf{braid} in $M$. 

One may also consider \textit{oriented} or \textit{framed} tangles by choosing an orientation or framing for each point in $N$ and for each connected component of $L$ and $B$ such that the choices are compatible with each other with respect to the smooth embedding. We define a \textbf{framing} of an tangle $T$ to be an extension to a smooth embedding of $T \times I$ into $M$ up to ambient isotopy, so that each connected component of (the image of) $T$ may be thought of as a ribbon. We impose the condition that the framings may not include half-twists, so a framing is essentially just an assignment of an integer to each connected component of $T$.
If $M = \Sigma \times I$, then we will assume that the points in $N$ are contained in $\Sigma \times \{ \frac{1}{2} \}$ and that their framings are thought to be embedded orthogonally to $\Sigma \times \{ \ast \}$. 
\end{definition}

\begin{figure}[h]
\centering
$\vcenter{\hbox{\includegraphics[height=3cm]{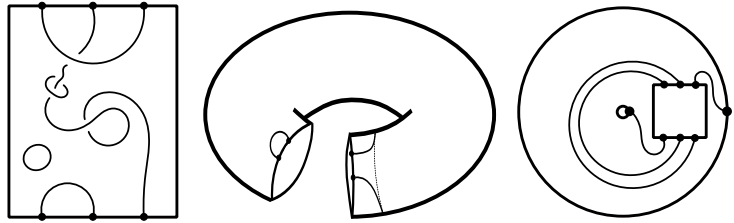}}}$
\caption{Some examples of framed tangle diagrams.}
\end{figure}

We say a framed tangle in $\Sigma \times I$ has \textbf{blackboard framing} if the entire framing is embedded orthogonally to $\Sigma$. Every framed link in $\Sigma \times I$ is isotopic to one with a blackboard framing by turning each twist into a loop with a local blackboard framing:
\[
    \vcenter{\hbox{\includegraphics[height=3cm]{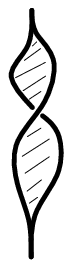}}} \quad = \quad \vcenter{\hbox{\includegraphics[height=3cm]{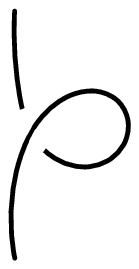}}}.
\]
This suggests that we may represent framed links in $\Sigma \times I$ as link diagrams on $\Sigma$. Indeed, equivalence under ambient isotopy is the same as equivalence under the three Reidemeister moves below.
\begin{equation}
\vcenter{\hbox{\includegraphics[height=3cm]{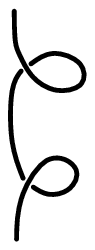}}} \quad = \quad \vcenter{\hbox{\includegraphics[height=3cm]{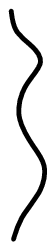}}}\quad = \quad \vcenter{\hbox{\includegraphics[height=3cm]{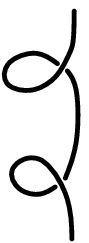}}}
\end{equation}
\begin{equation}
\vcenter{\hbox{\includegraphics[height=3cm]{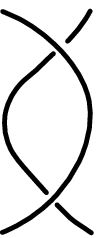}}} \quad = \quad \vcenter{\hbox{\includegraphics[height=3cm]{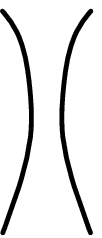}}} \quad = \quad \vcenter{\hbox{\includegraphics[height=3cm]{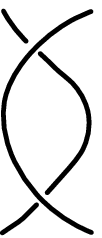}}}
\end{equation}
\begin{equation}
\vcenter{\hbox{\includegraphics[height=3cm]{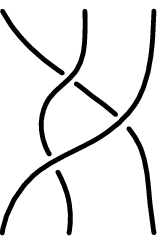}}} \quad = \quad \vcenter{\hbox{\includegraphics[height=3cm]{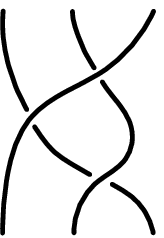}}}.
\end{equation}

Define the \textbf{writhe} of a tangle diagram is the number of positive crossings minus the number of negative crossings (a positive crossing is when the bottom-left to top-right strand is on top). It is easy to see that the Reidemeister moves above preserve the writhe of a diagram, so the writhe is a well defined integer on isotopy classes of diagrams. The writhe should be thought of as a grading on the free $R$-module on the set of tangles in the given space, which provides a good reason to work with framed links over ordinary links. Such a module is a main ingredient of this theory, so let's honor it with a proper discussion.

\begin{definition}
Let $\ct(M, N)$ be the free $R$-module generated by the set of framed $N$-tangles in $M$. Analagously, we can define $\ct^{or}(M,N)$ to be the free $R$-module generated by the set of oriented framed $N$-tangles in $M$. All definitions which are to follow in this section have an analagous definition using oriented tangles. Also, we will formally define $\ct_R(\varnothing, \varnothing) := R$. 
\end{definition}

The construction $\ct( -, -)$ is actually a symmetric monoidal functor $\ct: \sfc \to R\textsf{-Mod}$ for a careful choice of category $\sfc$ which we now describe. The objects of $\sfc$ are pairs $(M, N)$ of the same type as discussed previously. A morphism $(f, W): (M', N') \to (M, N)$ is a pair of a smooth, orientation-preserving embedding $f: M' \to M$ such that $M - f(M')$ is either a smooth 3-manifold or the empty set, and choice of $W \in \ct\big( M-f(M'), N \sqcup f(N') \big)$ (unless $M - f(M')$ is empty, in which case $W$ is a formal symbol for the ``empty link" in the empty set). Composition is given by gluing wirings. In other words, $(g, W') \circ (f, W) = (g \circ f, W' \cup W)$, which is associative since $\circ$ and $\cup$ are associative.

\begin{figure}[h]
\centering
$\vcenter{\hbox{\includegraphics[height=3cm]{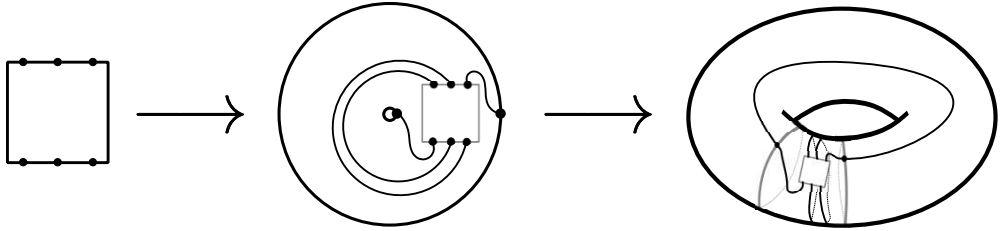}}}$
\caption{An example of how composition works in $\sfc$.}
\end{figure}

The induced map denoted $W: \ct(M', N') \to \ct(M, N)$ is a linear map defined by $W(T) = W \cup T$, and we will refer to such a linear map $W$ as a \textbf{wiring}. We are abusing notation by denoting this linear map by $W$, but it should be clear from the context what $f$ is since it is technically encoded in the data of the element $W \in \ct\big( M-f(M'), N \sqcup f(N') \big)$. It is true that $\ct$ preserves composition and identity morphisms which makes $\ct$ a functor. $\sfc$ can now be equipped with a symmetric monoidal structure via disjoint union. It is clear that 
\[\ct_R(M \sqcup M', N \sqcup N') \cong \ct_R(M, N) \underset{R}{\otimes} \ct_R(M', N')\]
for any sets of framed points $N \subset \partial M$ and $N' \subset \partial M'$. The unit is given by the object $(\varnothing, \varnothing) \in \sfc$ and define $\ct(\varnothing, \varnothing) := R$, which makes $\ct$ a symmetric monoidal functor. 

\begin{definition}
Let $B$ be the smooth closed $3$-ball, $N_i$ be a choice of $2i$ boundary points of $B$, and let $X \subset \underset{i \in \N}{\bigsqcup} \ct(B, N_i)$ be some (typically finite) set, which we will call a set of \textbf{skein relations}. Given any tangle module $\ct(M, N_M)$, there exists a submodule $\ci(X)$ generated by the set 
\[\{ W(x) \mid x \in X \text{ and } W:\ct(B, N_B) \to \ct(M, N_M) \text{ is a wiring diagram} \}.\] 
A quotient of the form $\cs_X(M, N) := \ct(M, N) / \ci(X)$ is called a \textbf{skein module} of $M$ relative to $N$. If $N = \varnothing$ is the empty set, we may use the notation $\cs_X(M) := \cs_X(M, \varnothing)$. Similar definitions may be given using oriented and/or unframed tangles instead. 
\end{definition}

A simple example of a skein relation is 
\begin{align}
&\pic[2.5]{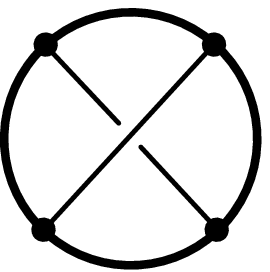} - \pic[2.5]{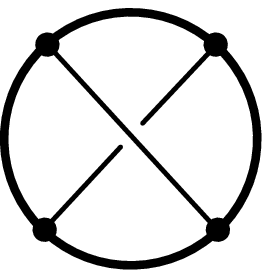}
\end{align}
which eliminates the over-under distinction for crossings. Imposing this skein relation on isotopy classes of tangles results in homotopy classes of tangles. We will explore more examples in Section \ref{sec:skeintheories}.

In general, the construction $\cs_X(-, -)$ is a functor in the same way that $\ct(-, -)$ is; a smooth, orientation-preserving embedding $f: M \to M'$ and an element $W \in \cs_X\big( M-f(M'), N \sqcup f(N') \big)$ (assume $N$ is disjoint from $\im(f)$) defines a linear map $W: \cs_X(M, N) \to \cs_X(M', N')$. In fact, the quotient maps $\alpha_{(M, N)}: \ct(M, N) \to \cs_X(M, N)$ yield a natural transformation. In other words, given a morphism $(M, N) \to (M', N')$ in $\sfc$, the diagram
\begin{center}
\begin{tikzcd}
	\ct(M, N) \arrow[d,"\alpha_{(M, N)}"] \arrow[r, "W"] 
	& \ct(M', N') \arrow[d, "\alpha_{(M', N')}"]\\
	\cs_X(M, N) \arrow[r, "W"] & \cs_X(M', N')
\end{tikzcd}
\end{center}
commutes. Such a functor will be called a $\textbf{skein theory}$ (or \textit{oriented} skein theory if the skein relations are based on oriented tangles). 

For any oriented surface $\Sigma$, we can define a category $\sfskein_X(\Sigma)$ which we will call a $\textbf{skein category}$. The objects of this category are finite sets of framed points (together with a choice of orientation if the skein theory is oriented) $N$ in $\Sigma$, and the morphisms $N \to N'$ are elements of $\cs_X\big(\Sigma \times I, (N \times \{0\}) \sqcup (N' \times \{1\})\big)$, so the category is $R$-linear. Write composition of morphisms by concatenation. If $y:N \to N'$ and $z:N' \to N''$ are morphisms, then their composite $yz:N \to N''$ is constucted by gluing one end of $y$ to the other end of $z$ through $N'$ and rescaling the interval coordinate appropriately.
\\
\begin{figure}[h]
\centering
$\pic[15]{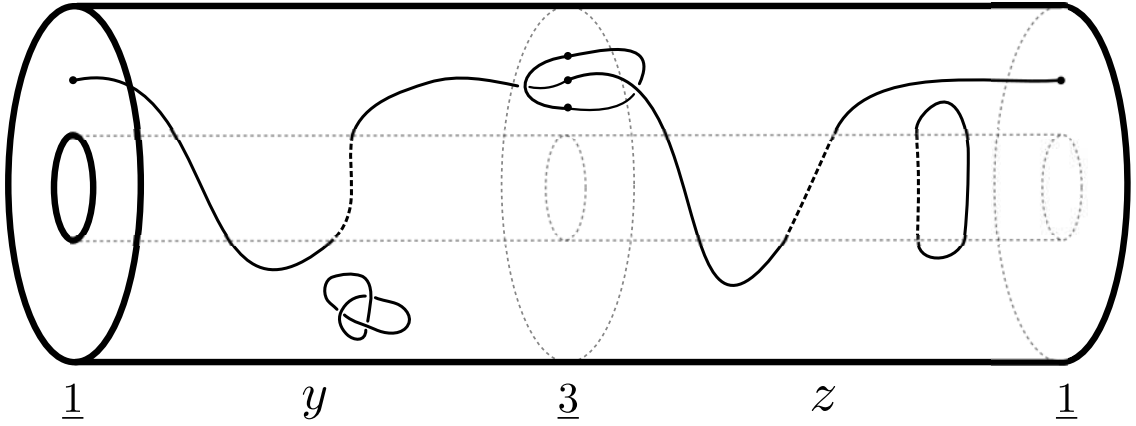}$
\caption{A composition of morphisms in the category $\sfskein_X(A)$ for an annulus $A$.}
\end{figure}

The endomorphism algebras in this category are called \textbf{skein algebras} and are denoted by $\cs_X(\Sigma, N) := \cs_X(\Sigma \times I, (N \times \{0\}) \sqcup (N \times \{1\}) \big)$. If $N$ is the empty set, then we reduce the notation to simply $\cs_X(\Sigma)$.

If $f: \Sigma \to \Sigma'$ is a smooth embedding of surfaces then there is an induced functor 
\[
\sfskein_X(f,W): \sfskein_X(\Sigma') \to \sfskein_X(\Sigma)
\]
defined on objects by $\sfskein_X(f)(N) = f(N)$ and on morphisms in the following way. First, extend $f$ trivially to $f \times \id_I: \Sigma \times I \to \Sigma' \times I$. Then, in the skein algebra of the complement of the image of $f \times \id_I$, choose the multiplicative identity element $e \in \cs_X\big( \Sigma' - \im(f) \big)$ which is the empty tangle. The pair $(f \times \id_I, e)$ is an object in the category $\sfc$, which gives rise to a wiring
\[e: \cs_X\Big(\Sigma \times I, \big(N \times \{0\}\big) \sqcup \big(N' \times \{1\}\big) \Big) \to \cs_X\Big(\Sigma' \times I, \big(f(N) \times \{0\}\big) \sqcup \big( f(N') \times \{1\} \big) \Big)\]
via the functor $\cs_X$. Now we may define what $\sfskein_X(f)$ does to morphisms: $\sfskein_X(f)(y) = e(y)$ for any $y \in \cs_X\Big(\Sigma \times I, (N \times \{0\}) \sqcup (N' \times \{1\}) \Big)$.

\begin{figure}[h]
\centering
$\pic[15]{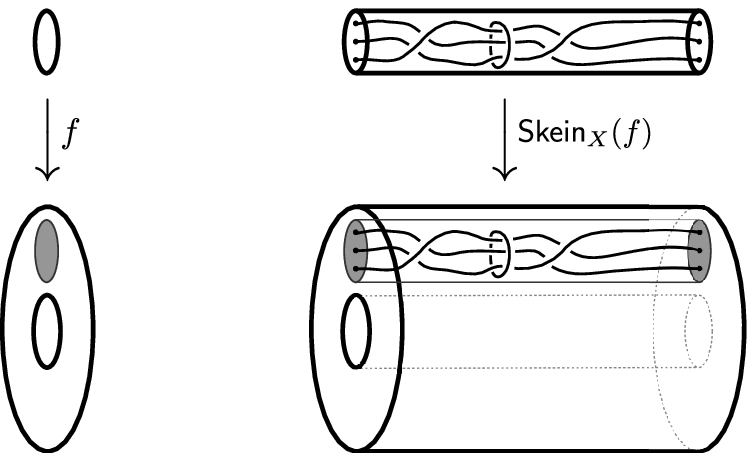}$
\caption{A morphism in $\sfskein_X(D^2)$ and its image under a functor to $\sfskein_X(A)$.}
\end{figure}

It is clear that $\sfskein_X(f)$ preserves composition and identity morphisms. Therefore, if we let $\mathsf{Surf}_{\textrm{emb}}$ be the category of smooth embeddings between smooth oriented surfaces, we can summarize our last few points by saying we have a functor
\[
\sfskein_X: \mathsf{Surf}_{\textrm{emb}} \to \mathsf{Cat}.
\]
In particular, $\sfskein_X(f)$ defines algebra homomorphisms on the skein algebras
\[e: \cs_X\big(\Sigma, N\big) \to \cs_X\big(\Sigma', f(N)\big).\]

Actually, more general functors may be defined. Given such an $f$ and a set of points $N^\circ$ in the complement of the image of $f$, one may define $\sfskein_X(f, N^\circ)$ on objects by $\sfskein_X(f, N^\circ)(N) = N \sqcup N^\circ$ and on morphisms in the same way above, by replacing $e$ with the identity of $\cs_X\big( \Sigma' - \im(f), N^\circ \big)$. Then $\sfskein_X(f, \varnothing) = \sfskein_X(f)$.

\begin{remark} \label{rem:skeinaction}
If $N$ is a set of framed points on $\Sigma$, then a smooth embedding $f: \Sigma \to \partial M$ induces a $\cs_X(\Sigma, N)$-module structure on $\cs_X\big(M, N'\big)$ for any $N'$ with $f(N) \subseteq N'$. The action is given by ``pushing tangles in through the boundary". In other words, the pre-composition of a smooth embedding of a collar neighborhood $g: \partial M \times I \to M$ with $f \times \id_I: \Sigma\times I \to \partial M \times I$ induces a bilinear map
\[
\cs_X(\Sigma, N) \times \cs_X(M, N') \to \cs_X(M, N')
\]
because $M$ minus a collar neighborhood is diffeomorphic to itself. Alternatively, a choice of element in $\cs_X(\Sigma, N')$ produces a wiring $\cs_X(M, N') \to \cs_X(M, N')$.
\end{remark}

\section{Examples of Skein Theories} \label{sec:skeintheories}

The last section leaves us with an important and unanswered question. Which sets of skein relations $X$ produce interesting skein theories? One class of examples is found by importing sets of relations satisfied by morphisms in a linear ribbon category as skein relations. Ribbon categories are braided monoidal categories which are rigid and equipped with a twist morphism for every object, satisfying some compatibility conditions \cite{RT90}. The axioms are such that morphisms built from braidings between tensor products of objects may be interpreted as framed braid diagrams. In particular, the morphisms satisfy the Reidemeister moves shown previously. We will discuss three examples of skein theories derived from skein relations which are meant to emulate linear relations satisfied by the braid and twist morphisms in certain ribbon categories coming from the representation theory of quantum groups, a topic which has generated a lot of interest from mathematicians since the 1980s. The categories of representations of quantum groups are ribbon categories with non-involutive braidings and the skein relations below capture how far off the braidings are from being involutive.  

Here, we are forced to fix a base ring. For our purposes, $R$ must be a commutative ring containing invertible elements $s$ and $v$. Typical choices of $R$ are $\Z[s^{\pm 1}, v^{\pm 1}], \Q(s, v)$, or some other ring in between these. \\

\begin{example}[\textit{Kauffman (Dubrovnik) Skein Relations}]
Let $X_1$ be the set of two unoriented skein relations
\begin{align}
&\pic[2.5]{poscross.eps} = \pic[2.5]{negcross.eps} + (s-s^{-1}) \left( \pic[2.5]{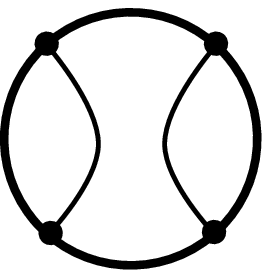} - \pic[2.5]{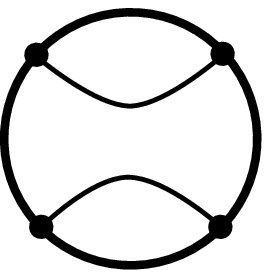} \right), \\ 
&\pic[2.5]{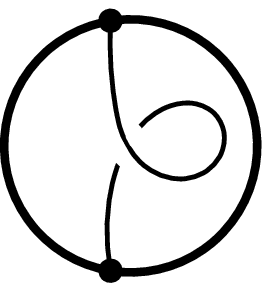} = v \pic[2.5]{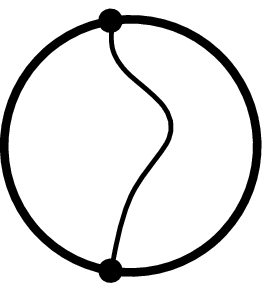}, \\
&\pic[2.5]{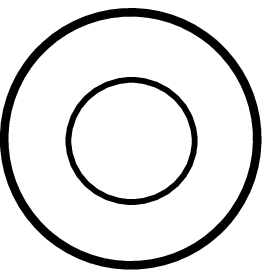} = \left( 1 - \frac{v-v^{-1}}{s-s^{-1}} \right) \pic[2.5]{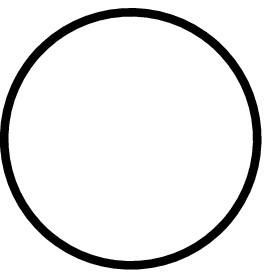}.
\end{align}
The value of the unknot will be denoted by $\delta_\cd := 1 - \frac{v-v^{-1}}{s-s^{-1}}$. The functor $\cd(-,-) := \cs_{X_1}(-,-)$ is the Dubrovnik skein theory (sometimes called the Kauffman skein theory). We will use the notation $\mathsf{D}(-) := \sfskein_{X_1}(-)$ for the Dubrovnik skein categories. Using the Dubrovnik variant is important for us (see Corollary \ref{cor:liealgebra1}). This theory is related to Dubrovnik polynomials in that the Dubrovnik polynomial of a link is a normalized value of the link in $\cd(S^3)$, depending on the framing.
\end{example}

\begin{example}[\textit{HOMFLYPT Skein Relations}]
Next, let $X_2$ be the set of two oriented skein relations
\begin{align}
&\pic[2.5]{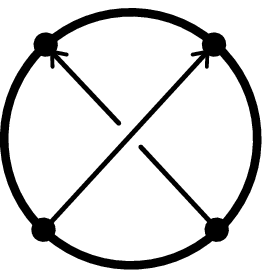} = \pic[2.5]{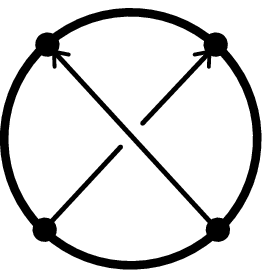} + (s-s^{-1}) \pic[2.5]{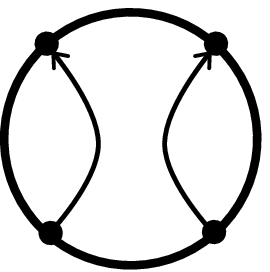}, \\
&\pic[2.5]{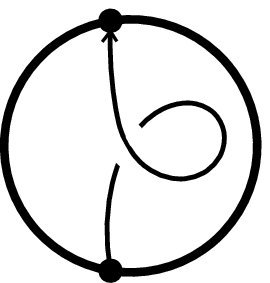} = v \pic[2.5]{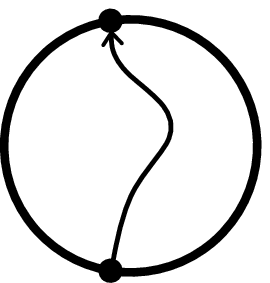}, \\
&\pic{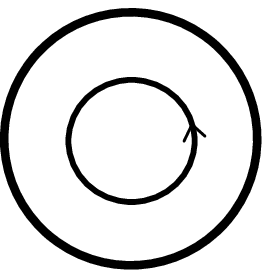} = \frac{v-v^{-1}}{s-s^{-1}} \pic[2.5]{empty.eps}.
\end{align}
The value of the unknot will be denoted by $\delta_\ch := \frac{v-v^{-1}}{s-s^{-1}}$. The functor $\ch(-,-) := \cs_{X_2}(-,-)$ is the HOMFLYPT skein theory and we use the notation $\mathsf{H}(-) := \sfskein_{X_2}(-)$ for the HOMFLYPT skein categories. As in the Dubrovnik case, this theory is related to HOMFLYPT polynomials so that the HOMFLYPT polynomial of a link is a normalized value of the link in $\ch(S^3)$.
\end{example}

\begin{example}[\textit{Kauffman Bracket Skein Relations}]
As a final example, let $X_3$ be the set of two unoriented skein relations
\begin{align}
&\pic[2.5]{poscross.eps} = s \pic[2.5]{idresolution.eps} + s^{-1} \pic[2.5]{capcupresolution.eps}, \\
&\pic[2.5]{vh.eps} = -s^{-3} \pic[2.5]{frameresolution.eps}, \\
&\pic[2.5]{unknot.eps} = -s^2 - s^{-2} \pic[2.5]{empty.eps}.
\end{align}
The value of the unknot will be denoted by $\delta_\ck := -s^2 - s^{-2}$. The functor $\ck(-,-) := \cs_{X_3}(-,-)$ is the Kauffman bracket skein theory (not be confused with the Kauffman skein theory) and we use $\mathsf{K}(-) := \sfskein_{X_3}(-)$ to notate the Kauffman bracket skein categories. The value of a link in $\ck(S^3)$ is equal to its bracket polynomial, which may be normalized to obtain its Jones polynomial. One could choose the framing parameter to be $v$ instead of $-s^{-3}$ like we did before without any obvious consequences, but it is standard in the literature to make the choice we present here.
\end{example}

\begin{remark} \label{rmk:naturaltransformation}
Any linear combination of tangles which satisfy the Dubrovnik skein relations will also satisfy the Kauffman bracket skein relations after making the specialization $v=-s^{-3}$. This can be seen by performing a calculation in the relative skein algebra of the ball with $4$ points. Let $\sigma^\pm$ be the diagrams of positive and negative crossings, $e$ the planar diagram with two vertical strands, and $c$ the planar diagram with two horizontal strands. Then the Dubrovnik skein relation is
\[
\sigma^+_\cd - \sigma^-_\cd - (s-s^{-1})e_\cd + (s-s^{-1})c_\cd = 0
\]
and the Kauffman bracket skein relation implies
\begin{align*}
\sigma^+_\ck &= se_\ck + s^{-1}c_\ck \\
\sigma^-_\ck &= sc_\ck + s^{-1}e_\ck
\end{align*}
where the subscript indicates which skein module the diagrams are in. Consider the assignment $d_\cd \mapsto d_\ch$ for any tangle $d$ in the ball relative to those $4$ points. Then observe:
\begin{eqnarray*}
& &\sigma^+_\ck - \sigma^-_\ck - (s-s^{-1})e_\ck + (s-s^{-1})c_\ck \\
=& &(se_\ck + s^{-1}c_\ck) - (sc_\ck + s^{-1}e_\ck) - (s-s^{-1})e_\ck + (s-s^{-1})c_\ck \\
=& &0
\end{eqnarray*}

Therefore, there is a natural transformation of skein theories 
\[
\eta: \cd(-,-) \to \ck(-,-)
\] 
whose components are essentially the identity map, using the same type of assignment as we did above. Note that any component of $\eta$ corresponding to a skein algebra is an algebra homomorphism since the map preserves the (topological) product structure of the manifold.
\end{remark}

In this work, we will be focused on generalizing existing results from the HOMFLYPT and Kauffman bracket skein theories to the Dubrovnik skein theory, but we will state a few facts regarding the HOMFLYPT and Kauffman bracket skein theories when it is valuable for us to do so. 

\begin{remark}
For a surface $\Sigma$, a link in $\Sigma \times I$ may be considered as a cycle in $H_1(\Sigma \times I, \Z) \cong H_1(\Sigma, \Z)$ and applying a HOMFLYPT skein relation to a link results in a linear combination of diagrams, each belonging to the same homology class. The product of two links in the skein algebra descend to the union of the two chains in the homology group. Therefore, the skein algebra $\ch(\Sigma)$ is graded by the first (singular) homology group $H_1(\Sigma, \Z)$. This isn't true for the skein algebras $\cd(\Sigma)$ and $\ck(\Sigma)$ because the ``cap-cup" resolution diagram in the Dubrovnik and Kauffman bracket skein relations, considered as a chain in $H_1(\Sigma, \Z)$, may not belong to the same homology class as the original link. Another problem is that the links in the skein algebras are unoriented, but chains in $H_1(\Sigma, \Z)$ are. Both of these problems are resolved by replacing $\Z$ with $\Z/2\Z$. So $\cd(\Sigma)$ and $\ck(\Sigma)$ are $H_1(\Sigma, \Z/2\Z)$-graded algebras. 
\end{remark}

\section{Skein Algebras of Tangles in a Cube} \label{sec:cube}

The Birman-Murakami-Wenzl (BMW) algebras, Hecke algebras (of type A), and Temperley-Lieb algebras are endomorphism algebras of certain representations of quantum groups. These algebras widely-studied due to their importance in representation theory, combinatorics, and geometry. For us, they are important because of their relation to the skein theories defined above. 

The endomorphism algebras in the categories $\mathsf{D}(I^2), \mathsf{H}(I^2)$, and $\mathsf{K}(I^2)$ are well-known and motivate the choices of their defining skein relations. For an integer $n \geq 1$, let $\underline{n}$ be a set of $n$ points in $I^2$, chosen to be evenly spaced along the line segment $\{ 1/2 \} \times I$ (choose all points to share the same orientation if in the context of an oriented skein theory). Then the endomorphism algebras 
\[
BMW_n := \End_\mathsf{D}(I^2)(\underline{n}), \qquad H_n := \End_\mathsf{H}(I^2)(\underline{n}), \qquad TL_n := \End_\mathsf{K}(I^2)(\underline{n})
\]
are known to be isomorphic to the Birman-Murakami-Wenzl, (Type A) Hecke, and Temperley-Lieb algebras respectively after certain specializations of the parameters (see \cite{Abr08}, \cite{AM98}, \cite{Mor10}).

\begin{figure}[H]
\centering
$\pic[13]{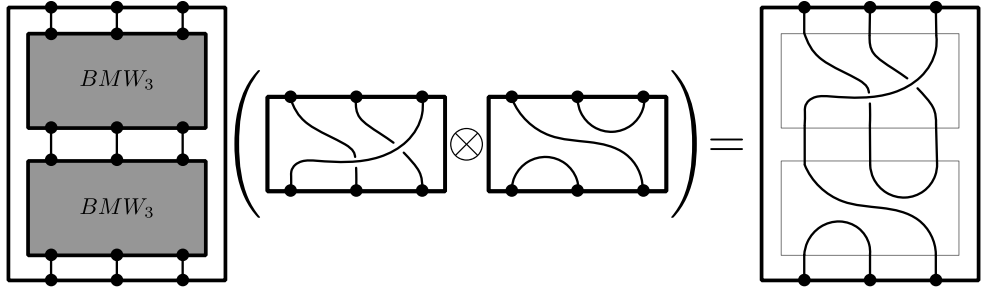}$
\caption{The multiplication structure in $BMW_n$ is induced by a wiring diagram.}
\end{figure}

Let $q \in \C$ be not a root of unity, $U_q(\mathfrak{gl}_N)$ be the Drinfeld-Jimbo quantum group associated to the Lie algebra $\mathfrak{gl}_N$, and $V$ be the natural representation of $U_q(\mathfrak{gl}_N)$ (see \cite{CP94}). Then $H_n$ acts on the $n$-fold tensor product $V^{\otimes n}$ by $U_q(\mathfrak{gl}_N)$-linear endomorphisms, and this action generates $\End_{U_q(\mathfrak{gl}_N)}(V^{\otimes n})$. 

This is actually part of the statement of quantum Frobenius-Schur-Weyl duality: $V^{\otimes n}$ is a $U_q(\mathfrak{gl}_N)$-$H_n$-bimodule, and the action of one algebra generates the linear endomorphisms with respect to the other. So a representation of one algebra determines a representation of the other via tensor product with $V^{\otimes n}$. The Birman-Murakami-Wenzl algebra plays the role of the Hecke algebra for $U_q(\mathfrak{g}_N)$ in the case where $\mathfrak{g}$ is one of the orthogonal or symplectic Lie algebras. For this reason, from a Lie theoretic point of view, the HOMFLYPT skein theory is thought of as a ``type A" theory, while the Dubrovnik skein theory is a ``types B, C, D" skein theory.

As for the Temperley-Lieb algebra $TL_n$, it acts on the $n$-fold tensor power $V^{\otimes n}$ of the natural representation of $U_q(\mathfrak{gl}_2)$ (one may replace $\mathfrak{gl}_2$ with $\mathfrak{sl}_2$). Actually, there is a surjective algebra homomorphism $\pi_{TL}: H_n \to TL_n$ and the action $H_n \to \End_{U_q(\mathfrak{gl}_2)}(V^{\otimes n})$ factors through this homomorphism (see \cite{Jim86}). One can conclude that the action of $H_n$ on $V^{\otimes n}$ is not faithful, at least in the case when $N=2$ and $n \geq 2$.

Much is known about these algebras and some of the results surrounding them are very useful in our context. One of the main ideas is that each of the algebras discussed above has a family of idempotents which provide algebraically nice closures to links in the skein algebra of the annulus. Let's first describe the Hecke algebra in more detail before focusing on the BMW algebra. 

Recall that a parition $\lambda$ of a non-negative integer $n$ is a tuple of weakly-decreasing positive integers $(\lambda_1, \dots, \lambda_k)$ that sum to $n$. For any partition $\lambda$ of $n$, there exists an element $y_\lambda \in H_n$ which is idempotent, so that $y_\lambda^2 = y_\lambda$, and minimal in the sense that it generates a minimal left-ideal of $H_n$ (see \cite{AM98}). The elements $y_n := y_{(n)}$ corresponding to single row partitions are called \textbf{(Hecke) symmetrizers}. These elements have a certain absorption property which makes them unique, which we will now describe. Let $\sigma_i \in H_n$ be the positive crossing between the $i^{\rm{th}}$ and $(i+1)^{\rm{th}}$ strands. The algebra $H_n$ is generated by the $\sigma_i$ and the symmetrizers are the unique idempotent elements satisfying $\sigma_i y_n = s y_n = y_n \sigma_i$. In this way, $y_n$ corresponds to a $1$-dimensional representation of $H_n$ (it is a deformation of the trivial representation of $\C S_n$ where $s=1$).

A similar story holds for the BMW algebra. Firstly, $BMW_n$ is generated by positive crossing elements $\sigma_i$ and cap-cup elements $c_i$. The $c_i$ generate a proper ideal $I_n$ in $BMW_n$. In \cite{BB01}, Beliakova and Blanchet show that the complement of $I_n$ in $BMW_n$ is isomorphic to $H_n$, giving an isomorphism $BMW_n \cong H_n \oplus I_n$. Then they construct an additive and multiplicative (but non-unital) homomorphism 
\begin{equation}
\Gamma_n : H_n \to BMW_n
\end{equation}
which is a section of the natural projection $BMW_n \to H_n$ such that 
\begin{equation} \label{eq:bbsectionproperty}
\Gamma_n(x)y = 0 = y \Gamma_n(x) \quad \textrm{for } x \in H_n, y \in I_n.
\end{equation}
Using this section, one can transport the minimal idempotents $y_\lambda \in H_n$ to minimal idempotents $\tilde{y}_\lambda := \Gamma_n(y_\lambda) \in BMW_n$. The elements $\tilde{y}_n := \tilde{y}_{(n)}$ are called the \textbf{(BMW) symmetrizers} and are the unique idempotent elements of $BMW_n$ satisfying the properties of the Hecke symmetrizers and property \eqref{eq:bbsectionproperty}. Let $[n]$ be the quantum integers
\[
[n] := \frac{s^n - s^{-n}}{s-s^{-1}}
\]
and define the constants
\[
\beta_n:=\frac{1-s^2}{s^{2n-1}v^{-1}-1}.
\]
By \cite{She16}, the BMW symmetrizers satisfy a very useful recurrence relation
\begin{equation} \label{eq:shellyrecurrence}
[n+1] \pic[2.7]{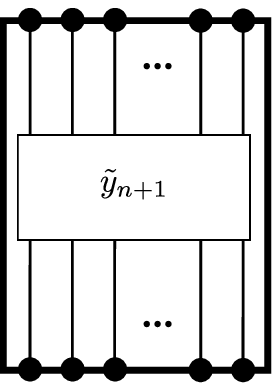} = [n]s^{-1} \pic[2.7]{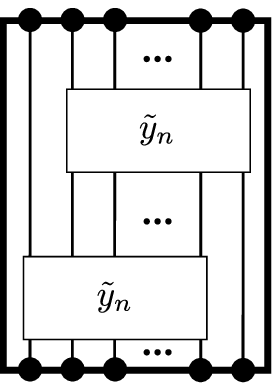} + \pic[2.7]{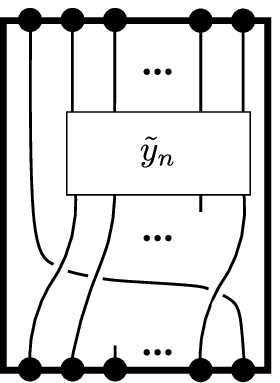} + [n]s^{-1} \beta_n \pic[2.7]{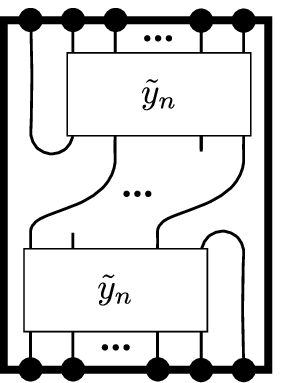}.
\end{equation}

For what it's worth, this relation descends to a well-known recurrence relation for the $y_n \in H_n$ via the projection map described above ($\tilde{y}_n$ gets sent to $y_n$ and the last diagram on the right-hand side gets sent to $0$).

\section{Skein Algebras of the Annulus} \label{sub:annulus}

Let's use the notation $A := S^1 \times [0,1]$ for the annulus. The thickened annulus represents perhaps the simplest space with non-trivial topology that can harbor links. There are often many different smooth embeddings of the thickened annulus into a given 3-manifold $M$ (one for every conjugacy class of $\pi_1(M)$, at the very least). Keep in mind that if we have a smooth embedding $f: A \times I \hookrightarrow M$, then we have an induced linear map between skein modules $\cs_X(f): \cs_X (A) \to \cs_X (M)$, allowing us to prod for information about $\cs_X(M)$. One useful type of embedding of $A$ is into a tubular neighborhood of a knot, which is often called \textit{decorating} or \textit{threading} a knot.

\begin{figure}[h]
\centering
$\pic{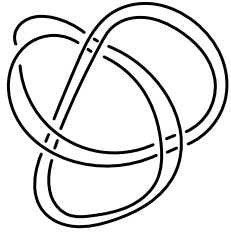}$
\caption{Threading a trefoil knot by a Turaev generator $z_1$.}
\end{figure}

A first observation is that any skein algebra of the form $\cs_X(A)$ is commutative because the link algebra $\cs_\varnothing(A)$ is a commutative algebra. To see this, consider the product of two links $L_1$ and $L_2$ in the link algebra and start by stretching $L_2$ towards the outer boundary past $L_1$, followed by moving it down below the furthest point of $L_1$, and finally contracting $L_2$ can to its original radial position. This is essentially the Eckmann-Hilton argument. 

We'll tackle describing the structure of $\ck(A)$ before the others since it is the easiest. The Kauffman bracket skein relation allows one to resolve all crossings in any diagram on any surface $\Sigma$.The set of non-trivial multi-curves (non-trivial meaning that no curve bounds a disk), together with the empty link in $A \times \{ \ast \}$ forms a basis of $\ck(A)$ (actually, the theorem is stated for any surface, see \cite{BP00} for example). There is only one non-trivial curve $z$ and the set $\{ z^k \}_{k \in \N}$ exhaust all of the multicurves. Therefore, $\ck(A)$ is a polynomial algebra $R[z]$. 

The HOMFLYPT and Dubrovnik cases are more complicated because the skein relations do not allow one to resolve all the crossings in a diagram. However, they do allow one to change a diagram with a negative crossing into a diagram with a positive crossing, plus diagrams with a lesser number of crossings. It follows that the skein algebras $\ch(A)$ and $\cd(A)$ are generated by knots with only positive crossings. Let $z_i$ be a knot in $A$ with winding number $i$ around the annulus. Turaev shows in \cite{Tur90} that the $z_i$ are algebraically independent by showing their HOMFLYPT and Dubrovnik polynomials are algebraically independent. Therefore, $\ch(A)$ is a polynomial algebra $R[z_i, i \in \Z_{\neq 0}]$. $\cd(A)$ has half as many generators since the $z_i$ in that setting are unoriented, but $\cd(A)$ is a polynomial algebra $R[z_i, i \in \Z_{> 0}]$ by the same arguments. 

\begin{figure}[H]
\centering
$\pic[5]{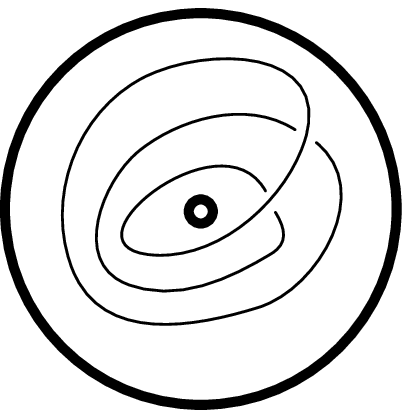}$
\caption{The generator $z_3 \in \cd(A)$.}
\end{figure}

The skein algebras of tangles in a cube relate to the skein algebras of the annulus via a wiring $\cl$. We will call the image of an element under this wiring the \textbf{annular closure} of the element. This map and it satisfies the ``trace property" $\cl(abc) = \cl(cab)$. Note that is it not unital nor multiplicative, but it is true that $\cl(a \otimes b) = \cl(a) \cl(b)$ where $a \otimes b$ is an element obtained by concatenating elements $a$ and $b$. Furthermore, it is well-known that any annular link is the closure of some braid.

\begin{figure}[h]
\centering
$\pic[5]{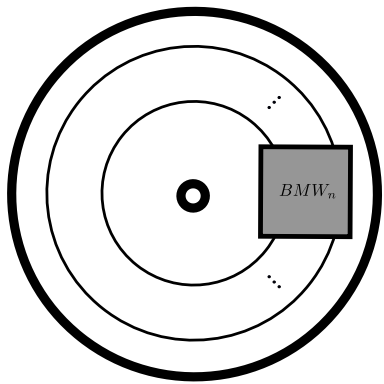}$
\caption{This annular closure wiring diagram.}
\end{figure}

Let $\widetilde{Q}_\lambda := \cl(\tilde{y}_\lambda) \in \cd(A)$. Zhong and Lu show in \cite{LZ02} that the set $\{ \widetilde{Q}_\lambda \}_\lambda$ forms a basis of $\cd(A)$ over the base ring $R=\Q(s, v)$, where $\lambda$ ranges over all partitions. Actually, it is an eigenbasis with respect to the meridian map $\phi:\cd(A) \to \cd(A)$ where the eigenvalue of $\widetilde{Q}_\lambda$ is 
\[
c_\lambda = \delta_\cd + ( s - s^{-1} ) \sum_{\square \in \lambda} v^{-1} s^{2 \textrm{cn}(\square)} - v s^{-2 \textrm{cn}(\square)}
\]
and where $\textrm{cn}(\square) := j - i$ is the \textit{content} of the box in the $\square$ in $i^\textrm{th}$ row and $j^\textrm{th}$ column of the Young diagram of the partition $\lambda$. Two distinct partitions $\lambda$ and $\lambda '$ give rise to distinct values of $c_\lambda$ and $c_\lambda'$. Therefore, each of the eigenspaces is $1$-dimensional. 

\begin{figure}[h]
\centering
$\pic[5]{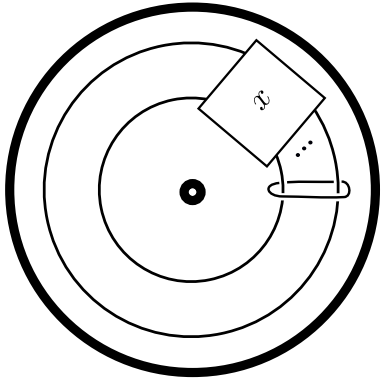}$
\caption{The meridian map $\phi$ applied to the annular closure of some $x$.}
\end{figure}

The behavior of $\widetilde{Q}_\lambda$ is similar to that of the $Q_\lambda := \cl^+(y_\lambda) \in \ch(A)$ where $\cl^+(-)$ is defined by orienting the strands in wiring digram of $\cl(-)$ counter-clockwise. Using the description $\ch(A) = R[z_i, i \in \Z_{\neq 0}]$, consider the subalgebras $\ch(A)^+ := R[z_i, i \in \Z_{> 0}]$ and $\ch(A)^- := R[z_i, i \in \Z_{< 0}]$. Then the two subalgebras are isomorphic via the linear involution defined by $z_i \mapsto z_{-i}$ and $\ch(A) \cong \ch(A)^+ \otimes \ch(A)^-$. The $\{ Q_\lambda \}_\lambda$ forms an eigenbasis of $\ch(A)^+$ with resepect to the meridian map and the eigenspaces are $1$-dimensional. As a side remark, the linear involution $z_i \mapsto z_{-i}$ may be realized topologically by rotating the thickened annulus $\pi$ radians around an appropriately centered axis parallel to $A$ and taking the induced skein map. We will call this map the \textit{flip map}.

In \cite{Luk05}, it's shown how to interpret $\ch(A)^+$ as the ring of symmetric functions $\Lambda$ (also see \cite{Mor02b}). There is an injective algebra homomorphism $\Lambda \to \ch(A)^+$ which sends the Schur function $s_\lambda$ to the minimal idempotent closure $Q_\lambda$. This homomorphism may be upgraded to an isomorphism if we extend the scalars of $\Lambda$ to $\Lambda_R = R \otimes \Lambda$. This theorem has plenty of implications. For example, the structure constants of $\Lambda$ in the basis $\{ s_\lambda \}_\lambda$ are the Littlewood-Richardson coefficients, which are then sent to structure constants of $\ch(A)^+$ in the basis $\{ Q_\lambda \}_\lambda$. In Section \ref{sec:Lukac}, we give a summary about the ring $\Lambda$ and discuss partial results surrounding the Dubrovnik analogue of this homomorphism.

Another application of the interpretation of $\ch(A)^+$ as $\Lambda$ is the definition of new special links in the skein algebra. There is a family of elements in $\Lambda$ known as the power sum symmetric functions, whose counterparts in $\ch(A)^+$ we will denote by $P_k$ for integers $k \geq 1$. These elements are algebraically independent in $\Lambda_\Q := \Q \otimes \Lambda$. Therefore, ordered monomials in the $P_k$ form a basis of $\ch(A)^+$.

\section{Skein Algebras of the Torus}

The power sum elements $P_k$ described above are known to behave wonderfully in skein theoretic computations. This allows for a simple description of the HOMFLYPT skein algebra of the torus $\ch(T^2)$ in terms of generators and relations. First, let's define the generators. Given an $r$ which is either a rational number, or $\pm \infty$, there is an oriented smooth embedding 
\[
\iota_{r}: A \hookrightarrow T^2
\]
of the annulus into a tubular neighborhood of the line of slope $r$ in the flat torus. Consider the embeddings $\iota_{r}$ and $\iota_{-r}$ to be the same embedding with opposite orientations. Now $\iota_r$ induces an algebra homomorphism
\[
\ch(\iota_r): \ch(A) \to \ch(T^2)
\]
on the level of skein algebras. The $\iota_r$ are distinct isotopically (even homotopically) and are exhaustive in the sense that any knot in the thickened torus is may be represented as being contained in the image of an $\iota_q$. Therefore, any generating set of $\ch(A)^+$ defines a generating set of $\ch(T^2)$. Let's consider this with respect to the the power-sum monomial basis of $\ch(A)^+$. More precisely, for any pair of integers $\xx = (a, b)$ where $k = \gcd(\xx)$, define 
\[
P_\xx := \ch(\iota_{a/b})(P_k).
\]

In \cite{MS17}, Morton and Samuelson prove that the skein algebra $\ch(T^2)$ admits a presentation with generators the elements of $\{ P_\xx \mid \xx \in \Z^2 \}$, subject to the single family of relations 
\[
[P_\xx, P_\yy] = \big( s^{\det(\xx,\yy)} - s^{-\det(\xx,\yy)} \big) P_{\xx + \yy}.
\]
This presentation exhibits a relationship to the elliptic Hall algebra $\mathcal{E}_{q,t}$: a 2-parameter family of algebras obtained via Hall algebra decategorification of categories of coherent sheaves over certain elliptic curves (see \cite{BS12}). Along the diagonal of the parameter space, the Burban-Schiffman presentation of the elliptic Hall algebra matches the Morton-Samuelson presentation of the skein algebra of the torus, which makes $\mathcal{E}_{s,s} \cong \ch(T^2)$. At the very least, this highlights two things. Firstly, the skein algebra of the torus admits a known deformation. Secondly, skein algebras are connected to surprising areas of mathematics and hence deserve our attention. 

Next we describe the Kauffman bracket skein algebra of the torus, although the order of exposition is opposite to the historical order of discovery. The story is actually quite similar to the HOMFLYPT case. Recall that $\ck(A)$ is a polynomial algebra $R[z]$ where $z$ is the simple closed curve around the hole of the annulus with winding number $1$. There exist polynomials known as \textit{Chebyshev polynomials} $T_k \in R[z]$ for all integers $k \geq 1$ which form a basis of $\ck(A)$. For any pair of integers $\xx = (a, b) \in \Z^2$, define elements $T_\xx \in \ck(T^2)$ as 
\[
T_\xx := \ck(\iota_{a/b})(T_k).
\]
Note that $T_\xx = T_{-\xx}$ because the $T_k$ are fixed under the flip map, which is possible since the Kauffman bracket skein theory is an unoriented skein theory. With this in mind, we may pass to the smaller indexing set $Z^2 / \langle \xx = -\xx \rangle$. Using these elements, Frohman and Gelca prove in \cite{FG00} that the skein algebra $\ck(T^2)$ admits a presentation with generators $T_\xx$ subject to the relations
\[
T_\xx T_\yy = s^{\det(\xx, \yy)} T_{\xx + \yy} + s^{-\det(\xx, \yy)} T_{\xx - \yy}
\]
which the authors call the ``product-to-sum" formulas. There is an algebra called the \textit{noncommutative torus}, which is a one-parameter deformation of the algebra of continuous functions on the torus. The Frohman-Gelca presentation of $\ck(T^2)$ matches a presentation of the invariant subalgebra of the noncommutative torus with respect to a certain involutive action.

Given the similarity between the presentations of $\ch(T^2)$ and $\ck(T^2)$, one might suspect that there exists a similar description of the Dubrovnik skein algebra $\cd(T^2)$. The answer to this question is affirmative, which is what we discuss in Chapter \ref{chap:torus}.

\section{A Relative Skein Algebra of the Annulus} \label{sec:relativeannulus}

At this point in the story we have introduced special links in the skein algebras of the annulus and described how we can transport these into other skein modules via threading knots. In order to say anything meaningful about these threadings, we ought to know how these special links will interact with other links once in the resulting space. For example, is there an easy way to describe how can we pass a special link through a single strand of another link? To answer this question universally, we should answer it in a certain relative skein algebra which is the endomorphism algebra of one point in the skein category of the annulus. Appropriate wirings from this skein algebra into other skein modules will give us our desired description. Here we will describe this algebra in more detail and summarize what is already known. 

First we will restrict ourselves to the HOMFLYPT case, for which we summarize the results from \cite{Mor02b}. The object we would like to discuss is $\ca_\ch := \ch(A, \underline{1})$, which is depicted diagramatically as an annulus with one point on each boundary component. This algebra is closely related to the affine Hecke algebra of type A, $\dot{H}_1$ (see \cite{MS17}). Two of the most basic elements of $\ca_\ch$ are the elements $e$, the identity of $\ca_\ch$, and $a$, which is a single strand once around the hole of the annulus. The product in the algebra in this digrammatic notation is given by nesting annuli. 

\begin{figure}[h]
\centering
$e = \pic[5]{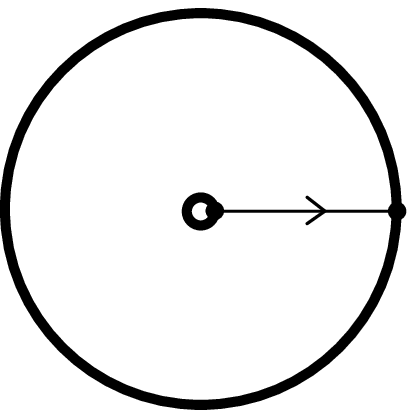} \qquad \qquad \qquad \qquad a = \pic[5]{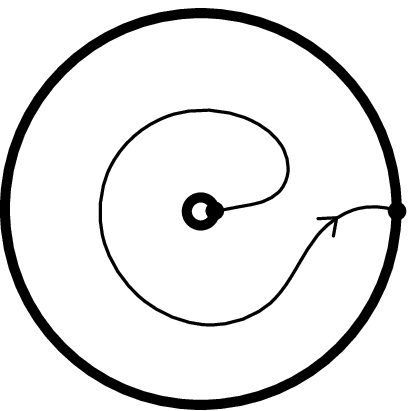}$
\caption{Left: the identity element $e \in \ca_\ch$. Right: $a \in \ca_\ch$.}
\end{figure}

Let $\cc_\ch := \ch(A)$. Then, $\ca_\ch$ admits both a left $\cc_\ch$-action by pushing links in front of tangles in $\ca_\ch$.
It is known that there is an equality of algebras between $\ca_\ch$ and the Laurent polynomial algebra $\cc_\ch[a, a^{-1}]$ with coefficients in $\cc_\ch$. In particular, $\ca_\ch$ is commutative and the left action is determined by how it acts on $e$. Analagously, there is a right $\cc_\ch$-action by pulling links in behind tangles in $\ca_\ch$. 

\begin{figure}[H]
\centering
$\cl(x) \cdot e = \pic[5]{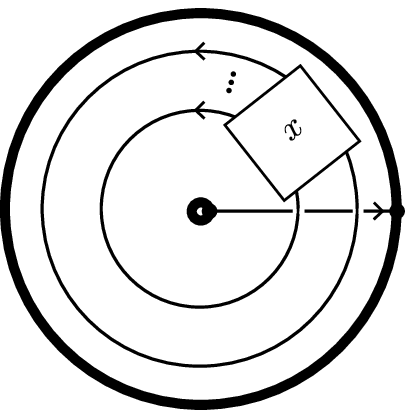} \qquad \qquad e \cdot \cl(x) = \pic[5]{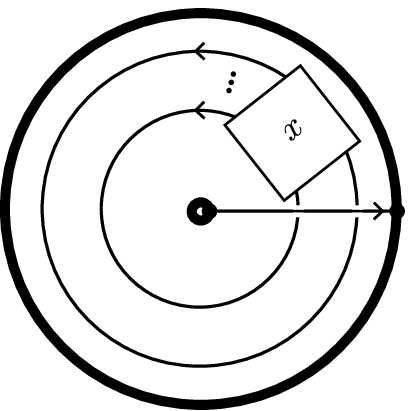}$
\caption{The annular closure of some $x$ acting separately on both the left and right of $e$.}
\end{figure}

Both actions are examples of those arising from embeddings into the boundary of the thickened annulus, as described in Remark \ref{rem:skeinaction}. The left and right actions obviously commute, endowing $\ca_\ch$ with a $\cc_\ch$-$\cc_\ch$-bimodule structure. It is probably worth pointing out that the left action is not equal to the right action. A simple measurement of how far off the two actions are from being equal would answer our question posed above. The answer is given in \cite{Mor02b} as a commutator relation 
\begin{equation} \label{eq:pkcommutator}
e \cdot P_k - P_k \cdot e = (s^k - s^{-k}) a^k.
\end{equation}
The proof involves the definition of a certain wiring diagram, defining a linear map $H_n \to \ca_\ch$. The idempotents $y_{n}$ satisfy a certain recurrence relation, and the commutator relation above follows from writing $P_k$ in terms of the $y_n$ and calculations involving the image of this recurrence relation.

Eventually we would like to prove a Dubrovnik analogue of Equation \eqref{eq:pkcommutator} (see Section \ref{sec:powersumelements}), but we still haven't defined any Dubrovnik analogue of the $P_k$. Nevertheless, we can discuss some of what was previously known about the algebra $\ca_\cd := \cd(A, [1])$ (see \cite{She16} for more details).

Let $\cc_\cd := \cd(A)$ and let $a, e \in \ca_\cd$ be the unoriented versions of the elements of the same name above. As in the HOMFLYPT case, there is a left-$\cc_\cd$ action on $\ca_\cd$ and the algebra $\ca_\cd$ is equal to the Laurent polynomial algebra $\cc_\cd[a^{\pm 1}]$. There is also a right-$\cc_\cd$ action on $\ca_\cd$, endowing $\ca_\cd$ with a bimodule structure. 

\begin{figure}[h] \label{fig:WandWstar}
\centering
$W = \pic[5]{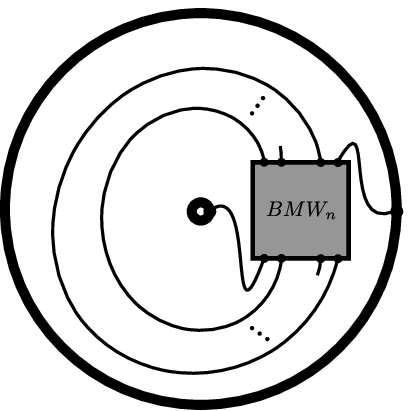} \qquad \qquad \qquad W^* = \pic[5]{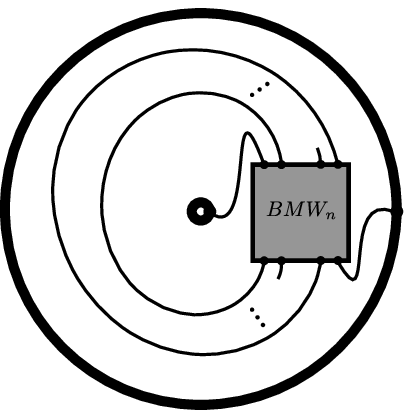}$
\caption{Wiring diagrams $W$ and $W^*$ which both induce linear maps.}
\end{figure}

Consider the wiring diagrams pictured above in Figure \ref{fig:WandWstar} which define linear maps $BMW_n \to \ca_\cd$. Let $W_n := W(\tilde{y}_{n+1})$ and $W^*_n :=  W^*(\tilde{y}_{n+1})$ where the subscript denotes the ``winding number'' around $A$. Taking the image of Equation \eqref{eq:shellyrecurrence} under $W$ gives a relation
\begin{equation} \label{eq:recursionina1}
[n+1] W_n = e \cdot \tilde{h}_n + [n] s^{-1} a W_{n-1} + [n] s^{-1} \beta_n a^{-1} W^*_{n-1}
\end{equation}
where $\tilde{h}_n := \widetilde{Q}_{(n)}$ is the annular closure of the symmetrizer $\tilde{y}_n$.

There are maps
\[
	(-)^*: BMW_n \to BMW_n \qquad \overline{(-)}: BMW_n \to BMW_n
\]
induced by the diffeomorphisms of the thickened square $(x, y, t) \mapsto (x, 1-y, 1-t)$ and $(x, y, t) \mapsto (x, y, 1-t)$, respectively. The map $\overline{(-)}$ is often called the \emph{mirror map} and is an $R$-anti-linear involution, while $(-)^*$ will be called the \emph{flip map} and is an $R$-linear involution. The symmetrizers $\tilde{y}_n$ are fixed under these maps because the mirror and flip maps preserve the properties that determine $\tilde{y}_n$ uniquely.
Using the quotient map defined by the equivalence relation $(x, 0, t) \sim (x, 1, t)$, we may analagously define maps
\[
	(-)^*: \ca_\cd \to \ca_\cd, \qquad \overline{(-)}: \ca_\cd \to \ca_\cd
\]
which are linear and anti-linear involutions, respectively. We will also call these the flip map and the mirror map; it will be clear from the context which is being applied. These maps satisfy the relations
\begin{align*}
	\left( W \left( x \right) \right)^* &= W^* \left( x^* \right), & (y \cdot e)^* &= e \cdot y^*, \\
	\overline{W \left( x \right)} &= W \left( \overline{x} \right), & \overline{(y \cdot e)} &= e \cdot \overline{y}
\end{align*}
for any $x \in BMW_n$ or $y \in \cc_\cd$. Apply the flip, the mirror, and the composite of the two separately to Equation \eqref{eq:recursionina1} to obtain alternate versions of the original recurrence relation:
\begin{align} 
[n+1] W^*_n = \tilde{h}_n \cdot e + [n] s^{-1} a W^*_{n-1} + [n] s^{-1} \beta_n a W_{n-1} \label{eq:recursionina2}, \\
[n+1] W_n = \tilde{h}_n \cdot e + [n] s a W_{n-1} + [n] s \bar{\beta}_n a^{-1} W^*_{n-1} \label{eq:recursionina3}, \\
\quad [n+1] W^*_n = e \cdot \tilde{h}_n + [n] s a^{-1} W^*_{n-1} + [n] s^{-1} \bar{\beta}_n a W_{n-1}. \label{eq:recursionina4}
\end{align}

Rearranging the difference of Equations \eqref{eq:recursionina1} and \eqref{eq:recursionina2} gives a relation
\begin{equation} \label{eq:annfund}
\tilde{h}_n \cdot e - e \cdot \tilde{h}_n = (s^n - s^{-n}) (a^{-1} W^*_{n-1} - a W_{n-1})
\end{equation}
which Shelly calls a \textit{fundamental skein relation} in $\ca_\cd$ since it reduces to to usual Dubrovnik skein relation when $n=1$. In Section \ref{sec:morecommutationrelations}, we will provide a similar relation which amounts to rewriting the right-hand side of the equation in terms of elements of the form $h_i \cdot a^k$ (alternatively $a^k \cdot h_i$).

\chapter{The Skein Algebra of the Torus} \label{chap:torus}

This chapter is dedicated to the study of the Dubrovnik skein algebra of the torus $\cd(T^2)$, the main result being a simple presentation for the algebra, stated in Theorem \ref{thm:toruspresentation}. The generators $\widetilde{P}_\xx$ are embeddings of certain special elements $\widetilde{P}_k \in \cd(A)$, which we define in Section \ref{sec:powersumelements}. The $\widetilde{P}_k$ generalize both the power sum elements $P_k \in \ch(A)$ and the Chebyshev polynomials $T_k \in \ck(A)$ which are both used widely in the context of their respective skein theories due to the simple relations they satisfy. In particular, we give a simple formula in Theorem \ref{thm:powersumcommutator} which explains how to move any embedding of $\widetilde{P}_k$ past a single strand in any skein module, which gives us a family of relations between the $\widetilde{P}_\xx$. In Section \ref{sec:presentation}, we show how this and another family of relations imply the relations given in the stated presentation. 

The definition of the $\widetilde{P}_k$ make it clear why these elements generalize their HOMFLYPT counterparts, but the same isn't immediately clear for the Kauffman bracket case for the $T_k$. In Section \ref{sec:compatibility} we justify this claim by showing that the image of $\widetilde{P}_k$ is equal to $T_k$ under the $A$-component of the natural transformation of skein theories $\eta:\cd \to \ck$ described in the previous chapter (see Remark \ref{rmk:naturaltransformation}). A corollary of this is that the presentation we give for $\cd(T^2)$ is compatible with $\ck(T^2)$ under $\eta$. Also described in this section is a ``compatibility" of the presentations of $\cd(T^2)$ and $\ch(T^2)$, given as an injective algebra homomorphism between the two. This homomorphism differs between the previous described $\cd(T^2) \to \ck(T^2)$ in the sense that it doesn't arise from any natural transformation between skein theories and is simply defined via generators and relations. However, this map is interesting because we show that $\cd(T^2)$ and $\ch(T^2)$ are universal enveloping algebras of some Lie algebras, and this homomorphism arises in the image of the universal enveloping algebra functor. 

In Section \ref{sec:action}, we give a description of the natural action of $\cd(T^2)$  on $\cd(D^2 \times S^1)$, which exists because the boundary of the solid torus is $T^2$. First we give a small generating set of the algebra $\cd(T^2)$, consisting of only five elements, together with an algorithm for expressing any $\widetilde{P}_\xx$ in terms of these other generators. Then, we show how these five elements act of $\cd(D^2 \times S^1)$, which implicitly describes how each $\widetilde{P}_\xx$ acts. 

This chapter includes many, but not all of the details from the paper \cite{MPS19} which is co-authored by Hugh Morton, Peter Samuelson, and myself. The proofs included in this chapter are the proofs from \cite{MPS19} that either I wrote or contributed significantly to.

\section{Power Sum Elements} \label{sec:powersumelements}

Recall that there is a injective algebra homomorphism $\Lambda \to \ch(A)^+$ which sends the Schur function $s_\lambda$ to the minimal idempotent closure $Q_\lambda$. Use $h_n := Q_{(n)}$ to denote the image of the $n^\textrm{th}$ complete homogeneous symmetric function under this homomorphism. In \cite{MS17} the authors import power sum elements from $\Lambda$ to $P_k \in \ch(A)$. The power sum elements have a concrete definition in $\Lambda$, but alternatively they may be defined using an equation of formal power series in the ring $\ch(A)[[t]]$ as
\begin{equation}
\sum_{k=1}^\infty \frac{P_k}{k} t^k = \ln \Bigg( 1 + \sum_{n=1}^\infty h_n t^n \Bigg)
\end{equation}
which writes each $P_k$ in terms of the generators $h_n$. 

Using the Beliakova-Blanchet section $\Gamma: H_n \to BMW_n$, we may emulate this definition to define ``power sum" elements $\widetilde{P}_k \in \cd(A)$ by the formal power series equation
\begin{equation}
\sum_{k=1}^\infty \frac{\widetilde{P}_k}{k} t^k = \ln \Bigg( 1 + \sum_{n=1}^\infty \tilde{h}_n t^n \Bigg)
\end{equation}
where $\tilde{h}_n := \widetilde{Q}_{(n)}$ is the annular closure of the BMW symmetrizers $\tilde{y}_n = \Gamma(y_n)$.

Let's now continue our discussion of Section \ref{sec:relativeannulus} with the following theorem.

\begin{theorem} \label{thm:powersumcommutator}
For any $k \geq 1$, the relation
\begin{equation} \label{eq:powersumcommutator}
e \cdot \widetilde{P}_k - \widetilde{P}_k \cdot e = (s^k - s^{-k}) (a^k - a^{-k})
\end{equation}
holds. Equivalently,
\begin{equation}
a^i \cdot \widetilde{P}_k - \widetilde{P}_k \cdot a^i = (s^k - s^{-k}) (a^{k+i} - a^{-k+i})
\end{equation}
for any integer $i$.
\end{theorem}

We will split the proof of this theorem into two technical lemmas.

\begin{lemma} \label{lem:powersumcommutator1}
The relations of Theorem \ref{thm:powersumcommutator} hold if and only if 
\begin{equation} \label{eq:skewcommutator1}
e \cdot (\tilde{h}_{n+2} + \tilde{h}_n) - (\tilde{h}_{n+2} + \tilde{h}_n) \cdot e = (sa + s^{-1}a^{-1}) (e \cdot \tilde{h}_{n+1}) - (s^{-1}a + sa^{-1}) (\tilde{h}_{n+1} \cdot e)
\end{equation}
for all integers $n \geq -1$, where $\tilde{h}_0 := 1$ and $\tilde{h}_{-1} := 0$.
\end{lemma}
\begin{proof}
The relations of Theorem \ref{thm:powersumcommutator} may be organized into a single power series equation
\begin{equation} \label{eq:skewcommutator1}
\sum_{k=1}^\infty \frac{e \cdot \widetilde{P}_k - \widetilde{P}_k \cdot e}{k} t^k = \sum_{k=1}^{\infty} \frac {(s^k - s^{-k}) (a^k - a^{-k})}{k} t^k
\end{equation}
in $\ca_\cd [[t]]$. Rewrite this equation as
\begin{equation} \label{eq:powersumcommutator2} 
e \cdot \Bigg( \sum_{k=1}^\infty \widetilde{P}_k \Bigg) - \Bigg( \sum_{k=1}^\infty \widetilde{P}_k \Bigg) \cdot e = \sum_{k=1}^{\infty} \frac {(sat)^k}{k} + \sum_{k=1}^{\infty} \frac {(s^{-1}a^{-1}t)^k}{k} - \sum_{k=1}^{\infty} \frac {(s^{-1}at)^k}{k} - \sum_{k=1}^{\infty} \frac {(sa^{-1}t)^k}{k}
\end{equation}
We can make sense of the left-hand side by extending the algebra homomorphism $x \mapsto e \cdot (x)$ to an algebra homomorphism of rings of formal power series
\begin{center}
\begin{tikzcd}
\cd(A) \arrow[r, "e \cdot (-)"] \arrow[d, hook] & \ca_\cd \arrow[d, hook] \\
\cd(A)[[t]] \arrow[r, "e \cdot (-)"] & \ca_\cd[[t]]
\end{tikzcd}
\end{center}
and similarly for $(-) \cdot e$. Now for shorthand, define 
\[
H(t) := 1 + \sum_{n=1}^\infty \tilde{h}_n t^n
\]
and recall the Taylor series expansion
\[
-\ln(1-x) = \sum_{k=1}^\infty \frac{x^k}{k}
\]
which is a variation of the Newton-Mercator series. Then the Equation \eqref{eq:skewcommutator1} becomes
\begin{equation} \label{eq:powersumcommutator3} 
e \cdot \Big( \ln\big(H(t)\big) \Big) - \Big( \ln \big(H(t)\big) \Big) \cdot e = - \ln(1 - sat) - \ln(1 - s^{-1}a^{-1}t) + \ln(1 - s^{-1}at) + \ln(1 - sa^{-1}t).
\end{equation}
The maps $e \cdot (-)$ and $(-) \cdot e$ commute with the natural logarithm. Use this and other natural log properties to write
\begin{equation}
\ln\Big( e \cdot \big( H(t) \big) \big( 1 - (sa + s^{-1}a^{-1})t + t^2 \big) \Big) = \ln\Big( \big( H(t) \big) \cdot e \big( 1 - (s^{-1}a + sa^{-1})t + t^2 \big) \Big).
\end{equation}
Exponentiating both sides and equating coefficients gives the system of equations defined in the statement of the lemma. Each step of the proof is invertible, and thus the two sets of relations are logically equivalent.
\end{proof}

\begin{lemma}
The relations of Lemma \ref{lem:powersumcommutator1} hold.
\end{lemma}
\begin{proof}
If $n=-1$, the relation we would like to show becomes 
\begin{equation*}
e \cdot \tilde{h}_1 - \tilde{h}_1 \cdot e = (s-s^{-1}) \left( a - a^{-1} \right)
\end{equation*}
which is just the Dubrovnik skein relation. 

For general values of $n$, the proof is a technical computation using repeated applications of the recursive formula for the $\tilde{h}_n$. Since we will need them here, let's recall the formulas given in Section \ref{sec:relativeannulus}. There are recursive formulas
\begin{align}
[n+1] W_n &= e \cdot \tilde{h}_n + [n] s^{-1} a W_{n-1} + [n] s^{-1} \beta_n a^{-1} W^*_{n-1}, \label{eq:r1} \\
[n+1] W^*_n &= \tilde{h}_n \cdot e + [n] s^{-1} a^{-1} W^*_{n-1} + [n] s^{-1} \beta_n a W_{n-1}, \label{eq:r2} \\
[n+1] W_n &= \tilde{h}_n \cdot e + [n] s a W_{n-1} + [n] s \bar{\beta}_n a^{-1} W^*_{n-1}, \label{eq:r3} \\
[n+1] W^*_n &= e \cdot \tilde{h}_n + [n] s a^{-1} W^*_{n-1} + [n] s^{-1} \bar{\beta}_n a W_{n-1} \label{eq:r4}
\end{align}
and the ``fundamental relation"
\begin{equation}
e \cdot \tilde{h}_n - \tilde{h}_n \cdot e = (s^n - s^{-n}) (a W_{n-1} - a^{-1} W^*_{n-1}). \label{eq:f}
\end{equation}

Let's use the shorthand notation 
\[
\{n\} := s^n - s^{-n}.
\]
Start by applying Equation \eqref{eq:f} to the relation of Lemma \ref{lem:powersumcommutator1} to obtain an equivalent relation
\begin{align}\label{eq_perprel4}
\begin{split}
\{n+2\} \left( aW_{n+1} - a^{-1}W^*_{n+1} \right) =& \left( sa+s^{-1}a^{-1} \right) \left( e \cdot \tilde{h}_{n+1} \right) - \left( sa^{-1}+s^{-1}a \right) \left( \tilde{h}_{n+1} \cdot e \right) \\
 & \qquad\qquad\qquad\qquad\qquad\qquad - \{n\}\left( aW_{n-1}-a^{-1}W^*_{n-1} \right).
\end{split}
\end{align}
We will show that the left-hand side of this equation may be reduced to the right-hand side by a series of applications of the recursive formulas, which we will signify with an asterisk $\ast$. 
\begin{eqnarray*}
&& \{n+2\} \left( aW_{n+1} - a^{-1}W^*_{n+1} \right) \\
\overset{\ast}{=}&& \{n+2\} \left( \frac{a}{[n+2]} \left( e \cdot \tilde{h}_{n+1} + [n+1]s^{-1}aW_n + [n+1]s^{-1}\beta_{n+1}a^{-1}W^*_n \right) \right.- \\
&&\left.\qquad\qquad\frac{a^{-1}}{[n+2]} \left( \tilde{h}_{n+1} \cdot e + [n+1]s^{-1}a^{-1}W^*_n + [n+1]s^{-1}\beta_{n+1}aW_n \right) \right)  \\
=&& \left( s - s^{-1} \right) \left( \left( a \cdot \tilde{h}_{n+1} + [n+1]s^{-1}a^2W_n + [n+1]s^{-1}\beta_{n+1}W^*_n \right) \right. \\
&&\qquad\qquad\,\,\,\,\,\left.-\left( \tilde{h}_{n+1} \cdot a^{-1} + [n+1]s^{-1}a^{-2}W^*_n + [n+1]s^{-1}\beta_{n+1}W_n \right)\right) \\
\overset{\ast}{=}&&\left( s-s^{-1} \right) \left( \left( a \cdot \tilde{h}_{n+1} + s^{-2}a \left( [n+2]W_{n+1} - \tilde{h}_{n+1} \cdot e - [n+1]s\bar{\beta}_{n+1}a^{-1}W^*_n \right) \right. \right. \\
&&\qquad\qquad\qquad\qquad\qquad\qquad\qquad\qquad\qquad\qquad\qquad\qquad
+ [n+1]s^{-1}\beta_{n+1}W^*_n \Big)  \\
&&\qquad\qquad\,\,\,\, \left. - \left( \tilde{h}_{n+1} \cdot a^{-1} + s^{-2}a^{-1}\left( [n+2]W^*_{n+1} - e \cdot \tilde{h}_{n+1} - [n+1]s\bar{\beta}_{n+1}aW_n \right) \right. \right. \\
&&\qquad\qquad\qquad\qquad\qquad\qquad\qquad\qquad\qquad\qquad\qquad\qquad\quad\quad
+ [n+1]s^{-1}\beta_{n+1}W_n \Big) \Big) \\
=&&\left( sa + s^{-1}a^{-1} \right) \left( e \cdot \tilde{h}_{n+1} \right) - \left( sa^{-1} + s^{-1}a \right) \left( \tilde{h}_{n+1} \cdot e \right) \\
&+& \left( s^{-1}a^{-1} + s^{-3}a \right) \left( \tilde{h}_{n+1} \cdot e \right) - \left( s^{-1}a + s^{-3}a^{-1} \right) \left( e \cdot \tilde{h}_{n+1} \right) \\
&+& \{n+2\}s^{-2}\left( aW_{n+1} - a^{-1} W^*_{n+1} \right) + \{n+1\}s^{-1}\left( \bar{\beta}_{n+1} - \beta_{n+1} \right) \left( W_n - W^*_n \right).
\end{eqnarray*}

We break the computation here to note that the first two terms in the last line also appear on the right hand side of \eqref{eq_perprel4}. Thus, we would like to prove the following equality:
\begin{eqnarray*}
&-&\{n\} \left( aW_{n-1} - a^{-1}W^*_{n-1} \right) \\
=&& \left( s^{-1}a^{-1} + s^{-3}a \right) \left( \tilde{h}_{n+1} \cdot e \right) - \left( s^{-1}a + s^{-3}a^{-1} \right) \left( e \cdot \tilde{h}_{n+1} \right) \\
&+&\{n+2\}s^{-2}\left( aW_{n+1} - a^{-1} W^*_{n+1} \right) - \{n+1\}s^{-1}\left( \bar{\beta}_{n+1} - \beta_{n+1} \right) \left( W_n - W^*_n \right).
\end{eqnarray*}
We will work the right-hand side of this above equation down to the left-hand side by continuing to apply the same identities. A large number of terms cancel and what remains is the desired relation.
\begin{eqnarray*}
&& \left( s^{-1}a^{-1} + s^{-3}a \right) \left( \tilde{h}_{n+1} \cdot e \right) - \left( s^{-1}a + s^{-3}a^{-1} \right) \left( e \cdot \tilde{h}_{n+1} \right) \\
&+& \{n+2\}s^{-2}\left( aW_{n+1} - a^{-1} W^*_{n+1} \right) - \{n+1\}s^{-1}\left( \bar{\beta}_{n+1} - \beta_{n+1} \right) \left( W_n - W^*_n \right) \\
=&& \left( s^{-1}a^{-1} + s^{-3}a \right) \left( \tilde{h}_{n+1} \cdot e \right) - \left( s^{-1}a+s^{-3}a^{-1} \right) \left( e \cdot \tilde{h}_{n+1} \right) \\
&+& [n+2]\left( s^{-1}-s^{-3} \right) \left(aW_{n+1} - a^{-1}W^*_{n+1} \right) \\
&+& [n+1]\left( 1-s^{-2} \right) \left( \bar{\beta}_{n+1}-\beta_{n+1} \right) \left(W_n - W^*_n \right) \\
=&& s^{-1}a\left( [n+2]W_{n+1} - e \cdot \tilde{h}_{n+1} \right) \\
&+& s^{-3}a^{-1}\left( [n+2]W^*_{n+1} - e \cdot \tilde{h}_{n+1} \right) \\
&-& s^{-3}a\left( [n+2]W_{n+1} - \tilde{h}_{n+1} \cdot e \right) \\
&-& s^{-1}a^{-1}\left( [n+2]W^*_{n+1} - \tilde{h}_{n+1} \cdot e \right) \\
&+& [n+1]\left( 1-s^{-2} \right) \left( \bar{\beta}_{n+1} -\beta_{n+1} \right) \left( W_n - W^*_n \right) \\
\overset{\ast}{=}&& s^{-1}a\left( [n+1]s^{-1}aW_n + [n+1]s^{-1}\beta_{n+1}a^{-1}W^*_n \right) \\
&+& s^{-3}a^{-1}\left( [n+1]sa^{-1}W^*_n + [n+1]s\bar{\beta}_{n+1}aW_n \right) \\
&-&s^{-3}a\left( [n+1]saW_n + [n+1]s\bar{\beta}_{n+1}a^{-1}W^*_n \right) \\
&-& s^{-1}a^{-1}\left( [n+1]s^{-1}a^{-1}W^*_n + [n+1]s^{-1}\beta_{n+1}aW_n \right) \\
&+& [n+1]\left(1-s^{-2} \right) \left( \bar{\beta}_{n+1} - \beta_{n+1} \right) \left( W_n - W^*_n \right)\\
=&& [n+1]\left( s^{-2}a^2W_n + s^{-2}\beta_{n+1}W^*_n + s^{-2}a^{-2}W^*_n + s^{-2}\bar{\beta}_{n+1}W_n - s^{-2}a^{2}W_n \right. \\
&&\qquad\quad \left. - s^{-2}\bar{\beta}_{n+1}W^*_n - s^{-2}a^{-2}W^*_n - s^{-2}\beta_{n+1}W_n + \bar{\beta}_{n+1}W_n - \bar{\beta}_{n+1}W^*_n - \beta_{n+1}W_n \right. \\
&&\qquad\quad \left. + \beta_{n+1}W^*_n - s^{-2}\bar{\beta}_{n+1}W_n + s^{-2}\bar{\beta}_{n+1}W^*_n + s^{-2}\beta_{n+1}W_n - s^{-2}\beta_{n+1}W^*_n \right) \\
=&& [n+1]\left( \bar{\beta}_{n+1} - \beta_{n+1} \right) \left( W_n - W^*_n \right) \\
=&& \left( \bar{\beta}_{n+1} - \beta_{n+1} \right) \left( \left( [n+1]W_n \right) - \left( [n+1]W^*_n \right) \right) \\
\overset{\ast}{=}&& \left( \bar{\beta}_{n+1} - \beta_{n+1} \right) \left( \left( e \cdot \tilde{h}_{n} + [n]s^{-1}aW_{n-1} + [n]s^{-1}\beta_{n}a^{-1}W^*_{n-1} \right) \right. \\
&& \qquad\qquad\qquad\qquad \left. -\left( e \cdot \tilde{h}_{n} + [n]sa^{-1}W^*_{n-1} + [n]s\bar{\beta}_{n}aW_{n-1} \right) \right) \\
=&& \left( \bar{\beta}_{n+1} - \beta_{n+1} \right) \left( [n]\left( s^{-1} - s\bar{\beta}_{n} \right) aW_{n-1} - [n]\left( s-s^{-1}\beta_{n} \right) a^{-1}W^*_{n-1} \right) \\
=&& [n]\left(\bar{\beta}_{n+1} - \beta_{n+1} \right) \left( s - s^{-1}\beta_{n} \right) \left( aW_{n-1} - a^{-1}W^*_{n-1} \right) \\
=&-& \{n\}\left( aW_{n-1} - a^{-1}W^*_{n-1} \right).
\end{eqnarray*}
Where the last equality follows from the identity
\begin{equation}
 \left(\bar{\beta}_{n+1} - \beta_{n+1} \right) \left( s - s^{-1}\beta_{n} \right) = -(s - s^{-1})
\end{equation}
which may be shown using simple algebraic manipulations in the base ring. This completes the proof.
\end{proof}

\begin{remark}
Theorem \ref{thm:powersumcommutator} is part of a Dubrovnik generalization of the ``Miraculous Cancellations" first proven by Bonahon and Wong in \cite{BW16} and again by Lê in \cite{Le15} using only skein theoretic methods. Under the natural transformation $\eta:\cd \to \ck$, the $\widetilde{P}_k$ are sent to the Chebyshev polynomials $T_k$ (see Section \ref{sec:bracketcompatibility}), and the relation of Theorem \ref{thm:powersumcommutator} is sent to a commutation relation for the $T_k$. Now, if $s$ is specialized to be a $2k$-th root of unity so that $s^k = s^{-k}$, then the commutation relation implies that any threading of $T_k$ in any skein algebra $\ck_{s^{2k}=1}(\Sigma)$ is a central element. More generally, a threading in any skein module $\ck_{s^{2k}=1}(M, N)$ is \textit{transparent}, which essentially means it may move freely about the space, past any tangle. This gives rise to the \textbf{Chebyshev-Frobenius homomorphism}, first defined for skein algebras in \cite{BW16} and extended to general relative skein modules in \cite{LP19}.

It would be nice to say that Theorem \ref{thm:powersumcommutator} implies a generalization of the Chebyshev-Frobenius homomorphism for Dubrovnik skein modules, but there are a couple problems currently preventing this. The first is that our choice of base ring is $\Q(s,v)$ doesn't allow for specializations to roots of unity. This choice is also used in \cite{BB01} and subsequently \cite{LZ02}, the results of which we use here in a fundamental way. Even in the literature regarding HOMFLYPT skein theory, the elements $(s^k - s^{-k})^{-1}$ are adjoined to the base ring. Solving this first problem would show that the threading of any knot by $\tilde{P}_k$ in any skein algebra $\cd(\Sigma)$ would be central. The second problem is that Theorem \ref{thm:powersumcommutator} doesn't imply self-transparency. That is, it doesn't show how $\tilde{P}_k$ may pass through itself. It would be interesting to see if this problem may be resolved somehow, perhaps by following arguments found in \cite{Le15}, but it is not something we address outside of this remark. 
\end{remark}

\section{A Presentation of $\cd(T^2)$} \label{sec:presentation}

In this section we will demonstrate the value of the elements $\widetilde{P}_k$  by showing that the skein algebra $\cd(T^2)$ admits a very simple presentation. First things first, let's define the generators. Recall that given an extended rational number $r$, there is a (homotopically) distinct oriented simple closed curve on the flat torus $T^2$ with slope $r$, and hence a smooth embedding of the annulus 
\[
\iota_r: A \hookrightarrow T^2
\]
into tubular neighborhood of the curve. This induces an algebra homomorphism
\[
\cd(\iota_r): \cd(A) \to \cd(T^2)
\]
on the level of skein algebras. Any knot in $\cd(T^2)$ is contained in the image of some $\cd(\iota_r)$. Since $\cd$ is an unoriented skein theory, let's only consider $r \geq 0$ to avoid redundancy due to the choice of orientations. Given any equivalence class $\xx = (a, b)$ in $\Z^2 / \langle \xx = - \xx \rangle$ with $k := \gcd(a, b)$, define the element
\[
\widetilde{P}_\xx := \cd\big(\iota_{a/b}\big)(\widetilde{P}_k)
\]
to be the embedding of the ``power sum" element $\widetilde{P}_k$ into this tubular neighborhood. Let's state the main theorem of the paper.

\begin{theorem} \label{thm:toruspresentation}
The skein algebra $\cd(T^2)$ is presented by generators 
\[
\{ \widetilde{P}_\xx \mid \xx \in \Z^2 / \langle \xx = - \xx \rangle \}
\]
and relations
\begin{equation} \label{eq:mainrelations}
[\widetilde{P}_\xx, \widetilde{P}_\yy] = (s^{\det(\xx, \yy)} - s^{-\det(\xx, \yy)}) (\widetilde{P}_{\xx + \yy} - \widetilde{P}_{\xx - \yy}).
\end{equation}
\end{theorem}

\begin{corollary} \label{cor:liealgebra1}
The linear span of the set $\{ \widetilde{P}_\xx \mid \xx \in \Z^2 / \langle \xx = - \xx \rangle \}$ is a Lie algebra, and  $\cd(T^2)$ is its universal enveloping algebra. 
\end{corollary}

The full proof of this theorem may be found in the collaboration \cite{MPS19}. Here, we will be focused on showing that the generators satisfy the relations given in the presentation. We formulate this as a proposition.
\begin{proposition}
The following special cases of Equation \eqref{eq:mainrelations} hold
\begin{align}
[\widetilde{P}_{1, 0}, \widetilde{P}_{0, n}] &= (s^n - s^{-n}) (\widetilde{P}_{1, n} - \widetilde{P}_{1, -n}) \label{eq:perprelations} \\
[\widetilde{P}_{1, 0}, \widetilde{P}_{1, n}] &= (s^n - s^{-n}) (\widetilde{P}_{2, n} - \widetilde{P}_{0, n}) \label{eq:angledrelations}
\end{align}
for any $n \geq 1$. Furthermore, these relations imply all of the relations defined by Equation \eqref{eq:mainrelations}.
\end{proposition}

Equation \eqref{eq:perprelations} is the image of Equation \eqref{eq:powersumcommutator} under the wiring of $\ca_\cd$ into $\cd(T^2)$ depicted below.

\begin{figure}[h]
\centering
$\pic[7]{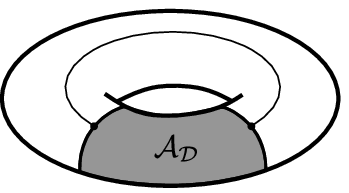}$
\caption{A wiring diagram which defines a linear map $\ca_\cd \to \cd(T^2)$.}
\end{figure}

We will not give the proof of Equation \eqref{eq:angledrelations} here, but the idea is similar to the idea for the proof of Equation \eqref{eq:perprelations}: first show a similar relation holds in the skein algebra of the annulus relative to two points $\cd(A, [2])$, and then Equation \eqref{eq:angledrelations} is the image of this relation under a simple wiring into $\cd(T^2)$. 

What's left for us to show here is the second statement of the proposition, that Equations \eqref{eq:perprelations} and \eqref{eq:angledrelations} imply Equation \eqref{eq:mainrelations}. We will devote the rest of this section to doing so, starting with two technical lemmas and ending with Proposition \ref{lemma_allfromsome}. 

We use the notation
\begin{align*} 
d(\xx, \yy) &:= \det\left[\xx\,\, \yy\right] \quad \quad \,\,  \textrm{for } \xx, \yy \in \Z^2, \\
d(\xx) &:= gcd(m,n) \quad \quad \textrm{when } \xx = (m,n).
\end{align*} 
We will also use the following terminology: 
\[
(\xx,\yy) \in \Z^2 \times \Z^2 \textrm{ is \emph{good} if  } [\widetilde{P}_\xx,\widetilde{P}_\yy] = \{d(\xx,\yy)\} \left( \widetilde{P}_{\xx+\yy}-\widetilde{P}_{\xx-\yy} \right).
\]

\begin{remark}\label{remark_goodsymmetry}
Note that because $\widetilde{P}_{\xx}=\widetilde{P}_{-\xx}$, if $(\xx, \yy)$ is good, then the pairs $(\pm\xx, \pm\yy)$ are good as well. 
\end{remark}

The idea of the proof is to induct on the absolute value of the determinant of the matrix with columns $\xx$ and $\yy$. To induct, we write $\xx = \aab + \bb$ for carefully chosen vectors $\aab, \bb$ and then use the following lemma. 

\begin{lemma}\label{lemma_trueforab}
	Assume $\aab + \bb = \xx$ and that $(\aab,\bb)$ is good. Further assume that the five pairs of vectors $(\yy, \aab)$, $(\yy, \bb)$, $(\yy+\aab,\bb)$, $(\yy+\bb,\aab)$, and $(\aab-\bb, \yy)$, are good. Then the pair $(\xx,\yy)$ is good.
\end{lemma}
\begin{proof}
We  use the Jacobi identity and the goodness assumptions to compute
	\begin{eqnarray*}
&-&\{d(\aab, \bb)\} [\widetilde{P}_{\aab+\bb}, \widetilde{P}_y] + \{d(\aab, \bb)\} [\widetilde{P}_{\aab-\bb}, \widetilde{P}_y] \\
=&-&[[\widetilde{P}_{\aab}, \widetilde{P}_{\bb}], \widetilde{P}_{\yy}] \\
=&&[[\widetilde{P}_{\yy}, \widetilde{P}_{\aab}], \widetilde{P}_{\bb}] + [[\widetilde{P}_{\bb}, \widetilde{P}_y], \widetilde{P}_{\aab}] \\
=&&\{d(\yy, \aab)\} [\widetilde{P}_{\yy+\aab} - \widetilde{P}_{\yy-\aab}, \widetilde{P}_{\bb}] + \{d(b,y)\} [\widetilde{P}_{\bb+\yy} - \widetilde{P}_{\bb-\yy}], \widetilde{P}_{\aab}] \\
=&&\{d(\yy,\aab)\} \left( \{d(\yy+\aab, \bb)\} \left( \widetilde{P}_{\yy+\aab+\bb} - \widetilde{P}_{\yy+\aab-\bb} \right) \right. \\
&& \left. \qquad\qquad -\{d(\yy-\aab,\bb)\} \left( \widetilde{P}_{\yy-\aab+\bb} - \widetilde{P}_{\yy-\aab-\bb} \right) \right) \\
&+&\{d(\bb,\yy)\} \left( \{d(\bb+\yy, \aab)\} \left( \widetilde{P}_{\bb+\yy+\aab} - \widetilde{P}_{\bb+\yy-\aab} \right) \right. \\
&& \left. \qquad\qquad -\{d(\bb-\yy, \aab)\} \left( \widetilde{P}_{\bb-\yy+\aab} - \widetilde{P}_{\bb-\yy-\aab} \right) \right) \\
=&& \left( \{d(\yy,\aab)\} \{d(\yy+\aab,\bb)\} + \{d(\bb,\yy)\} \{d(\bb+\yy,\aab)\} \right) \widetilde{P}_{\aab+\bb+\yy} \\
&+& \left( \{d(\yy,\aab)\} \{d(\yy-\aab,\bb)\} - \{d(\bb,\yy)\} \{d(\bb-\yy,\aab)\} \right) \widetilde{P}_{\aab+\bb-\yy} \\
&-& \left( \{d(\yy,\aab)\} \{d(\yy+\aab,\bb)\} - \{d(\bb,\yy)\} \{d(\bb-\yy,\aab)\} \right) \widetilde{P}_{\aab-\bb+\yy} \\
&-& \left( \{d(\yy,\aab)\} \{d(\yy-\aab,\bb)\} + \{d(\bb,\yy)\} \{d(\bb+\yy,\aab)\} \right) \widetilde{P}_{\aab-\bb-\yy} \\
=:&& c_1 \widetilde{P}_{\aab+\bb+\yy} + c_2 \widetilde{P}_{\aab+\bb-\yy} - c_3 \widetilde{P}_{\aab-\bb+\yy} - c_4 \widetilde{P}_{\aab-\bb-\yy}.
	\end{eqnarray*}
Define constants $\{n\}^+ := s^n + s^{-n}$, and these satisfy the relations
\[
\{n\}\{m\} = \{n+m\}^+ - \{n-m\}+.
\]
Using some simple algebra, we can show
	\begin{eqnarray*}
c_1 \,\,\, =& &  \{d(\yy,\aab)\} \{d(\yy+\aab,\bb)\} + \{d(\bb,\yy)\} \{d(\bb+\yy,\aab)\} \\
=&& \{d(\yy,\aab)+d(\yy+\aab,\bb)\}^+ - \{d(\yy,\aab)-d(\yy+\aab,\bb)\}^+ \\
&+& \{d(\bb,\yy)+d(\bb+\yy,\aab)\}^+ - \{d(\bb,\yy)-d(\bb+\yy,\aab)\}^+  \\
=&& \{d(\yy,\aab+\bb)+d(\aab,\bb)\}^+ - \{d(\yy,\aab-\bb)-d(\aab,\bb)\}^+ \\
&+& \{d(\yy,\aab-\bb)-d(\aab,\bb)\}^+ - \{d(\aab,\bb)-d(\yy,\aab+\bb)\}^+ \\
=&& \{d(\aab,\bb)+d(\yy,\aab+\bb)\}^+ - \{d(\aab,\bb)-d(\yy,\aab+\bb)\}^+ \\
=&& \{d(\aab,\bb)\} \{d(\yy,\aab+\bb)\} \\
=&& - \{d(\aab,\bb)\} \{d(\xx,\yy)\}.
	\end{eqnarray*}
Similar computations for the other $c_i$ show that 
	\begin{align}
\frac{1}{\{d(\aab,\bb)\}}[[\widetilde{P}_\aab,\widetilde{P}_\bb],\widetilde{P}_\yy] &=  \frac{-1}{\{d(\aab,\bb)\}} \left( c_1 \widetilde{P}_{\aab+\bb+\yy} + c_2 \widetilde{P}_{\aab+\bb-\yy} - c_3 \widetilde{P}_{\aab-\bb+\yy} - c_4 \widetilde{P}_{\aab-\bb-\yy}\right) \notag\\ 
&=  \{d(\xx,\yy)\} \left( \widetilde{P}_{\xx+\yy} - \widetilde{P}_{\xx-\yy} \right) 
-\{d(\aab-\bb,\yy)\} \left( \widetilde{P}_{\aab-\bb+\yy} - \widetilde{P}_{\aab-\bb-\yy} \right)\notag \\
&= \{d(\xx,\yy)\} \left( \widetilde{P}_{\xx+\yy} - \widetilde{P}_{\xx-\yy} \right) 
- [\widetilde{P}_{\aab-\bb},\widetilde{P}_\yy]. \label{eq:astep}
	\end{align}
	Since the pair $(\aab,\bb)$ is good, we have 
\begin{equation}\label{eq:added}
[\widetilde{P}_{\aab}, \widetilde{P}_{\bb}] = \{d(\aab, \bb )\} \big( \widetilde{P}_{\xx} - \widetilde{P}_{\aab-\bb} \big).
\end{equation}
Finally, combining equations \eqref{eq:added} and \eqref{eq:astep} shows that the pair $(\xx,\yy)$ is good.
\end{proof}

Now we'll need to prove the following elementary lemma (which is a slight modification of Lemma 1 in \cite{FG00}). This is used to make a careful choice of vectors $\aab, \bb$ so that the previous lemma can be applied. 

\begin{lemma}\label{lemma_diophantine}
	Suppose $p,q \in \N$ are relatively prime with $q < p$. Then there exist $u,v,w,z \in \Z$ such that the following conditions hold:
	\begin{eqnarray*}
	0 \,<\, u, w &<& p, \\
	\det\left[ \begin{array}{cc} u&w\\v&z\end{array}\right] &=& 1, \\
	\left[ \begin{array}{cc} u&w\\v&z\end{array}\right]\left[\begin{array}{c}1\\1\end{array}\right] &=& \left[ \begin{array}{c}p\\q\end{array}\right].
	\end{eqnarray*}
\end{lemma}
\begin{proof}
	Since $p$ and $q$ are relatively prime, there exist $a,b \in \Z$ with $ bq - ap = 1$. This solution can be modified to give another solution $a' = a + q$ and $b' = b + p$, so we may assume $0 \leq b < p$. We then define 
	\[
	u=b,\quad v=a,\quad w=p-b,\quad z=q-a.
	\]
	By definition, $u,v,w,z$ satisfy $u + w = p$ and $v + z = q$, and the inequalities $0 \leq b < p$ and $p > 1$ imply the condition $0 < u, w < p$. To finish the proof, we compute
	\[
	uz - wv = b(q-a) - a(p-b) = bq - ap = 1.
	\]
\end{proof}

\begin{remark}\label{remark_gl2z}
The mapping class group of a (smooth) surface $\Sigma$ is the group of isotopy classes of (smooth) orientation-preserving homeomorphisms $\Sigma \to \Sigma$. Any skein algebra $\cs_X(\Sigma)$ is acted on by the mapping class group of $\Sigma$: the mapping class group is generated by Dehn twists on $\Sigma$ and may be extended trivially onto $\Sigma \times I$, which induce algebra endomorphisms on the skein algebra due to functoriality of $\cs_X(-)$. Thus, there is an $SL_2(\Z)$-action on $\cd(T^2)$ and is such that $g \cdot \widetilde{P}_\xx = \widetilde{P}_{g\xx}$ for any $g \in \SL_2(\Z)$. One could choose to extend this to an action of the \textit{extended} mapping class group of $\Sigma$ (that is, to include orientation-reversing diffeomorphisms) so that $GL_2(Z)$ acts on $\cd(T^2)$. If matrices of determinant $-1$ are to act on $\cd(T^2)$ in the same way as before, then they will induce anti-linear algebra anti-homomorphisms on $\cd(T^2)$. One could turn these into honest algebra homomorphisms by composing with the mirror map (note that $\widetilde{P}_\xx$ is fixed by the mirror map since the BMW symmetrizers are). It will be important that the orbits of this action on the $\widetilde{P}_{\xx}$ are the fibers of the assignment $\widetilde{P}_{\xx} \mapsto d(\xx)$ (which is essentially the statement of Lemma \ref{lemma_diophantine}). In other words, up to the action of $GL_2(\Z)$, vectors in $\Z^2$ are classified by the GCD of their entries.
\end{remark}

\begin{proposition}\label{lemma_allfromsome}
Suppose $A$ is an algebra with elements $\widetilde{P}_\xx$ for $\xx \in \Z^2/\langle -\xx = \xx \rangle$ that satisfy equations \eqref{eq:perprelations} and \eqref{eq:angledrelations}. Furthermore, suppose that there is a $\GL_2(\Z)$-action on $A$ so that $g \cdot \widetilde{P}_\xx = \widetilde{P}_{g \xx}$ for any $g \in \GL_2(\Z)$. Then, any pair $(\xx, \yy)$ is a good pair. 
\end{proposition}
\begin{proof}
	The proof proceeds by induction on $\lvert d(\xx,\yy)\rvert $, and the base case $\lvert d(\xx,\yy)\rvert = 1$ is immediate from Remark \ref{remark_gl2z} and the assumption \eqref{eq:perprelations} for $\xx = (1,0)$ and $\yy = (0,1)$. We now make the following inductive assumption:
	\begin{equation}\label{assumption1} 
	\textrm{For all } \xx',\yy' \in \Z^2\textrm{ with } \lvert d(\xx',\yy')\rvert < \lvert d(\xx,\yy) \rvert, \textrm{ the pair } (\xx',\yy') \textrm{ is good.}
	\end{equation}
	We would like to show that $[\widetilde{P}_{\xx},\widetilde{P}_{\yy}] = \{d(\xx,\yy)\} \big( \widetilde{P}_{\xx + \yy} - \widetilde{P}_{\xx-\yy} \big)$. By Remark \ref{remark_gl2z}, we may assume
	\[
	\yy = \begin{bmatrix}0\\r\end{bmatrix},\quad \xx = \begin{bmatrix}p\\q\end{bmatrix},\quad d(\xx) \leq d(\yy),\quad 0 \leq q < p.
	\]
	
	If $p=1$, then this equation follows from Equation \eqref{eq:angledrelations}, so we may also assume $p > 1$. Furthermore, we may assume that $r>0$ by Remark \ref{remark_goodsymmetry}. 

	We will now show that if either $d(\xx)=1$ or $d(\yy)=1$, then $(\xx,\yy)$ is good. By symmetry of the above construction of $\xx$ and $\yy$, we may assume $d(\xx)=1$, which immediately implies $q>0$. Furthermore, we may now assume that $r>1$ by assuming Equation \eqref{eq:angledrelations}. We apply Lemma \ref{lemma_diophantine} to $p, q$ to obtain  $u,v,w,z \in \Z$ satisfying 
	\begin{equation}\label{assumption1.49}
	uz - vw = 1,\quad uq - vp = 1,\quad u + w = p,\quad v + z = q,\quad 0 < u,w < p.
	\end{equation}
	We then define vectors $\aab$ and $\bb$ as follows:
	\begin{equation}\label{assumption1.9}
	\aab := \begin{bmatrix}  u\\ v\end{bmatrix},\quad 
	\bb := \begin{bmatrix} w\\ z\end{bmatrix},
	\quad \aab + \bb = \xx,\quad d(\aab, \bb) = 1. 
	\end{equation}

	Using Lemma \ref{lemma_trueforab} and Assumption (\ref{assumption1}), it is sufficient to show that each of $\lvert d(\aab, \bb) \rvert$, $\lvert d(\yy,\bb) \rvert$, $ \lvert d(\yy,\aab) \rvert$, $\lvert d(\yy+\aab,\bb) \rvert$, $\lvert d(\yy+\bb,\aab) \rvert$, and $\lvert d(\aab-\bb,\yy) \rvert$ are strictly less than $pr = \lvert d(\xx,\yy) \rvert$. First, $ \lvert d(\aab,\bb) \rvert = 1$ is strictly less than $pr$ since $p>1$ and $r>0$. Second, $ \lvert d(\yy,\aab) \rvert = ur $ and $ \lvert d(\yy,\bb) \rvert = wr $ are strictly less than $pr$ by the inequalities in (\ref{assumption1.49}). Third, we compute
	\begin{eqnarray*}
		\vert d(\yy+\aab, \bb) \rvert &=&\vert - d(\yy+\aab, \bb) \rvert \\
		&=& \lvert - d(\yy,\bb) - d(\aab, \bb) \rvert \\
		&=& \lvert wr - 1 \rvert \\
		&=& wr-1 \\
		&<& wr \\
		&<& pr.
	\end{eqnarray*}

	Fourth, we compute
	\begin{eqnarray*}
		\lvert -d(\yy+\bb, \aab) \rvert &=& \lvert -d(\yy+\bb, \aab) \rvert \\
		&=& \lvert -d(\yy,\aab) - d(\bb, \aab) \rvert \\
		&=& \lvert ur + 1 \rvert \\
		&=& ur+1 \\
		&\leq& (p-1)r+1 \\
		&=& pr-r+1 \\
		&<& pr.
	\end{eqnarray*}

	Finally, we compute
	\begin{eqnarray*}
		\lvert d(\aab-\bb,\yy) \rvert &=& \lvert d(\aab,\yy) - d(\bb,\yy) \rvert \\
		&=& \lvert d(\yy,\aab) - d(\yy,\bb) \rvert \\
		&=& \lvert ur - wr \rvert \\
		&=& \lvert u-w \rvert r \\
		&<& \lvert u+w \rvert r \\
		&=& pr.	
	\end{eqnarray*}
	
	So we have shown that $(\xx,\yy)$ is good if $d(\xx)=1$ or $d(\yy)=1$. Let us now turn our attention to the more general case. We will immediately split this into cases depending on $q$.
	
	\noindent \emph{Case 1:} Assume $0 < q$. 
	
	Let $p' = p / d(\xx)$ and $q' = q / d(\xx)$. By the assumption $0 < q$, we see that $d(\xx) < p$, so $p' > 1$. We can therefore apply Lemma \ref{lemma_diophantine} to $p',q'$ to obtain $u,v,w,z \in \Z$ satisfying 
	\begin{equation}\label{assumption1.5}
	uz - vw = 1,\quad uq' - vp' = 1,\quad u + w = p',\quad v + z = q',\quad 0 < u,w < p'.
	\end{equation}
	In a way similar to the above, we may pick vectors $\aab$ and $\bb$ like so: 
	\begin{equation}\label{assumption2}
	\aab := \begin{bmatrix}  d(\xx)u\\ d(\xx)v\end{bmatrix},\quad 
	\bb := \begin{bmatrix} d(\xx)w\\ d(\xx)z\end{bmatrix},
	\quad \aab + \bb = \xx,\quad d(\aab, \bb) = d(\xx)^2. 
	\end{equation}
	
	As before, it is sufficient to show that each of $\lvert d(\aab, \bb) \rvert$, $\lvert d(\yy,\bb) \rvert$, $\lvert d(\yy,\aab) \rvert$, $\lvert d(\yy+\aab,\bb) \rvert$, $\lvert d(\yy+\bb,\aab) \rvert$, and $\lvert d(\aab-\bb,\yy) \rvert$ are strictly less than $pr  = \lvert d(\xx,\yy) \rvert$. First,
	\[
	\lvert d(\aab,\bb) \rvert = d(\xx)^2 \leq d(\xx)d(\yy) = d(\xx)r < pr
	\]
	where the last inequality follows from the assumption $0 < q < p$. Second, we can compute $\lvert d(\yy,\bb) \rvert = d(\xx)w r $ and $ \lvert d(\yy,\aab) \rvert =  d(\xx)u r $ are strictly less than $pr$ by the inequalities in (\ref{assumption1.5}). Third, we compute
	\begin{eqnarray*}
		\lvert d(\yy+\aab, \bb) \rvert &=& \lvert -d(\yy+\aab, \bb) \rvert \\
		&=& \lvert -d(\yy,\bb) - d(\aab, \bb) \rvert\\
		&=& \lvert d(\xx)wr - d(\xx)^2 \rvert \\
		&=& d(\xx)wr - d(\xx)^2 \\ 
		&<& d(\xx)wr\\
		&\leq& pr.
	\end{eqnarray*}
	Finally, we compute
	\begin{eqnarray*}
		\lvert d(\yy+\bb, \aab) \rvert &=& \lvert -d(\yy+\bb, \aab) \rvert \\
		&=& \lvert -d(\yy,\aab) - d(\bb, \aab) \rvert \\
		&=& d(\xx)ur + d(\xx)^2 \\
		&\leq& \left(d(\xx)u + d(\xx)\right)d(\yy)\\
		&=& (u+1)d(\xx)r.
	\end{eqnarray*}
	
	Therefore, we will be finished once we show that $(u+1)d(\xx)$ is strictly less than $p$. We now split into subcases:\\[2mm]
	
	\noindent\emph{Subcase 1a:} If $u + 1 < p'$, then $(u+1)d(\xx)r < p'd(\xx)r = pr$, and we are done. 
	
	\noindent \emph{Subcase 1b:} Assume $u + 1 = p'$. By equation (\ref{assumption1.5}), we have 
	\[
	1 = uq' - vp' = (p'-1)q' - vp'  \implies p'(q'-v) = 1 + q' < 1 + p'.
	\]
	Since $p' > 1$, the last inequality implies $q' - v = 1$, which implies $v = q'-1$ and $z = 1$ since $v + z = q'$. Now the equation $uz-vw = 1$ implies $(p'-1)-(q'-1) = 1$, which implies $q' = p'-1$. Writing $d=d(\xx)$ for short, we have 
	\[
	\lvert d(\yy + \bb, \aab) \rvert = \left| \det\left[ \begin{array}{cc}d&p -d\\d + r&p-2d\end{array}\right] \right| =  \left| d(p-2d) - (p-d)(d+r) \right| =  \left| rp + d\big(d-r\big) \right|
	\]
	which is at most $rp$ since we are assuming $d(\xx) \leq r$. If this inequality is strict, then we are done. Otherwise, we move onto the next subcase.
	
	\noindent\emph{Subcase 1c:} In this subcase, we are reduced to showing the following pair of vectors is good: 
	\[
	\yy = (0,r),\quad \quad \xx = (rp', rp'-r).
	\]
	If $r=1$, then $d(\yy)=1$, which makes $(\xx,\yy)$ good. Thus, we may assume that $r>1$. 
	
	We must replace our previous choice of $\aab$ and $\bb$ with a choice which is better adapted to this particular subcase. We define
	\[
	\aab := \begin{bmatrix} 1\\-1 \end{bmatrix},\quad \bb := \begin{bmatrix} rp'-1\\rp' - r + 1 \end{bmatrix}.
	\]

	We know that the pairs $(\aab, \bb), (\yy, \aab), (\yy+\bb,\aab)$ are good since $d(\aab)=1$. Since $r > 1$ and $p'>1$, we can compute that the determinants of the pairs $(\yy,\bb), (\yy+\aab,\bb), (\aab-\bb,\yy)$ are strictly less than $r^2p' = \lvert d(\xx,\yy) \rvert$:
	\begin{eqnarray*}
		\lvert d(\aab-\bb,\yy)\rvert &=&  \lvert r^2p' - 2r \rvert \\
		&=& r^2p' - 2r \\
		&<& r^2p'
	\end{eqnarray*}

	\begin{eqnarray*}
		\lvert d(\yy,\bb)\rvert &=&  \lvert r^2p' - r \rvert \\
		&=& r^2p' - r \\
		&<& r^2p'
	\end{eqnarray*}

	\begin{eqnarray*}
		\lvert d(\yy+\aab, \bb)\rvert &=&  \lvert r^2p' - (rp'+r-1) \rvert \\
		&=&  r^2p' - (rp'+r-1) \\
		&<& r^2p'.
	\end{eqnarray*}

This together with Assumption (\ref{assumption1}) and Lemma \ref{lemma_trueforab} shows that $(\xx,\yy)$ is good, which finishes the proof of this subcase and finishes the proof of Case 1. \\[2mm]
	
	\noindent \emph{Case 2:} In this case we assume $q=0$. We define $\aab, \bb$ similarly to Subcase 1c, so we have
	\[
	\yy = \begin{bmatrix} 0\\r\end{bmatrix},\quad \xx = \begin{bmatrix} p\\0\end{bmatrix},\quad \aab := \begin{bmatrix} 1\\-1 \end{bmatrix},\quad \bb := \begin{bmatrix} p-1\\ 1 \end{bmatrix}.
	\]

	Since $d(\aab) = d(\bb) = 1$, the pairs $(\aab,\bb), (\yy,\aab), (\yy,\bb), (\yy+\aab,\bb), (\yy+\bb,\aab)$ are all good. We must check that $\lvert d(\aab-\bb,\yy) \rvert < pr = \lvert d(\xx,\yy) \rvert$. If $r=1$, then the relation \eqref{eq:perprelations} implies that the pair $(\xx,\yy)$ is good. Thus, we may assume that $r>1$. We may also assume that $p>1$. Finally, we check $\lvert d(\aab-\bb,\yy) \rvert = \lvert rp-2r \rvert = rp-2r < rp$. By using Assumption (\ref{assumption1}) and Lemma \ref{lemma_trueforab}, this completes Case 2 which completes the proof. 
\end{proof}

\section{Relationships With $\ck(T^2)$ and $\ch(T^2)$} \label{sec:compatibility}

\subsection{The Kauffman Bracket Case} \label{sec:bracketcompatibility}

Recall that the Dubrovnik skein relations satisfy the Kauffman Bracket skein relations, and hence there is a natural transformation of functors
\[
\eta: \cd \to \ck
\]
whose components are essentially the identity map (more details are given in Section \ref{sec:skeintheories}). In particular, there is an algebra homomorphism from the Birman-Murakami-Wenzl algebra to the Temperley-Lieb algebra.

\begin{theorem} \label{thm:symmtoJW}
Under the algebra homomorphism
\[
\eta_{(I^2, \underline{n})}: BMW_n \to TL_n
\]
the BMW symmetrizer $\tilde{y}_n$ is sent to the Jones-Wenzl idempotent $JW_n$.
\end{theorem}
\begin{proof}
The element $JW_n$ is the unique element of $TL_n$ satisfying the properties
\begin{enumerate}
\item $JW_n \neq 0$,
\item $JW_n^2 = JW_n$,
\item $JW_n c_i = c_i JW_n = 0$ for all cap-cup generators $c_i$. 
\end{enumerate}
Let's abuse some notation by using $\eta$ instead of $\eta_{(I^2, \underline{n+1})}$ since it should be clear from context what we mean. Since $\eta$ is an algebra map, it is true that $\eta(\tilde{y}_n)$ are idempotent and absorb $c_i$ in the same way as $JW_n$, but it is not clear if these images are non-zero. If $n=1$, then the statement is true since $\tilde{y}_1$ and $JW_1$ are the simply the identity, a single strand. We will proceed assuming the induction hypothesis, that  $\eta(\tilde{y}_k) = JW_k$ for all $k \leq n$. The Jones-Wenzl idempotents satisfy the recurrence relation
\begin{equation}\label{eq:jwformula}
\pic[2.7]{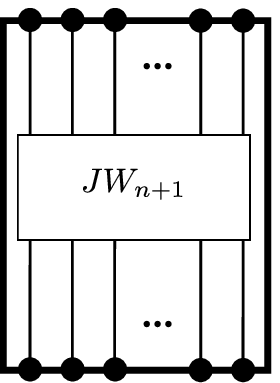} = \pic[2.7]{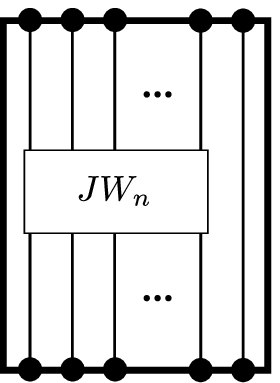} + \frac{s^{2n} - s^{-2n}}{s^{2n+2} - s^{-2n-2}} \pic[2.7]{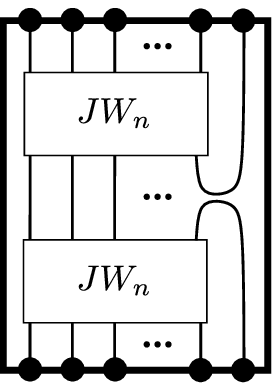}
\end{equation}
(see \cite{Mor17} and substitute $q=s^2$) and recall the recurrence relation for the BMW symmetrizers
\begin{equation}\label{eq:bmwsymmformula}
[n+1] \pic[2.5]{y_n+1.eps} = [n]s^{-1} \pic[2.5]{shelly1.eps} + \pic[2.5]{shelly2.eps} + [n]s^{-1} \beta_n \pic[2.5]{shelly3.eps}
\end{equation}
where \[\beta_n = \frac{1-s^2}{s^{2n-1}v^{-1} - 1}.\]
We finish the proof by showing that the image of Equation \eqref{eq:bmwsymmformula} is Equation \eqref{eq:jwformula} under $\eta$. So specialize $v=-s^{-3}$ and apply $\eta$ to Equation \eqref{eq:bmwsymmformula} and use the induction hypothesis to get
\begin{equation} \label{eq:jwproof1}
[n+1] \pic[2.5]{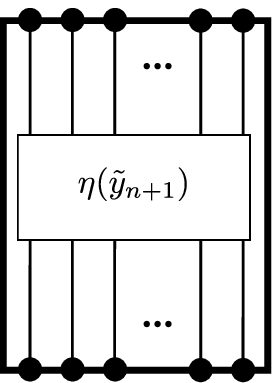} = [n]s^{-1} \pic[2.5]{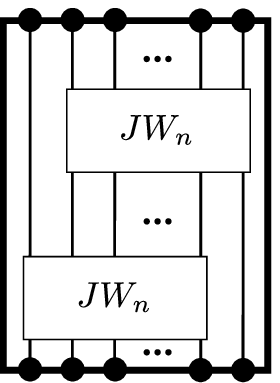} + \pic[2.5]{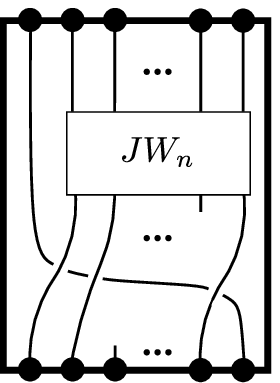} + \frac{s^n - s^{-n}}{s^{2n+2} + 1} \pic[2.5]{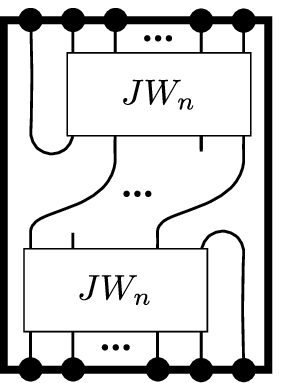}
\end{equation}
which is true after calculating $\eta([n]s^{-1} \beta_n)$. First, let's work with the second diagram on the right-hand side.
\begin{align*}
\pic[2.7]{JWshelly2.eps} = \pic[2.7]{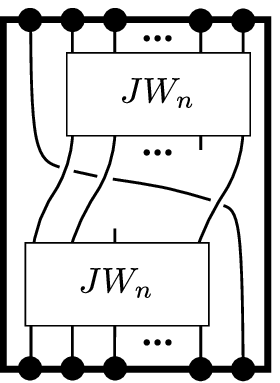} &= s \pic[2.7]{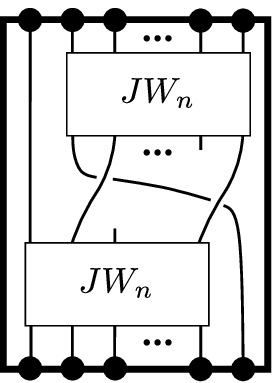} + s^{-1} \pic[2.7]{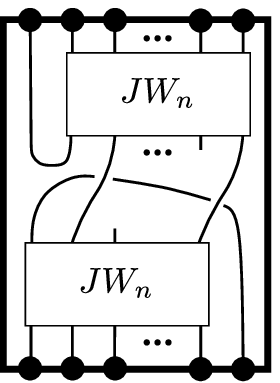} \\
&= s^n \pic[2.7]{JWshelly1.eps} + s^{-n} \pic[2.7]{JWshelly3.eps}
\end{align*}
since $\sigma_i JW_n = JW_n \sigma_i = s JW_n$ for a positive crossing $\sigma_i$. So Equation \eqref{eq:jwproof1} becomes
\begin{equation} \label{eq:jwproof2}
[n+1] \pic[2.2]{eta_n+1.eps} = \left([n]s^{-1} + s^n\right) \pic[2.2]{JWshelly1.eps} + \left(\frac{s^n - s^{-n}}{s^{2n+2} + 1} + s^{-n}\right) \pic[2.2]{JWshelly3.eps}.
\end{equation}
Let's divide both sides of Equation \eqref{eq:jwproof2} by $[n+1]$ and simplify the constants slightly.
\begin{equation} \label{eq:jwproof3}
\pic[2.7]{eta_n+1.eps} = \pic[2.7]{JWshelly1.eps} + \frac{(s^{n+2} - s^{-n-2}) + (s^n - s^{-n})}{(s^{n+2} - s^{-n-2}) - (s^n - s^{-n})} \pic[2.7]{JWshelly3.eps}
\end{equation}
Take the first diagram on the right-hand side and rewrite it using the following technique.
\begin{align*}
\pic[2.7]{JWshelly1.eps} = \pic[2.7]{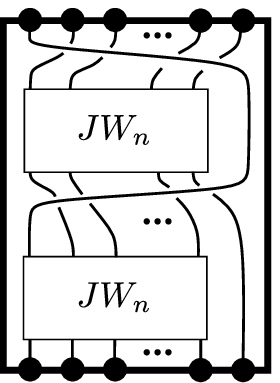} &= s \pic[2.7]{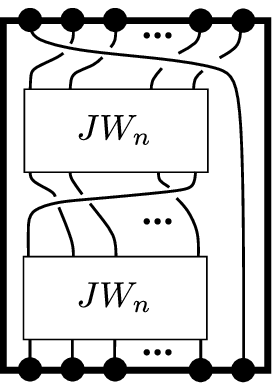} + s^{-1} \pic[2.7]{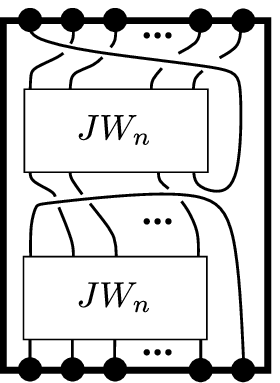} \\
&= s^n \pic[2.7]{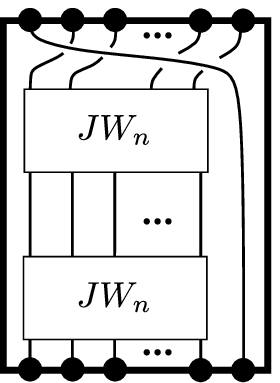} + s^{n-2} \pic[2.7]{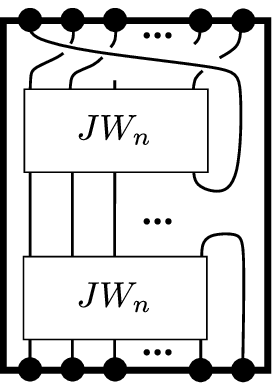} \\
&= s^n \pic[2.7]{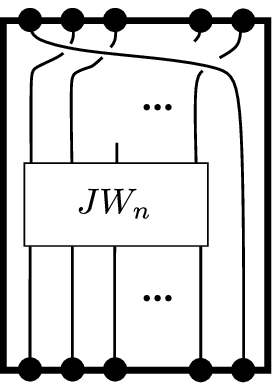} + s^{n-2} \pic[2.7]{JWshelly1wrapb2.eps}
\end{align*}
We can also rewrite the second diagram on the right-hand side of Equation \eqref{eq:jwproof3} in a similar way.
\begin{align*}
\pic[2.7]{JWshelly3.eps} &= \pic[2.7]{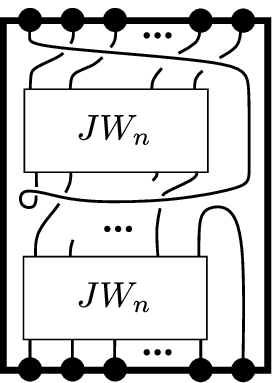} = -s^{-3} \pic[2.7]{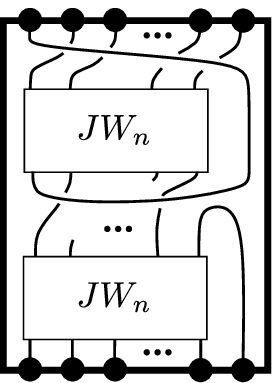} = -s^{-n-2} \pic[2.7]{JWshelly1wrapb2.eps}
\end{align*}
Then Equation \eqref{eq:jwproof3} becomes
\begin{equation}
\pic[2.4]{eta_n+1.eps} = s^n \pic[2.4]{JWshelly1wrapa2absorb.eps} + \left(s^{n-2} - \frac{1 - s^{-2n-4} + s^{-2} - s^{-2n-2}}{s^{n+2} - s^{-n-2} - s^n + s^{-n}} \right) \pic[2.4]{JWshelly1wrapb2.eps}
\end{equation}
Multiply both sides of the equation by $\sigma_n \cdots \sigma_1$ to get the equation
\begin{equation} \label{eq:jwproof5}
\pic[2.4]{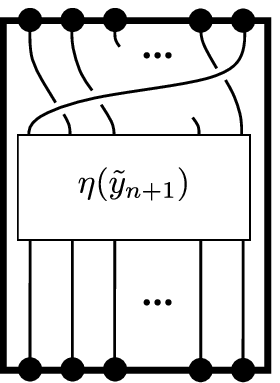} = s^n \pic[2.4]{JWntimes1.eps} + \left(s^{n-2} - \frac{1 - s^{-2n-4} + s^{-2} - s^{-2n-2}}{s^{n+2} - s^{-n-2} - s^n + s^{-n}} \right) \pic[2.4]{JWncJWn.eps}
\end{equation}
On the left-hand side, the crossings are absorbed and scale the diagram by $s^n$. Divide both sides by $s^n$. One can show that the equality of scalars
\[
s^{-n} \left(s^{n-2} - \frac{1 - s^{-2n-4} + s^{-2} - s^{-2n-2}}{s^{n+2} - s^{-n-2} - s^n + s^{-n}} \right) = \frac{s^{2n} - s^{-2n}}{s^{2n+2} - s^{-2n-2}}
\]
holds in the base ring. The right-hand side of the resulting equation is precisely the right-hand side of Equation \eqref{eq:jwformula}. Therefore, $\eta(\tilde{y}_{n+1}) = JW_{n+1}$, which completes the proof. 
\end{proof}

\begin{remark}
In \cite{She16}, it is pointed out that, under the composition of algebra maps $BMW_n \to H_n \to TL_n$ (see Section \ref{sec:cube}), the BMW symmetrizer $\tilde{y}_n$ is sent to the Jones-Wenzl idempotent $JW_n$. However, these maps are defined abstractly on generators, and so \textit{a priori} they are not induced by any natural transformations of skein theories. Furthermore, the diagram
\begin{center}
\begin{tikzcd}
	& BMW_n \arrow[dl, "{\pi_H}", two heads] \arrow[dd, "{\eta_{(I^2, \underline{n})}}", two heads] \\
H_n \arrow[dr, "{\pi_{TL}}", two heads] & \\
	& TL_n
\end{tikzcd}
\end{center}
\textit{does not commute} since $\pi_H(c_i)=0$. However, Theorem \ref{thm:symmtoJW} shows that restricting the maps of the diagram on the relevant idempotents \textit{does} commute.
\end{remark}

It was mentioned in Section \ref{sec:skeintheories} that \textit{Chebyshev polynomials} are important elements of the skein algebra $\ck(A) = R[z]$. Let's now elaborate what that means. The 
\textbf{Chebyshev polynomials of the second kind} $C^S_n = C^S_n(z)$ are defined recursively as
\[
C^U_0 = 1, \qquad \qquad C^U_1 = 2z, \qquad \qquad C^U_n = 2 z C^U_{n-1} - C^U_{n-2}.
\]
Let's consider slightly modified versions $U_n := C_n^U(z/2)$. These polynomials have amazing properties, most of which we won't discuss here. One such fact is that the annular closure of the Jones-Wenzl idempotent $JW_n$ is equal to $U_n(z)$. This can be shown by taking the annular closure of Equation \eqref{eq:jwformula}, the recurrence relation for the $JW_n$, and using some well-known identities to show that the result is the same as the recurrence relation $U_{n+1} = z U_{n} - U_{n-1}$.

There are also \textbf{Chebyshev polynomials of the first kind} $C^T_k = C^T_k(z)$ which are defined recursively as
\[
C^T_0 = 1, \qquad \qquad C^T_1 = z, \qquad \qquad C^T_k = 2 z C^T_{k-1} - C^T_{k-2}.
\]
Again, we will need slightly modified versions of these, $T_k := 2C^T_k(z/2)$.  

It may be interesting to point out that these Chebyshev polynomials satisfy the relations
\begin{align}
C^T_k(z + z^{-1}) &= z^k + z^{-k}  \\
C^U_k(z + z^{-1}) &= \frac{z^k + z^{-k}}{z + z^{-1}}
\end{align}
however, we won't use these facts here.

\begin{corollary}
Under the algebra homomorphism
\[
\eta_A: \cd(A) \to \ck(A)
\]
the image of the special element $\widetilde{P}_k$ is equal to the Chebyshev polynomial $T_k$.
\end{corollary}
\begin{proof}
Since $\eta$ is natural, the closure into the annulus induces a commuting square
\begin{center}
\begin{tikzcd}
BMW_n \arrow[r, "\cl"] \arrow[d, "\eta"] & \cd(A) \arrow[d, "\eta"] \\
TL_n \arrow[r, "\cl"] & \ck(A)
\end{tikzcd}
\end{center}
Recall the $\widetilde{P}_k$ are uniquely determined by the power series equation
\[
\sum_{k=1}^\infty \frac{\widetilde{P}_k}{k} t^k = \ln \Bigg( 1 + \sum_{n=1}^\infty \tilde{h}_n t^n \Bigg).
\]
If we show that the following equality holds
\[
\sum_{k=1}^\infty \frac{T_k}{k} t^k = \ln \Bigg( 1 + \sum_{n=1}^\infty U_n t^n \Bigg)
\]
then the proof is complete since $\eta_A(\tilde{h}_n)=U_n$. The Chebyshev polynomials admit well-known generating functions (see Wikipedia, for example)
\begin{align*}
\sum_{n \geq 0} C_n^U(z)t^n &= \frac{1}{1-2tz+t^2} \\
\sum_{k \geq 1} C_k^T(z)\frac{t^k}{k} &= \ln \left( \frac{1}{\sqrt{1-2tz+t^2}} \right).
\end{align*}
These imply
\begin{align*}
\sum_{n \geq 0} C_n^U(z/2)t^n &= \frac{1}{1-tz+t^2} \\
\sum_{k \geq 1} C_k^T(z/2)\frac{t^k}{k} &= \ln \left( \frac{1}{\sqrt{1-tz+t^2}} \right).
\end{align*}
Multiply the second equation by 2 to get
\[
\sum_{k \geq 1} 2C_k^T(z/2)\frac{t^k}{k} = \ln \left( \frac{1}{1-tz+t^2} \right)
\]
which implies the desired equation since $T_k = 2C_k^T(z/2)$ and $U_n := C_n^U(z/2)$.
\end{proof}

Recall once more the embedding of the annulus
\[
\iota_r: A \hookrightarrow T^2
\]
into a tubular neighborhood of the simple closed curve of rational slope $r=a/b$ for $(a, b)$ considered as an element of $\Z^2 / \langle \xx = - \xx \rangle$. These embeddings induce naturality squares
\begin{center}
\begin{tikzcd}
\cd(A) \arrow[r, "\cd(\iota_r)"] \arrow[d, "\eta"] & \cd(T^2) \arrow[d, "\eta"] \\
\ck(A) \arrow[r, "\ck(\iota_r)"] & \ck(T^2)
\end{tikzcd}
\end{center}
which imply that $\eta_{T^2}(\widetilde{P}_\xx) = T_\xx$ for any $\xx$ where the $T_\xx$ are the generators in the Frohman-Gelca presentation of $\ck(T^2)$. Indeed, one can check that the $T_k$ satisfy the relations Theorem \ref{thm:toruspresentation}:
\begin{align*}
[T_\xx, T_\yy] &= T_\xx T_\yy - T_\yy T_\xx \\
&= (s^{\det(\xx, \yy)} T_{\xx + \yy} + s^{-\det(\xx, \yy)} T_{\xx - \yy}) - (s^{\det(\yy, \xx)} T_{\yy + \xx} + s^{-\det(\yy, \xx)} T_{\yy - \xx}) \\
&= (s^{\det(\xx, \yy)} T_{\xx + \yy} + s^{-\det(\xx, \yy)} T_{\xx - \yy}) - (s^{-\det(\xx, \yy)} T_{\xx + \yy} + s^{\det(\xx, \yy)} T_{\xx - \yy}) \\
&= (s^{\det(\xx, \yy)} - s^{-\det(\xx, \yy)}) (T_{\xx + \yy} - T_{\xx - \yy})
\end{align*}
Let's summarize this discussion with a corollary.
\begin{corollary} \label{cor:frohmangelcacompatibility}
The presentation of $\cd(T^2)$ is compatible with the Frohman-Gelca presentation of $\ck(T^2)$ under the algebra homomorphism
\[
\eta_{T^2}: \cd(T^2) \to \ck(T^2).
\]
\end{corollary}

It should probably be emphasized that this homomorphism is defined topologically on the level of diagrams, and not simply via an abstract assignment of generators. This is important because skein theoretic methods will translate through these types of homomorphisms.

\subsection{The HOMFLYPT Case}

Recall that $\ch(T^2)$ is presented by generators 
\[
\{ P_\xx \mid \xx \in \Z^2 \}
\]
and relations
\[
[P_\xx, P_\yy] = (s^{\det(\xx, \yy)} - s^{-\det(\xx, \yy)}) P_\xx.
\]
This implies that the span of the generators is a Lie algebra and $\ch(T^2)$ is its universal enveloping algebra. Let's use the notation 
\[
\mathfrak{g}_\ch := R\{ P_\xx \mid \xx \in \Z^2 \} \qquad \qquad \mathfrak{g}_\cd := R\{ \widetilde{P}_\xx \mid \xx \in \Z^2 / \langle \xx = -\xx \rangle \}.
\]

\begin{proposition}
There is an injective Lie algebra homomorphism 
\[
\varphi: \mathfrak{g}_\cd \to \mathfrak{g}_\ch 
\]
defined by $\varphi(\widetilde{P}_\xx) = P_\xx + P_{-\xx}$. This induces an algebra homomorphism
\[
U(\varphi): \cd(T^2) \to \ch(T^2)
\]
where $U(-)$ is the universal enveloping algebra functor. Furthermore, the restriction of this homomorphism to an annulus provides an identification for $\cd(A)$ as a subalgebra of $\ch(A)$.
\end{proposition}
\begin{proof}
The assignment is linear since it's defined on basis elements. Let's check that
\[
\varphi\big( [\widetilde{P}_\xx, \widetilde{P}_\yy] \big) = [\varphi(\widetilde{P}_\xx), \varphi(\widetilde{P}_\yy)]
\]
for any $\xx$ and $\yy$. The left-hand side is
\begin{eqnarray*}
&&\varphi\big( [\widetilde{P}_\xx, \widetilde{P}_\yy] \big) \\
=&& \varphi\big( (s^{\det(\xx, \yy)} - s^{-\det(\xx, \yy)})(\widetilde{P}_{\xx+\yy} - \widetilde{P}_{\xx-\yy}) \big) \\
=&& (s^{\det(\xx, \yy)} - s^{-\det(\xx, \yy)}) \big( \varphi(\widetilde{P}_{\xx+\yy}) - \varphi(\widetilde{P}_{\xx-\yy}) \big) \\
=&& (s^{\det(\xx, \yy)} - s^{-\det(\xx, \yy)}) \big( P_{\xx+\yy} + P_{-\xx-\yy} - P_{\xx-\yy} - P_{-\xx+\yy} \big) \\
=&& (s^{\det(\xx, \yy)} - s^{-\det(\xx, \yy)}) P_{\xx+\yy} + (s^{\det(\xx, \yy)} - s^{-\det(\xx, \yy)}) P_{-\xx-\yy} \\
&-& (s^{\det(\xx, \yy)} - s^{-\det(\xx, \yy)}) P_{\xx-\yy} - (s^{\det(\xx, \yy)} - s^{-\det(\xx, \yy)}) P_{-\xx+\yy} \\
=&& (s^{\det(\xx, \yy)} - s^{-\det(\xx, \yy)}) P_{\xx+\yy} + (s^{\det(-\xx, -\yy)} - s^{-\det(-\xx, -\yy)}) P_{-\xx-\yy} \\
&-& (s^{-\det(\xx, -\yy)} - s^{\det(\xx, -\yy)}) P_{\xx-\yy} - (s^{-\det(-\xx, \yy)} - s^{\det(-\xx, \yy)}) P_{-\xx+\yy} \\
=&& [P_\xx, P_\yy] + [P_{-\xx}, P_{-\yy}] + [P_{\xx}, P_{-\yy}] + [P_{-\xx}, P_{\yy}] \\
=&& [P_{\xx} + P_{-\xx}, P_{\yy} + P_{-\yy}] \\
=&& [\varphi(\widetilde{P}_\xx), \varphi(\widetilde{P}_\yy)].
\end{eqnarray*}
Note that $\varphi(\widetilde{P}_\xx) = \varphi(\widetilde{P}_{-\xx})$, so the map is well-defined. This completes the proof.
\end{proof}

Let $\Z^\times$ be the group of units of $\Z$. There is a $\Z^\times$-action on $\ch(T^2)$ given by $(-1) \cdot P_{\xx} = P_{-\xx}$. Then the proposition essentially says that the image of $\phi$ is the subalgebra of invariants with respect to this action.

\begin{remark}
Maybe it's worth pointing out that this homomorphism is simply defined abstractly on generators. It is not clear whether or not this map is induced by something topological in nature, as was the case for the map $\eta_{T^2}: \cd(T^2) \to \ck(T^2)$. After all, we are relating an unoriented skein theory with an oriented one. If this map were to arise from topology, it would have to involve a choice of orientations for each component of each link. All attempts to find such a choice have been unsuccessful. Therefore, it isn't clear that a description of the natural action of $\cd(T^2)$ on $\cd(D^2 \times I)$ is compatible with $\phi$. In other words, it's not immediately clear whether or not the formulas of \cite{MS17} which describe the action of $\ch(T^2)$ on $\ch(D^2 \times I)$ describe analagous formulas in the Dubrovnik case. We examine this closer in the section that follows.
\end{remark}

\section{An Action $\cd(T^2) \curvearrowright \cd(D^2 \times I)$} \label{sec:action}

Recall that the skein algebra of the boundary of a 3-manifold acts on the skein module of that manifold, hence $\cd(D^2 \times I)$ carries the structure of a $\cd(T^2)$-module. In this section we give an implicit description of this action. The first step is to show that $\cd(T^2)$ is generated by five specific elements. The proof we give may be viewed as an algorithm to write each basis element as a product of generators. Next, formulas are given for how each of those generators act. Note that $\cd(D^2 \times I)$ also carries an algebra structure and is isomorphic to $\cd(A)$ as algebras, but the $\cd(T^2)$ action does not commute with the product. Thus, we consider $\cd(D^2 \times I)$ simply as a vector space with basis $\{\tilde{y}_\lambda \}$.

\subsection{A Small Generating Set of $\cd(T^2)$} \label{sec:generators}

First, we prove a technical lemma. 
\begin{lemma} \label{lem:generateline}
Let $\xx, \yy, \zz \in \Z^2$ be  such that $\xx = \yy-\zz$ and so $d := \det(\yy,\zz)$ is non-zero. Then $\widetilde{P}_{\xx+n\zz}$ is an element of the subalgebra generated by $\widetilde{P}_\xx, \widetilde{P}_\yy, \widetilde{P}_\zz$ for all $n \in \Z$. 
\end{lemma}
\begin{proof}
The statement is trivial if $n=0$. By induction on positive $n$, one can express $\widetilde{P}_{\yy+n\zz}$ as
\[
    \widetilde{P}_{\yy+n\zz} = \{ d \}^{-1} \Big( [\widetilde{P}_{\yy+ (n-1)\zz}, \widetilde{P}_\zz] + \{ d \} \widetilde{P}_{\yy+(n-2)\zz} \Big)
\]
which holds for $n=1$ by Theorem \ref{thm:toruspresentation}. Since $\widetilde{P}_\zz = \widetilde{P}_{-\zz}$, we can repeat a similar induction to generate $\widetilde{P}_{\yy -n\zz}$:
\[
    \widetilde{P}_{\yy-n\zz} = \{ d \}^{-1} \Big( [ \widetilde{P}_{\yy-(n-1)\zz}, \widetilde{P}_{-\zz} ] + \{ d \} \widetilde{P}_{\yy-(n-2)\zz} \Big)
\]
where now the base case can be any $n \geq 1$. 
\end{proof}
The above lemma may the thought of geometrically as a way of generating the integer points of a line with rational slope in the lattice $\{ \widetilde{P}_{n,m} \vert (n,m) \in \Z^2 \}$, given the right initial conditions.

\begin{proposition} \label{prop:generators}
The algebra $\cd(T^2)$ is generated by the elements $\widetilde{P}_{0,0}, \widetilde{P}_{1,0}, \widetilde{P}_{0,1}, \widetilde{P}_{1,1}, \widetilde{P}_{2,0}$.
\end{proposition}
\begin{proof}
Let's visualize and keep track of which $\widetilde{P}_{n, m}$ we've generated by dots on a lattice. So, we start with the configuration below.

\begin{center}
\begin{tikzpicture}[scale=.7]
\draw[thin, gray, ->] (0,-5) -- (0,5) node[anchor=north west] {m};
\draw[thin, gray, ->] (-5,0) -- (5,0) node[anchor=south east] {n};
\foreach \x in {-4,-3,-2,-1,1,2,3,4}
\draw[thin, gray] (\x,2pt) -- (\x,-2pt);
\foreach \y in {-4,-3,-2,-1,1,2,3,4}
\draw[thin, gray] (2pt,\y) -- (-2pt,\y);

\node[draw,circle,inner sep=2pt,fill] at (0,0) {};
\node[draw,circle,inner sep=2pt,fill] at (1,0) {};
\node[draw,circle,inner sep=2pt,fill] at (0,1) {};
\node[draw,circle,inner sep=2pt,fill] at (1,1) {};
\node[draw,circle,inner sep=2pt,fill] at (2,0) {};
\end{tikzpicture}
\end{center}

First, we generate the element
\[
    \widetilde{P}_{1,-1} = \{ 1 \}^{-1} \Big( [\widetilde{P}_{1,0}, D_{0,1}] + \{ 1 \}^{-1} D_{1,1} \Big).
\]

\begin{center}
\begin{tikzpicture}[scale=.7]
\draw[thin, gray, ->] (0,-5) -- (0,5) node[anchor=north west] {m};
\draw[thin, gray, ->] (-5,0) -- (5,0) node[anchor=south east] {n};
\foreach \x in {-4,-3,-2,-1,1,2,3,4}
\draw[thin, gray] (\x,2pt) -- (\x,-2pt);
\foreach \y in {-4,-3,-2,-1,1,2,3,4}
\draw[thin, gray] (2pt,\y) -- (-2pt,\y);

\node[draw,circle,inner sep=2pt,fill] at (0,0) {};
\node[draw,circle,inner sep=2pt,fill] at (1,0) {};
\node[draw,circle,inner sep=2pt,fill] at (0,1) {};
\node[draw,circle,inner sep=2pt,fill] at (1,1) {};
\node[draw,circle,inner sep=2pt,fill] at (2,0) {};
\node[draw,circle,inner sep=2pt,fill] at (1,-1) {};
\end{tikzpicture}
\end{center}

Then, we can generate two parallel lines of minimal distance apart by applying \ref{lem:generateline} to the two sets of values
\begin{align*}
        \xx_1 = (0,1), \quad \yy_1 = (1,0), \quad \zz_1 = (1,-1) \\
        \xx_2 = (0,2), \quad \yy_2 = (1,1), \quad \zz_2 = (1,-1)
\end{align*}
thus generating $\widetilde{P}_{n, 1-n}, \widetilde{P}_{n, 2-n}$ for all $n \in \Z$.

\begin{center}
\begin{tikzpicture}[scale=.7]
\draw[thin, gray, ->] (0,-5) -- (0,5) node[anchor=north west] {m};
\draw[thin, gray, ->] (-5,0) -- (5,0) node[anchor=south east] {n};
\foreach \x in {-4,-3,-2,-1,1,2,3,4}
\draw[thin, gray] (\x,2pt) -- (\x,-2pt);
\foreach \y in {-4,-3,-2,-1,1,2,3,4}
\draw[thin, gray] (2pt,\y) -- (-2pt,\y);

\node[draw,circle,inner sep=2pt,fill] at (0,0) {};
\node[draw,circle,inner sep=2pt,fill] at (1,0) {};
\node[draw,circle,inner sep=2pt,fill] at (0,1) {};
\node[draw,circle,inner sep=2pt,fill] at (1,1) {};
\node[draw,circle,inner sep=2pt,fill] at (2,0) {};
\node[draw,circle,inner sep=2pt,fill] at (1,-1) {};

\draw [thin, red,-latex, <->] (-3.5,4.5) -- (4.5, -3.5) {};
\draw [thin, red,-latex, <->] (-2.5,4.5) -- (4.5, -2.5) {};

\foreach \x in {-3,-2,-1,1,2,3,4}
\node[draw,circle,inner sep=2pt,fill] at (\x,1 - \x) {};
\foreach \x in {-2,-1,0,1,2,3,4}
\node[draw,circle,inner sep=2pt,fill] at (\x,2 - \x) {};

\draw [ultra thick,-latex,red] (0,0) -- (1,-1) node [below] {$\zz$};
\draw [ultra thick,-latex,red] (0,1) -- (1,0) {};
\draw [ultra thick,-latex,red] (1,1) -- (2,0) {};
\end{tikzpicture}
\end{center}

Then for any $n \neq 0$, we may apply Lemma \ref{lem:generateline} to the triple
\[
    \xx = (n, 1-n), \quad \yy = (n, 2-n), \quad \zz = (0,1) 
\]
to generate the vertical line of points through $(n,0)$. 

\begin{center}
\begin{tikzpicture}[scale=.7]
\draw[thin, gray, ->] (0,-5) -- (0,5) node[anchor=north west] {m};
\draw[thin, gray, ->] (-5,0) -- (5,0) node[anchor=south east] {n};
\foreach \x in {-4,-3,-2,-1,1,2,3,4}
\draw[thin, gray] (\x,2pt) -- (\x,-2pt);
\foreach \y in {-4,-3,-2,-1,1,2,3,4}
\draw[thin, gray] (2pt,\y) -- (-2pt,\y);

\node[draw,circle,inner sep=2pt,fill] at (0,0) {};
\node[draw,circle,inner sep=2pt,fill] at (1,0) {};
\node[draw,circle,inner sep=2pt,fill] at (0,1) {};
\node[draw,circle,inner sep=2pt,fill] at (1,1) {};
\node[draw,circle,inner sep=2pt,fill] at (2,0) {};
\node[draw,circle,inner sep=2pt,fill] at (1,-1) {};

\draw [ultra thick,-latex,red] (0,0) -- (0,1) node [below left] {$\zz$};

\node[draw,circle,inner sep=2pt,fill] at (0,2) {};

\foreach \x in {-4,-3,-2,-1,1,2,3,4}
\draw [thin, red,-latex, <->] (\x, -4.5) -- (\x, 4.5) {};

\foreach \x in {-4,-3,-2,-1,1,2,3,4}
\foreach \y in {-4,-3,-2,-1,0,1,2,3,4}
\node[draw,circle,inner sep=2pt,fill] at (\x,\y) {};

\foreach \x in {-2,-1,1,2,3,4}
\draw [ultra thick,-latex,red] (\x,-\x + 1) -- (\x,-\x+2) {};

\end{tikzpicture}
\end{center}

Finally, any $\widetilde{P}_{0,m}$ for $m \neq 0$ can be expressed as
\[
    \widetilde{P}_{0,m} = \{ m \}^{-1} \Big( [\widetilde{P}_{1,m}, \widetilde{P}_{-1,0}] + \{ m \} \widetilde{P}_{2,m} \Big)
\]
which generates the last of what is left of the $\widetilde{P}_{n,m}$. This completes the proof.
\end{proof}

\begin{remark}
The proof above provides an algorithm for expressing any $\widetilde{P}_{n,m}$ in terms of the five generators. Of course, this expression is complicated, but it is not surprising since the generating set is so small.
\end{remark}

\subsection{Recalling Some Formulas From the Literature}

Let $\Lambda = (\Lambda_1, \dots, \Lambda_n)$ be a sequence of Young diagrams such that $\Lambda_1 = (1)$, $|\Lambda_{i+1}| = |\Lambda_{i}| \pm 1$, either $\Lambda_{i+1} \subset \Lambda_{i}$ or $\Lambda_{i} \subset \Lambda_{i+1}$, and $\Lambda_n=\lambda$. 
We will call such a sequence an \textit{up-down tableau} of length $n$ and shape $\lambda(\Lambda) := \Lambda_n$. 
If $\Lambda = (\Lambda_1, \dots, \Lambda_{n-1}, \Lambda_n)$, then define $\Lambda' = (\Lambda_1, \dots, \Lambda_{n-1})$. 
In particular, $|\lambda(\Lambda)| = |\lambda(\Lambda')| \pm 1$. Following \cite{BB01}, recursively define morphisms $\fa_\Lambda$ and $\fb_\Lambda$ by:
\begin{align*}
    \fa_1 &:= \id_1 =: \fb_1 
\end{align*}
and if $|\Lambda_n| = |\Lambda_{n-1}| + 1$, then 
\begin{align*}
    \fa_\Lambda &:= (\fa_{\Lambda'} \otimes 1_1) \tilde{y}_{\Lambda(\Lambda)} \\
    \fb_\Lambda &:= \tilde{y}_{\lambda(\Lambda)} (\fb_{\Lambda'} \otimes 1_1)
\end{align*}
and if $|\Lambda_n| = |\Lambda_{n-1}| - 1$, then
\begin{align*}
    \fa_\Lambda &:= \frac{ \langle \Lambda_n \rangle}{ \langle \Lambda_{n-1} \rangle} (\fa_{\Lambda'} \otimes 1_1) (\tilde{y}_{\lambda(\Lambda)} \otimes \cap) \\
    \fb_\Lambda &:= (\tilde{y}_{\lambda(\Lambda)} \otimes \cup) (\fb_{\Lambda'} \otimes 1_1)
\end{align*}
where the morphisms $\cup \in \Hom_{\sfd(I^2)} \big( \varnothing, [2] \big)$ and $\cap \in \Hom_{\sfd(I^2)} \big( [2], \varnothing \big)$ are the obvious cup and cap morphisms, $\otimes$ is the monoidal product of $\sfd(I^2)$ induced by some standard embedding $I^2 \sqcup I^2 \to I^2$, and $\langle \lambda \rangle$ for a partition $\lambda$ is the quantum trace of a certain path idempotent in $BMW_n$ (see \cite{BB01} for details). 

\begin{theorem}[\cite{BB01}, Section 5] \label{thm:bmwbasis}
Let $d_\lambda^{(n)}$ be the number of up-down tableaux of length $n$ and shape $\lambda$. 
Let $\mathcal{M}_{d_\lambda}$ denote the $d_\lambda^{(n)} \times d_\lambda^{(n)}$ matrix algebra. 
Also, let
\[
    q_\Lambda := \alpha_\Lambda \beta_\Lambda
\]

\begin{enumerate} 
\item There is an algebra isomorphism
\[
    \rho: \underset{|\lambda| = n, n-2, \dots}{\bigoplus} \mathcal{M}_{d_{\lambda}^{(n)}} \overset{\sim}{\longrightarrow} BMW_n
\]
such that the image of a standard basis element is $\rho( e_{\Lambda,\Xi}^\lambda) = \alpha_\Lambda \beta_\Xi$ where $\lambda(\Lambda) = \lambda = \lambda(\Xi)$. \\
\item (Branching formula)
\[
    q_\Lambda \otimes 1_1 = \sum_{\Xi'=\Lambda} q_\Xi
\] 
for all up-down tableaux $\Lambda$, and
\[
\tilde{y}_\lambda \otimes \id_1 = \sum_{\substack{\lambda \subset \mu \\ \mu = \lambda + \square}} (\tilde{y}_\lambda \otimes \id_1) \tilde{y}_\mu (\tilde{y}_\lambda \otimes \id_1) + \sum_{\substack{\nu \subset \lambda \\ \lambda = \nu + \square}} \frac{\langle \nu \rangle}{\langle \lambda \rangle} (\tilde{y}_\lambda \otimes \id_1) (\tilde{y}_\nu \otimes c_1) (\tilde{y}_\lambda \otimes \id_1) 
\]
\item (Braiding coefficient)
\[
\tilde{y}_\mu (\tilde{y}_\lambda \otimes \id_1) \sigma_{n-1} \cdot \sigma_1 \sigma_1 \cdot \sigma_{n-1} \tilde{y}_\mu = s^{2\cn(\square)}\tilde{y}_\mu
\]
where the Young diagram of $\lambda$ is obtained by removing the cell $\square$ from the Young diagram of $\mu$.  
\item 
\[
    \fb_\Lambda \fa_\Lambda = \tilde{y}_{\lambda(\Lambda)}
\]
for all up-down tableaux $\Lambda$.
\end{enumerate}
\end{theorem}

\begin{remark} \label{rmk:qclosure}
Recall our notation $\widetilde{Q}_\lambda := \cl(\tilde{y}_\lambda)$. Then Theorem \ref{thm:bmwbasis} implies
\[
    \cl(q_\Lambda) = \cl(\fa_\Lambda \fb_\Lambda) = \cl(\fb_\Lambda \fa_\Lambda) = \cl(\tilde{y}_{\lambda(\Lambda)}) = \widetilde{Q}_{\lambda(\Lambda)}.
\]
\end{remark}

Let's formally restate a theorem which was mentioned in Section \ref{sub:annulus}.

\begin{theorem}[\cite{LZ02}, Corollary 2, Proposition 2.1]\label{thm:ann_basis} 
$\,$
\begin{enumerate}
    \item The elements $\widetilde{Q}_\lambda$ form a basis of the skein module of a solid torus $\cd(D^2 \times S^1)$, where $\lambda$ ranges over all Young diagrams. 
    \item Let $\phi: \cc \to \cc$ be the meridian map so that $\phi(x)$ is a simple loop wrapped around $x$. Then
    \[
        \phi \big( \widetilde{Q}_\lambda \big) = c_\lambda \widetilde{Q}_\lambda
    \]
    where 
        \[c_\lambda = \delta_\cd + (s - s^{-1}) \Bigg( v^{-1}\sum_{x \in \lambda} s^{2\cn(x)} - v\sum_{x \in \lambda} s^{-2\cn(x)} \Bigg)
    \]
\end{enumerate}
\end{theorem}

This is the basis we will choose to act by $\cd(T^2)$ on in Section \ref{sec:action}.

\subsection{Showing How Generators Act} \label{sec:action}

We will fix the convention that $\widetilde{P}_{1,0}$ acts as a meridian link and $\widetilde{P}_{0,1}$ acts as a longitude link. This next proposition shows how each element of the generating set of Proposition \ref{prop:generators} act.

\begin{proposition} \label{prop:generatoractions}
In the basis $\{ \widetilde{Q}_\lambda \}$ of $\cd(D^2 \times S^1)$, the action of $\cd(T^2)$ is determined by the equations below.
\begin{align}
    \widetilde{P}_{1,0} \cdot \widetilde{Q}_\lambda &= \Bigg(\langle \widetilde{P}_1 \rangle + \{1\} \Big( v^{-1}\sum_{\square\in\lambda} s^{2\cn(\square)} - v \sum_{\square\in\lambda} s^{-2\cn(\square)} \Big) \Bigg) \widetilde{Q}_\lambda \label{eq:1,0}\\
    \widetilde{P}_{2,0} \cdot \widetilde{Q}_\lambda &= \Bigg(\langle \widetilde{P}_2 \rangle + \{2\} \Big( v^{-2}\sum_{\square\in\lambda} s^{4\cn(\square)} - v^2 \sum_{\square\in\lambda} s^{-4\cn(\square)} \Big) \Bigg) \widetilde{Q}_\lambda \label{eq:2,0}\\
    \widetilde{P}_{0,1} \cdot \widetilde{Q}_\lambda &= \sum_{\substack{\lambda \subset \mu \\ \mu = \lambda + \square}} \widetilde{Q}_\mu + \sum_{\substack{\nu \subset \lambda \\ \lambda = \nu + \square}} \widetilde{Q}_\nu \label{eq:0,1}\\
    \widetilde{P}_{1,1} \cdot \widetilde{Q}_\lambda &= v^{-1} \sum_{\substack{\lambda \subset \mu \\ \mu = \lambda + \square}} s^{2\cn(\square)} \widetilde{Q}_\mu + v \sum_{\substack{\nu \subset \lambda \\ \lambda = \nu + \square}} s^{-2\cn(\square)} \widetilde{Q}_\nu \label{eq:1,1}
\end{align}
\end{proposition}
\begin{proof}
The equation \eqref{eq:1,0} follows by considering the annular closure of the branching formula of Theorem \ref{thm:bmwbasis}, together with Remark \ref{rmk:qclosure}. 
Also, \eqref{eq:0,1} is precisely the result of the meridian map eigenvalue computation of Theorem \ref{thm:ann_basis}. Equation \eqref{eq:2,0} follows from Proposition \ref{prop:zlgeneralization} for $k=2$. All that's left to show is equation \eqref{eq:1,1}. To do this, we will apply the branching rule and braiding coefficient identities of Theorem \ref{thm:bmwbasis}. We will work with rectangular diagrams with the top edge identified with the bottom edge, creating annular diagrams. First, apply a framing relation and the branching rule
\begin{align*}
\pic{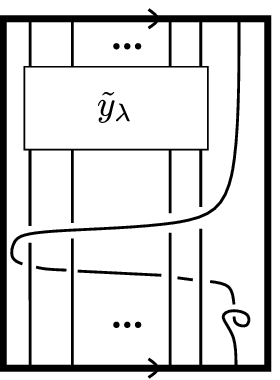} &= v^{-1} \pic{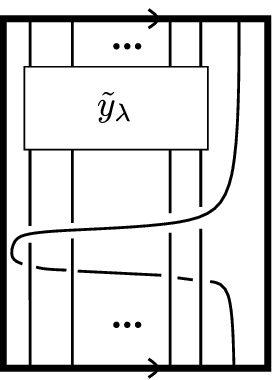} \\
&= v^{-1} \sum_{\substack{\lambda \subset \mu \\ \mu = \lambda + \square}} \pic{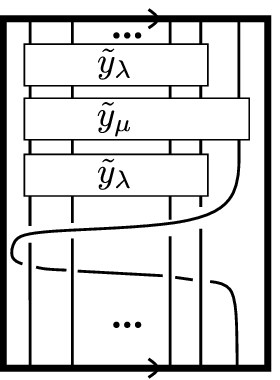} + v^{-1} \sum_{\substack{\nu \subset \lambda \\ \lambda = \nu + \square}} \frac{\langle \nu \rangle}{\langle \lambda \rangle} \pic{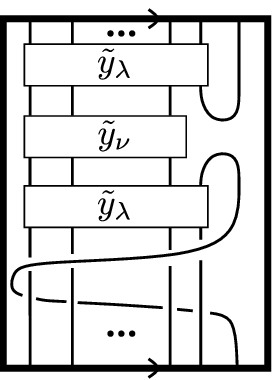}
\end{align*}
where the first sum is over the set of partitions $\mu$ whose Young diagram is obtained by adding a single cell $\square$ to $\lambda$, and similarly for the second sum. Next, we can use the idempotent property to duplicate and absorb idempotents. The right-hand side of the above equation becomes
\[
v^{-1} \sum_{\substack{\lambda \subset \mu \\ \mu = \lambda + \square}} \pic{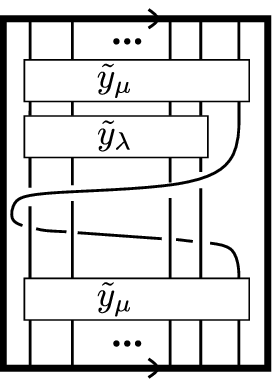} + v^{-1} \sum_{\substack{\nu \subset \lambda \\ \lambda = \nu + \square}} \frac{\langle \nu \rangle}{\langle \lambda \rangle} \pic{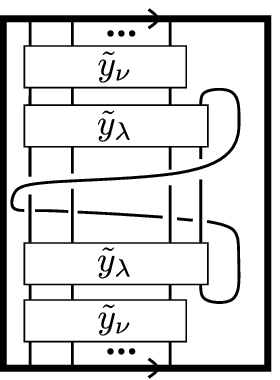}.
\]
By the braiding coefficient formula, the diagram in the first sum is equal to $s^{2\cn(\square)} \widetilde{Q}_\mu$. The diagram in the second sum is calculated during the proof of Proposition 6.1 in \cite{BB01} to be $v^2 s^{-2\cn(\square)} \langle \lambda \rangle / \langle \nu \rangle \widetilde{Q}_\nu$. This completes the proof. 
\end{proof}

\chapter{More Formulas and Partial Results}

\section{Commutation Relations For BMW and Hecke Symmetrizer Closures} \label{sec:morecommutationrelations}

This next theorem follows directly from Equation \eqref{eq:skewcommutator2}, which makes it equivalent to Theorem \ref{thm:powersumcommutator} in some sense. This expresses the left $\cd(A)$-action on $\ca_\cd$ in terms of the right action, and vice versa. This implies a commutation relation for the closures of the BMW symmetrizers in terms of either the elements of the set $\{\tilde{h}_j \cdot a^i \}_{j, i}$ or $\{ a^i \cdot \tilde{h}_j \}_{j, i \geq 0}$ which are subsets of the bases $\{ \widetilde{Q}_\lambda \cdot a^i \}_{i \geq 0, \lambda}$ and $\{ a^i \cdot \widetilde{Q}_\lambda \}_{i \geq 0, \lambda}$ of $\ca_\cd$, respectively. These supersets are bases since $\ca_\cd = \cd(A)[a, a^{-1}]$ as algebras in the category of left $\cd(A)$-modules and because the map defined by $\widetilde{Q}_\lambda \cdot a^i \mapsto a^i \cdot \widetilde{Q}_\lambda$ is an invertible algebra homomorphism. 

\begin{theorem} \label{prop:hncommutator}
For any $n \geq 1$, the relations
\begin{equation}
\tilde{h}_n \cdot e = \sum_{i=0}^n d_i (e \cdot \tilde{h}_{n-i})
\end{equation}
and
\begin{equation}
e \cdot \tilde{h}_n = \sum_{i=0}^n \bar{d}_i (\tilde{h}_{n-i} \cdot e)
\end{equation}
hold in $\ca_\cd$, where
\begin{align*}
d_0 & = 1, \\
d_i & = \sum_{l=0}^{i-1} (1 - s^2) s^{2l-i} a^{i-2l} + (1 - s^{-2}) s^{i-2l} a^{2l-i} \qquad \forall i \geq 1, \\
\bar{d}_i & = \sum_{l=0}^{i-1} (1 - s^{-2}) s^{i-2l} a^{i-2l} + (1 - s^{2}) s^{2l-i} a^{2l-i} \qquad \forall i \geq 1.
\end{align*}
Equivalently,
\begin{equation}
e \cdot \tilde{h}_n - \tilde{h}_n \cdot e = \sum_{i=1}^n \bar{d}_i (\tilde{h}_{n-i} \cdot e)
\end{equation}
or 
\begin{equation}
\tilde{h}_n \cdot e - e \cdot \tilde{h}_n = \sum_{i=1}^n d_i (e \cdot \tilde{h}_{n-i}).
\end{equation}
\end{theorem}
\begin{proof}
The formulas for the $d_i$ were discovered experimentally by coding a solver using the SymPy package in Python. The second equation is just the mirror map applied to the first equation, so we will just prove the first equation.

The idea of the proof depends on a reformulation of Equation \eqref{eq:skewcommutator2} as
\[
\tilde{h}_n \cdot e = e \cdot \tilde{h}_n - ( s a + s^{-1} a^{-1} ) ( e \cdot \tilde{h}_{n-1} ) + e \cdot \tilde{h}_{n-2} + ( s^{-1} a + s a^{-1} ) \tilde{h}_{n-1} \cdot e - \tilde{h}_{n-2} \cdot e
\]
and a recursive application of this formula to its last two terms on the right-hand side of the equation. 

The case of $n=0$ is trivial. For $n=1$, just apply the Kauffman skein relation. Now assume the induction hypothesis, that the formula in the statement is true for all $k \leq n-1$. Then apply this assumption to Equation \eqref{eq:skewcommutator2}:
\begin{align*}
\tilde{h}_n \cdot e &= e \cdot \tilde{h}_n - ( s a + s^{-1} a^{-1} ) ( e \cdot \tilde{h}_{n-1} ) + e \cdot \tilde{h}_{n-2} + ( s^{-1} a + s a^{-1} ) ( \tilde{h}_{n-1} \cdot e ) - \tilde{h}_{n-2} \cdot e \\
&= e \cdot \tilde{h}_n - ( s a + s^{-1} a^{-1} ) ( e \cdot \tilde{h}_{n-1} ) + e \cdot \tilde{h}_{n-2} + ( s^{-1} a + s a^{-1} ) \sum_{i=0}^{n-1} d_i (e \cdot \tilde{h}_{n-1-i}) \\
&\quad \,\, - \sum_{i=0}^{n-2} d_i (e \cdot \tilde{h}_{n-2-i}) \\
&= e \cdot \tilde{h}_n + d_1 ( e \cdot \tilde{h}_{n-1} ) + ( s^{-1} a + s a^{-1} ) \sum_{i=1}^{n-1} d_i (e \cdot \tilde{h}_{n-1-i}) - \sum_{i=1}^{n-2} d_i (e \cdot \tilde{h}_{n-2-i}) \\
&= e \cdot \tilde{h}_n + d_1 ( e \cdot \tilde{h}_{n-1} ) + ( s^{-1} a + s a^{-1} ) \sum_{i=0}^{n-2} d_{i+1} (e \cdot \tilde{h}_{n-2-i}) - \sum_{i=1}^{n-2} d_i (e \cdot \tilde{h}_{n-2-i}) \\
&= e \cdot \tilde{h}_n + d_1 ( e \cdot \tilde{h}_{n-1} ) + ( s^{-1} a + s a^{-1} ) d_1 ( e \cdot \tilde{h}_{n-2} ) \\
&\quad\,\, + \sum_{i=1}^{n-2} \big( ( s^{-1} a + s a^{-1} ) d_{i+1} - d_i \big) (e \cdot \tilde{h}_{n-2-i}). \\
\end{align*}
It is a straightforward computation to show that $( s^{-1} a + s a^{-1} ) d_1 = d_2$:
\begin{align*}
( s^{-1} a + s a^{-1} ) d_1 &= ( s^{-1} a + s a^{-1} ) \big( ( 1 - s^2 ) s^{-1} a + ( 1 + s^{-2} ) s a^{-1} \big) \\
&= ( 1 - s^2 ) s^{-2} a^2 + (1 - s^{-2} ) s^0 a^0 + ( 1 - s^2 ) s^0 a^0 + ( 1 - s^{-2} ) s^2 a^{-2} \\
&= d_2.
\end{align*}
It's slightly more tedious to show that $( s^{-1} a + s a^{-1} ) d_{i+1} - d_i = d_{i+2}$ for all $i \geq 1$:
\begin{eqnarray*}
&&( s^{-1} a + s a^{-1} ) d_{i+1} - d_i \\
=&& ( s^{-1} a + s a^{-1} ) \sum_{l=0}^{i} (1 - s^2) s^{2l-i} a^{i-2l} + (1 - s^{-2}) s^{i-2l} a^{2l-i} \\
&-& \sum_{l=0}^{i-1} (1 - s^2) s^{2l-(i-1)} a^{(i-1)-2l} + (1 - s^{-2}) s^{(i-1)-2l} a^{2l-(i-1)} \\
=&& \sum_{l=0}^{i} (1 - s^2) s^{2l-(i+1)} a^{(i+1)-2l} + (1 - s^{-2}) s^{(i+1)-2l} a^{2l-(i+1)} \\
&+& \sum_{l=0}^{i} (1 - s^2) s^{2l-(i-1)} a^{(i-1)-2l} + (1 - s^{-2}) s^{(i-1)-2l} a^{2l-(i-1)} \\
&-& \sum_{l=0}^{i-1} (1 - s^2) s^{2l-(i-1)} a^{(i-1)-2l} + (1 - s^{-2}) s^{(i-1)-2l} a^{2l-(i-1)} \\
=&& \sum_{l=0}^{i} (1 - s^2) s^{2l-(i+1)} a^{(i+1)-2l} + (1 - s^{-2}) s^{(i+1)-2l} a^{2l-(i+1)} \\
&+& (1 - s^2) s^{i+1} a^{-1-i} + (1 - s^{-2}) s^{-1-i} a^{i+1} \\
=&& d_{i+2}.
\end{eqnarray*}
This completes the proof of the statement. 
\end{proof}

\begin{remark}
There exists an algebra homomorphism from $\cd(A)$ to the ring of symmetric functions $\Lambda_R$ (see Section \ref{sec:Lukac}). Conjecturally, this map is an isomorphism, which would imply that the sets $\{ \tilde{h}_\lambda \cdot a^i \}_{\lambda, i}$ and $\{ a^i \cdot \tilde{h}_\lambda \}_{\lambda, i}$ over integers $i$ and partitions $\lambda$, where $\tilde{h}_\lambda := \tilde{h}_{\lambda_1} \cdots \tilde{h}_{\lambda_r}$, form bases of $\ca_\cd = \cd(A)[a, a^{-1}]$. If so, then Theorem \ref{prop:hncommutator} provides transition formulas between these two bases, which would then give a full description of $\ca_\cd$ as a $\cd(A)$-$\cd(A)$-bimodule. 
\end{remark}

One might expect similar formulas to hold in the HOMFLYPT case. To our knowledge, there is no HOMFLYPT analogue of Theorem \ref{prop:hncommutator} written down in the literature. Let's do that here. 

\begin{lemma} \label{lem:homfly1}
For all integers $n$, the following relation holds in $\ca_\ch$
\begin{equation}
e \cdot h_n - h_n \cdot e = s a \cdot h_{n-1} - h_{n-1} \cdot  s^{-1} a
\end{equation}
where we use the convention $h_0 = 1$ and $h_n = 0$ if $n < 0$. 
\end{lemma}
\begin{proof}
Recall the power sum elements $P_k$ satisfy the power series equation
\begin{equation} \label{def:Pk}
\sum_{k=1}^\infty \frac{P_k}{k} x^k = \ln \Big( \sum_{n=0}^\infty h_n x^n \Big)
\end{equation}
By Theorem 4.2 of \cite{Mor02b}, the power sum elements satisfy a commutation relation in $\ca_\ch$
\begin{equation}
e \cdot P_k - P_k \cdot e = (s^{k} - s^{-k}) a^k
\end{equation} 
which may be rephrased as a power series equation 
\[
e \cdot \Big( \sum_{k=1}^\infty \frac{P_k}{k} x^k \Big) - \Big( \sum_{k=1}^\infty \frac{P_k}{k} x^k \Big) \cdot e = \sum_{k=1}^\infty s^k a^k - \sum_{k=1}^\infty s^{-k} a^k.
\]
On the left-hand side, use the defining equation \eqref{def:Pk}. Use the power series formulation of natual log on the right-hand side. So we have
\[
\ln \Bigg( e \cdot \Big( \sum_{k=0}^\infty h_k x^k \Big) \Bigg) - \ln \Bigg( \Big( \sum_{k=0}^\infty h_k x^k \Big) \cdot e \Bigg) = \ln ( 1 - s a x ) - \ln ( 1 - s^{-1} a x ).
\]
After moving terms around, using properties of natural log, and exponentiating both sides, we arrive at the equation
\[
\Big( \sum_{n=0}^\infty (h_n \cdot e ) x^n \Big) ( 1 - s a x ) = \Big( \sum_{n=0}^\infty ( e \cdot h_n ) x^k \Big) ( 1 - s^{-1} a x )
\]
which implies the statement of the lemma.
\end{proof}

Recall that the algebra $\ca_\ch$ is equal to the Laurent polynomial ring $\ch(A)^+[a, a^{-1}]$. Under the isomorphism between $\ch(A)^+$ and the ring of symmetric functions $\Lambda_R$, the $h_n$ identify with the complete homogeneous symmetric functions. It is well-known that ordered monomials in the complete homogeneous symmetric functions form a basis of $\Lambda$, hence the sets $\{h_\lambda \cdot a^i \}_{\lambda, i}$ and $\{a^i \cdot h_\lambda \}_{\lambda, i}$ over integers $i$ and partitions $\lambda$, where $h_\lambda := h_{\lambda_1} \cdots h_{\lambda_r}$, form bases of $\ca_\ch$. The following theorem gives transition formulas between these two bases. 

\begin{theorem} \label{prop:homfly2}
The Hecke symmetrizers $h_n$ satisfy the equations
\[
h_n \cdot e = e \cdot h_n + ( 1 - s^2 ) \sum_{l=1}^{n} s^{-l} ( a^l \cdot h_{n-l} )
\]
and
\[
e \cdot h_n = h_n \cdot e + (1 - s^{-2} ) \sum_{l=1}^{n} s^l ( h_{n-l} \cdot a^l ).
\]
\end{theorem}
\begin{proof}
We will prove the first equality. The second is completely analagous. Proceed by induction. When $n=1$, the statement follows from the HOMFLY skein relation. 

We can rearrange the terms of Lemma \ref{lem:homfly1} to get
\begin{equation} \label{eq:homfly1b}
h_n \cdot e = e \cdot h_n + s^{-1} a ( h_{n-1} \cdot e ) - s a ( e \cdot h_{n-1} ).
\end{equation}
By the induction hypothesis,
\begin{align*}
h_n \cdot e & = e \cdot h_n + s^{-1} a ( h_{n-1} \cdot e ) - s a ( e \cdot h_{n-1} ) \\
& = e \cdot h_n + s^{-1} a \Big( e \cdot h_{n-1} + ( 1 - s^2 ) \sum_{j=1}^{n-1} s^{-j} a^j ( e \cdot h_{n-1-j} ) \Big) - s a ( e \cdot h_{n-1} ) \\
& = e \cdot h_n + ( s^{-1} - s ) a ( e \cdot h_{n-1} ) + ( 1 - s^2 ) \sum_{j=1}^{n-1} s^{-j-1} a^{j+1} ( e \cdot h_{n-1-j} ) \\ 
& = e \cdot h_n + ( 1 - s^2 ) s^{-1} a ( e \cdot h_{n-1} ) + ( 1 - s^2 ) \sum_{j=1}^{n-1} s^{-(j+1)} a^{j+1} ( e \cdot h_{n-(j+1)} ) \\
&= e \cdot h_n + ( 1 - s^2 ) \sum_{l=1}^{n} s^{-l} a^{l} ( e \cdot h_{n-l} )
\end{align*}
where the last equality follows from the substitution $j=l+1$. 
\end{proof}

\section{Type B/C/D Schur Functions and BMW Idempotent Closures} \label{sec:Lukac}
In \cite{Luk05}, it is shown using skein theory techniques that there is an algebra isomorphism between the ring of symmetric functions $\Lambda$ and the positive part of the skein algebra of the annulus $\ch(A)^+$. The isomorphism is defined on generators by sending the complete homogeneous symmetric functions to the annular closures of Hecke symmetrizers. Futhermore, it is shown that the image of the Schur function $s_\lambda$ is the idempotent closure $Q_\lambda$. This fact has many implications. For example, the structure constants of $\ch(A)^+$ in the basis $\{ Q_\lambda \}_\lambda$ are the Littlewood-Richardson constants. Also, the definition of $P_k$ implies that these elements truly correspond with the power sum symmetric functions. Through this isomorphism, one could transfer structure of $\Lambda$ to $\ch(A)^+$, such as the Hopf algebra structure or the plethysm structure (see \cite{MM08}). 

From a Lie-theoretic perspective, the ring of symmetric polynomials $\Lambda_N$ in $N$ variables is the ring of polynomial representations of $GL_n(\C)$. There are maps of graded rings $\Lambda_n \to \Lambda_{n-1}$ defined by specializing the $N^{\mathrm{th}}$ variable to $0$, which together form an inverse system whose inverse limit in the category of graded rings is isomorphic to $\Lambda$. This strengthens the existing relationship between the HOMFLYPT skein theory and $GL_n(\C)$ (recall that the skein relations are modeled after relations of morphisms in $U_q(\mathfrak{gl}_n)$-Mod). 

This section is an attempt to emulate Lukac's argument in the Dubrovnik case. First we will review some of the theory behind character rings of the orthogonal and symplectic groups. In particular, they are isomorphic as rings and there are ``Schur functions" indexed over partitions for the different types. Next, for any of the character rings, we can define a homomorphism to $\cd(A)$. We conjecture that the Schur functions are sent to the annular closures of the BMW idempotents $\widetilde{Q}_\lambda$, and we prove the conjecture when the length of $\lambda$ is at most $2$.

\subsection{Universal Character Rings of Orthogonal and Symplectic Type}

Here we will summarize some results from \cite{KT87}. Recall that a character of a (finite-dimensional) group representation is the post-composition of the representation with the trace function. The character of a direct sum of representations is the pointwise sum of the individual characters. Also, the character of a tensor product of representations is the pointwise product of the individual characters. For compact connected Lie groups $G$, two representations of $G$ are isomorphic only if their characters are equal. In this way, the character ring of $G$ is a decategorification of the category of finite-dimensional representations of $G$. Let $R(B_n)$, $R(C_n)$, and $R(D_n)$ denote the character rings of the Lie groups $SO_{2n+1}(\C)$, $SP_{2n}(\C)$, and $SO_{2n}(\C)$ respectively (the type $D_n$ case will require some subtle care, but we will not address these issues here). 

The irreducible representations of $SO_{2n+1}(\C)$ and $SP_{2n}(\C)$ are indexed by partitions $\lambda$ of length at most $n$, and denote their respective characters by $sb_{n, \lambda}$ and $sc_{n, \lambda}$ respectively. Let $sd_{n, \lambda}$ denote the character of the restriction of the appropriate irreducible representation of $O_{2n}(\C)$ to $SO_{2n}(\C)$. These are called the \textbf{Schur polynomials} of types $B_n$, $C_n$, and $D_n$.

Let $G$ be any of $SO_{2n+1}(\C)$, $SP_{2n}(\C)$, and $SO_{2n}(\C)$. A character of $G$ is determined by its value on a maximal torus of $G$, which may be chosen to be a set of diagonal matrices. 
\begin{align*}
&B_n: &\quad T &= \{ \textrm{diag}(t_1, \dots, t_n, 1, t_n^{-1}, \dots, t_1^{-1}) \} \\ 
&C_n: &\quad T &= \{ \textrm{diag}(t_1, \dots, t_n, t_n^{-1}, \dots, t_1^{-1}) \} \\ 
&D_n: &\quad T &= \{ \textrm{diag}(t_1, \dots, t_n, t_n^{-1}, \dots, t_1^{-1}) \}
\end{align*}
Furthermore, it can be shown that $R(G) = \Z[t_1^{\pm 1}, \dots, t_n^{\pm 1}]^{\Z_2^n \rtimes \mathfrak{S}_n} = \Z[t_1+t_1^{-1}, \dots, t_n+t_n^{-1}]^{\mathfrak{S}_n}$. This is isomorphic to the ring of symmetric polynomials $\Lambda_n = \Z[c_1, \dots, c_n]^{S_n}$ under the identification $c_i = t_i + t_i^{-1}$. 

Let $V$ be the natural representation of the group $G$. Let $h_{G,i} := \textrm{Sym}^i(V)$ and $e_{G,i} := \textrm{Alt}^i(V)$ be the symmetric and alternating powers of $V$. Also, let $h_{G,i}^\circ = h_{G,i} - h_{G_i-2}$ and $e_{G,i}^\circ = e_{G,i} - e_{G_i-2}$. The following statements are true.
\begin{enumerate}
\item The characters $h_{B_n,i}^\circ$ and $e_{B_n,i}$ are irreducible, and \[R(B_n) = \Z[h_{B_n,1}^\circ, \dots, h_{B_n,n}^\circ] = \Z[e_{B_n,1}, \dots, e_{B_n,n}]. \]
\item The characters $h_{C_n,i}$ and $e_{C_n,i}^\circ$ are irreducible, and  \[R(C_n) = \Z[h_{C_n,1}, \dots, h_{C_n,n}] = \Z[e_{C_n,1}^\circ, \dots, e_{C_n,n}^\circ]. \]
\item The characters $h_{D_n,i}^\circ$ are irreducible, $e_{D_n,i}$ is irreducible if $i \neq n$, and \[R(D_n) = \Z[h_{D_n,1}^\circ, \dots, h_{D_n,n}^\circ] = \Z[e_{D_n,1}, \dots, e_{D_n,n}]. \]
\end{enumerate}
The characters $h_{G,i}$ are equal to the Schur polynomials corresponding to a single row, and $e_{G_i}$, a single column.

The ring $\Lambda_n$ is called the ring of symmetric polynomials in $n$ indeterminates, and it may be identified with the character ring of finite-dimensional polynomial representations of the Lie group $GL_n(\C)$. The projection maps $\Lambda_n \to \Lambda_{n-1}$ form an inverse system in the category of graded rings, whose inverse limit is what we call the \textbf{ring of symmetric functions} $\Lambda$. The ring $\Lambda$ may be concretely defined as the restriction of the ring of formal power series in countable many indeterminates to those series which have bounded degree and which are invariant under permutations of the indeterminates. 

There are symmetric functions $sb_\lambda=sd_\lambda$ (they are equal) which project to the Schur polynomials $sb_{n,\lambda}$ and $sd_{n, \lambda}$ under the projections. Similarly, there are $sc_\lambda$ which project to $sc_{n, \lambda}$. These symmetric functions are called \text{Schur functions} of types B, C, and D. The sets $\{sb_\lambda \}_\lambda$ and $\{sc_\lambda \}_\lambda$ form $\Z$-bases of $\Lambda$. Furthermore, there is an involutive algebra automorphism defined by $\omega(sb_\lambda) = sc_{\lambda^t}$ where $\lambda^t$ is the transpose partition of $\lambda$. This implies that the structure constants for both bases totally coincide, i.e. if 
\[
sb_{\mu} sb_{\nu} = \sum_\lambda b_{\mu \nu}^\lambda sb_\lambda
\]
and
\[
sc_{\mu} sc_{\nu} = \sum_\lambda c_{\mu \nu}^{\lambda} sc_{\lambda}
\]
then $c_{\mu \nu}^\lambda = b_{\mu^t \nu^t}^{\lambda^t}$. Furthermore, Koike and Terada show that $b_{\mu \nu}^\lambda = b_{\mu^t \nu^t}^{\lambda^t}$ and $c_{\mu \nu}^\lambda = c_{\mu^t \nu^t}^{\lambda^t}$. Therefore, $c_{\mu \nu}^\lambda = b_{\mu \nu}^\lambda$. These structure constants are natural numbers, and may be expressed as sums of products of Littlewood-Richardson numbers (see \cite{Koi89}).

In regards to $\omega$, if we let $hb_i := sb_{(i)}$ be the Schur functions of type B corresponding to one row and $ec_i := sc_{(1^i)}$ be the Schur functions of type C corresponding to one column, then 
\begin{equation} \label{eq:omega1}
\omega(hb_i) = ec_i.
\end{equation}
Let $hb_i^\circ := hb_i - hb_{i-2}$. For a partition $\lambda=(\lambda_0, \lambda_1, \dots, \lambda_{r-1})$ of length $r$, set $hb_\lambda^\circ(i,j) := hb_{\lambda_i - i + j}^\circ + hb_{\lambda_i - i - j}^\circ$. The classical Jacobi-Trudi formula for the type A case generalizes to the type B case as
\begin{equation} \label{eq:jacobitrudi}
sb_\lambda = \frac{1}{2} \det \big( hb_\lambda(i, j) \big)_{0 \leq i, j \leq r-1}
\end{equation}
which implies a type C generalization by applying Equation \eqref{eq:omega1}. The $1/2$ is there simply to compensate for the first column for when $j=0$.

\subsection{Determinantal Calculations}

The computations in this section will be very technical. Let's fix some notation to hopefully make everything slightly easier to read. First, let
\begin{align*}
t_n &:= \tilde{h}_n \cdot e - e  \cdot \tilde{h}_n, \\
h(n, k) &:= \tilde{h}_{n+k} + \tilde{h}_{n-k}, \\
t(n, k) &:= t_{n+k} + t_{n-k} = h(n, k) \cdot e - e \cdot h(n, k).
\end{align*}
Let's also define the `skew commutator' elements
\begin{align*}
\varsigma^+ (n, k) &:= s^{-1} h(n, k) \cdot a - s a \cdot h(n, k) \\
\varsigma^- (n, k) &:=  s h(n, k) \cdot a^{-1} - s^{-1} a^{-1} \cdot h(n, k).
\end{align*}
It might be worth pointing out these simple relations: 
\begin{equation}
\begin{split}
t(n, -k) &= t(n, k), \\
h(n, -k) &= h(n, k), \\
\varsigma^\pm (n, -k) &= \varsigma^\pm (n, k), \\
t(n, 0) &= 2 t_n, \\
h(n, 0) &= 2 \tilde{h}_n. \\
\end{split}
\end{equation}
There is a way to write a commutator in terms of skew commutators. This is given by the equation
\begin{equation} \label{eq:skewcommutator3}
t(n, k) = \frac{-1}{s-s^{-1}} \big( a^{-1} \varsigma^+(n, k) - a \varsigma^-(n, k) \big)
\end{equation}
for all $k \geq 0$. Also, we may restate the identity of Lemma \ref{lem:powersumcommutator1} in terms of this new notation.
\begin{equation} \label{eq:skewcommutator2}
t(n, 1) = \frac{1}{2} \varsigma^+ (n, 0) + \frac{1}{2} \varsigma^- (n, 0)
\end{equation}

\begin{lemma} \label{lemma:skewcommutatordecomp}
For all $k \geq 0$, the following holds in $\ca$.
\[
t(n, k+1) =  \varsigma^+ (n, k) + \varsigma^- (n, k) - t(n, k-1)
\]
\end{lemma}
\begin{proof}
The base case is $2$ times Equation \ref{eq:skewcommutator2} since $t(n, 1) = t(n, -1)$. To show the general case, use equation \eqref{eq:skewcommutator2} twice.
\begin{align*}
t(n, k+1)
&= t_{n+(k+1)} + t_{n-(k+1)} \\
&= \big( ( s^{-1} \tilde{h}_{n+k} \cdot e - s e \cdot \tilde{h}_{n+k} ) a + ( s \tilde{h}_{n+k} \cdot e - s^{-1} e \cdot \tilde{h}_{n+k} ) a^{-1} - t_{n+(k-1)} \big) \\
&\qquad + \big( ( s^{-1} \tilde{h}_{n-k} \cdot e - s e \cdot \tilde{h}_{n-k} ) a + ( s \tilde{h}_{n-k} \cdot e - s^{-1} e \cdot \tilde{h}_{n-k} ) a^{-1} - t_{n-(k-1)} \big) \\
&= \big( s^{-1} h(n, k) \cdot e - s e \cdot h(n, k) \big) a + \big( s h(n, k) \cdot e - s^{-1} e \cdot h(n, k) \big) a^{-1} \\
&\qquad - t(n, k-1) \\
&= \varsigma^+ (n, k) + \varsigma^- (n, k) - t(n, k-1)
\end{align*}
\end{proof}

\begin{remark}
An alternate formulation of the first statement of the lemma above is
\begin{equation} \label{eq:skewcommutator4}
t(n, k+1) + t(n, k-1) = \varsigma^+ (n, k) + \varsigma^- (n, k)
\end{equation}
which is a generalization of Equation \eqref{eq:skewcommutator2}. Also, one may apply the identity to itself recursively to get a closed form for the $t(n,k)$ in terms of skew-commutators, but the closed form splits into cases depending on $k$. Thirdly, this lemma yields a kind of recursive formula for the skew commutators, given as
\begin{equation} \label{eq:skewcommutator5}
\varsigma^+ (n, k) + \varsigma^- (n, k) = \frac{-1}{s-s^{-1}} \Big( a^{-1} \big( \varsigma^+(n, k+1) + \varsigma^+(n, k-1) \big) - a \big( \varsigma^-(n, k+1) + \varsigma^-(n, k-1) \big) \Big)
\end{equation}
which follows from Equation \eqref{eq:skewcommutator3}
\end{remark}

\begin{lemma} \label{lemma:detgymnastics3}
The following identities hold in $\ca$. 
\leavevmode 
\begin{enumerate}
\item
\begin{equation*}
\begin{vmatrix}
h(n, k) \cdot e & \varsigma^+ (n, k) \\
h(m, k) \cdot e & \varsigma^+ (m, k)
\end{vmatrix}
= s^2
\begin{vmatrix}
e \cdot h(n, k) & \varsigma^+ (n, k) \\
e \cdot h(m, k) & \varsigma^+ (m, k)
\end{vmatrix}
\end{equation*}
\item
\begin{equation*}
\begin{vmatrix}
h(n, k) \cdot e & \varsigma^- (n, k) \\
h(m, k) \cdot e & \varsigma^- (m, k) 
\end{vmatrix}
= s^{-2}
\begin{vmatrix}
e \cdot h(n, k) & \varsigma^- (n, k) \\
e \cdot h(m, k) & \varsigma^- (m, k) 
\end{vmatrix}
\end{equation*}
\end{enumerate}
\end{lemma}

\begin{proof}
Here is the computation for the first item. The proof of the second runs completely parallel. First, expand the determinants.
\[
\big( h(n, k) \cdot e \big) \varsigma^+ (m, k) - \big( h(m, k) \cdot e \big) \varsigma^+ (n, k) = s^2 \big( e \cdot h(n, k) \big) \varsigma^+ (m, k) - s^2 \big( e \cdot h(m, k) \big) \varsigma^+ (n, k)
\]
Collect terms in the following way.
\[
\big( h(n, k) \cdot e - s^2 e \cdot h(n, k) \big) \varsigma^+ (m, k) = \big( h(m, k) \cdot e - s^2 e \cdot h(m, k) \big) \varsigma^+(n, k)
\]
The right-hand side is equal to the left-hand side with the indices $n$ and $m$ interchanged. So the equation holds if and only if the left hand side is invariant under permuting $n$ and $m$. Let $P(n,m)$ equal the left-hand side. Expand the product in $P(n, m)$ using the definition of $\varsigma^+(m, k)$
\begin{align*}
P(n, m) & = \big( h(n, k) \cdot e - s^2 e \cdot h(n, k) \big) \varsigma^+ (m, k) \\
& = \big( h(n, k) \cdot e - s^2 e \cdot h(n, k) \big) \big( s^{-1} h(n, k) \cdot a - s a \cdot h(n, k) \big) \\
& = \Big( s^{-1} \big( h(n, k) h(m, k) \big) \cdot e + s^3 e \cdot \big( h(n, k) h(m, k) \big) \\
& \qquad  - s h(n, k) \cdot e \cdot h(m, k) - s h(m, k) \cdot e \cdot h(n, k) \Big) a
\end{align*}
Use that $\cc$ is commutative to get that $P(n, m) - P(m, n) = 0$. This completes the proof. 
\end{proof}

\begin{remark}
It may be worth pointing out that the factor of $a$ in $\varsigma^+(n, k)$ and the factor of $a^{-1}$ in $\varsigma^-(n, k)$ don't affect the identities above. In other words, scaling identity (1) by $a^{-1}$ gives
\begin{equation}
\begin{vmatrix}
h(n, k) \cdot e & s^{-1} h(n, k) \cdot e - s e \cdot h(n, k) \\
h(m, k) \cdot e & s^{-1} h(m, k) \cdot e - s e \cdot h(m, k)
\end{vmatrix}
= s^2
\begin{vmatrix}
e \cdot h(n, k) & s^{-1} h(n, k) \cdot e - s e \cdot h(n, k) \\
e \cdot h(m, k) & s^{-1} h(m, k) \cdot e - s e \cdot h(m, k)
\end{vmatrix}
\end{equation}
and scaling identity (2) by $a$ gives
\begin{equation}
\begin{vmatrix}
h(n, k) \cdot e &  s h(n, k) \cdot e - s^{-1} e \cdot h(n, k) \\
h(m, k) \cdot e & s h(m, k) \cdot e - s^{-1} e \cdot h(m, k)
\end{vmatrix}
= s^{-2}
\begin{vmatrix}
e \cdot h(n, k) & s h(n, k) \cdot e - s^{-1} e \cdot h(n, k) \\
e \cdot h(m, k) & s h(m, k) \cdot e - s^{-1} e \cdot h(m, k)
\end{vmatrix}.
\end{equation}
\end{remark}

The previous two lemmas together imply the next lemma. 

\begin{lemma} \label{lemma:detgymnastics4}
For all $k \geq 0$, we have
\begin{align*}
&\quad \,
\begin{vmatrix}
h(n, k) \cdot e & t(n, k+1) \\
h(m, k) \cdot e & t(m, k+1)
\end{vmatrix} \\
&=
\begin{vmatrix}
e \cdot h(n, k) & s^2 \varsigma^+ (n, k) + s^{-2} \varsigma^- (n, k) - t(n, k-1) \\
e \cdot h(m, k) & s^2 \varsigma^+ (m, k) + s^{-2} \varsigma^- (m, k) - t(m, k-1)
\end{vmatrix}\\
& \qquad +
\sum_{l=1}^{k}
\begin{vmatrix}
e \cdot h(n, k-l) & (s^2 - 1) \varsigma^+ (n, k-l) + (s^{-2} - 1) \varsigma^- (n, k-l) \\
e \cdot h(m, k-l) & (s^2 - 1) \varsigma^+ (m, k-l) + (s^{-2} - 1) \varsigma^- (m, k-l)
\end{vmatrix}.
\end{align*} 
\end{lemma}
\begin{proof}
First, when $k=0$, the statement should be read as
\begin{align*}
\begin{vmatrix}
h(n, 0) \cdot e & t(n, 1) \\
h(m, 0) \cdot e & t(m, 1)
\end{vmatrix} 
=
\frac{1}{2}
\begin{vmatrix}
e \cdot h(n, k) & s^2 \varsigma^+ (n, 0) + s^{-2} \varsigma^- (n, 0)\\
e \cdot h(m, k) & s^2 \varsigma^+ (m, 0) + s^{-2} \varsigma^- (m, 0)
\end{vmatrix}
\end{align*}
which follows straightforwardly by applying Equation \ref{eq:skewcommutator2} to the entries $t(n,1)$ and $t(m,1)$ before applying Lemma \ref{lemma:detgymnastics3}. In general, use Lemma \ref{lemma:skewcommutatordecomp} to write 
\begin{align*}
\begin{vmatrix}
h(n, k) \cdot e & t(n, k+1) \\
h(m, k) \cdot e & t(m, k+1) \\
\end{vmatrix}
&=
\begin{vmatrix}
h(n, k) \cdot e & \varsigma^+ (n, k) \\
h(m, k) \cdot e & \varsigma^+ (m, k)
\end{vmatrix}
+
\begin{vmatrix}
h(n, k) \cdot e & \varsigma^- (n, k) \\
h(m, k) \cdot e & \varsigma^- (m, k) 
\end{vmatrix} \\
& \quad +
\begin{vmatrix}
h(n, k) \cdot e & - t(n, k-1)\\
h(m, k) \cdot e & - t(m, k-1)
\end{vmatrix}.
\end{align*}
On the first two summands, apply Lemma \ref{lemma:detgymnastics3} to switch the order of the action in the first columns at the price of the specified scalar. For the third summand, observe
\begin{align*}
\begin{vmatrix}
h(n, k) \cdot e & - t(n, k-1) \\
h(m, k) \cdot e & - t(m, k-1) 
\end{vmatrix}
&=
\begin{vmatrix}
t(n, k-1) & h(n, k) \cdot e \\
t(m, k-1) & h(m, k) \cdot e
\end{vmatrix} \\
&= 
\begin{vmatrix}
t(n, k-1) & t(n, k) \\
t(m, k-1) & t(m, k)
\end{vmatrix} 
+
\begin{vmatrix}
t(n, k-1) & e \cdot h(n, k) \\
t(m, k-1) & e \cdot h(m, k)
\end{vmatrix} \\
&=
\begin{vmatrix}
h(n, k-1) \cdot e & t(n, k) \\
h(m, k-1) \cdot e & t(m, k)
\end{vmatrix} 
- 
\begin{vmatrix}
e \cdot h(n, k-1) & t(n, k) \\
e \cdot h(m, k-1) & t(m, k)
\end{vmatrix} \\
& \quad +
\begin{vmatrix}
e \cdot h(n, k) & - t(n, k-1) \\
e \cdot h(m, k) & - t(m, k-1)
\end{vmatrix}.
\end{align*}
Thus, for all $k \geq 0$, we have
\begin{equation}\label{eq:detgymnastics5}
\begin{split}
\begin{vmatrix}
h(n, k) \cdot e & t(n, k+1) \\
h(m, k) \cdot e & t(m, k+1)
\end{vmatrix}
&=
\begin{vmatrix}
e \cdot h(n, k) & s^2 \varsigma^+ (n, k) + s^{-2} \varsigma^- (n, k) - t(n, k-1) \\
e \cdot h(m, k) & s^2 \varsigma^+ (m, k) + s^{-2} \varsigma^- (m, k) - t(m, k-1)
\end{vmatrix} \\
& \quad -
\begin{vmatrix}
e \cdot h(n, k-1) & t(n, k) \\
e \cdot h(m, k-1) & t(m, k)
\end{vmatrix} 
+
\begin{vmatrix}
h(n, k-1) \cdot e & t(n, k) \\
h(m, k-1) \cdot e & t(m, k)
\end{vmatrix}.
\end{split}
\end{equation}
Next, recursively apply Equation \ref{eq:detgymnastics5} to its own trailing term $k$ times so that the right-hand side becomes
\begin{align*}
&\sum_{l=0}^{k-1} \left(
\begin{vmatrix}
e \cdot h(n, k-l) & s^2 \varsigma^+ (n, k-l) + s^{-2} \varsigma^- (n, k-l) - t(n, k-1-l) \\
e \cdot h(m, k-l) & s^2 \varsigma^+ (m, k-l) + s^{-2} \varsigma^- (m, k-l) - t(m, k-1-l)
\end{vmatrix} \right. \\
& \left. \qquad -
\begin{vmatrix}
e \cdot h(n, k-1-l) & t(n, k-l) \\
e \cdot h(m, k-1-l) & t(m, k-l)
\end{vmatrix} \right)
+
\begin{vmatrix}
h(n, 0) \cdot e & t(n, 1) \\
h(m, 0) \cdot e & t(m, 1)
\end{vmatrix}.
\end{align*}
Through some careful reindexing work, we rewrite this as
\begin{align*}
&\begin{vmatrix}
e \cdot h(n, k) & s^2 \varsigma^+ (n, k) + s^{-2} \varsigma^- (n, k) - t(n, k-1) \\
e \cdot h(m, k) & s^2 \varsigma^+ (m, k) + s^{-2} \varsigma^- (m, k) - t(m, k-1)
\end{vmatrix}\\
&+
\sum_{l=1}^{k-1}
\begin{vmatrix}
e \cdot h(n, k-l) & s^2 \varsigma^+ (n, k-l) + s^{-2} \varsigma^- (n, k-l) - t(n, k-1-l) - t(n, k+1-l)\\
e \cdot h(m, k-l) & s^2 \varsigma^+ (m, k-l) + s^{-2} \varsigma^- (m, k-l) - t(m, k-1-l) - t(m, k+1-l)
\end{vmatrix} \\
&-
\begin{vmatrix}
e \cdot h(n, 0) & t(n, 1) \\
e \cdot h(m, 0) & t(m, 1) \\
\end{vmatrix}
+
\begin{vmatrix}
h(n, 0) \cdot e & t(n, 1) \\
h(m, 0) \cdot e & t(m, 1)
\end{vmatrix}.
\end{align*}
Apply Lemma \ref{lemma:skewcommutatordecomp} to the terms in the entries in the second column of each matrix in the big sum. Finally, observe the following applications of Lemmas \ref{lemma:detgymnastics3} and  \ref{lemma:skewcommutatordecomp}:
\begin{align*}
&-
\begin{vmatrix}
e \cdot h(n, 0) & t(n, 1) \\
e \cdot h(m, 0) & t(m, 1) \\
\end{vmatrix}
+
\begin{vmatrix}
h(n, 0) \cdot e & t(n, 1) \\
h(m, 0) \cdot e & t(m, 1)
\end{vmatrix} \\
=& 
-
\begin{vmatrix}
e \cdot h(n, 0) & t(n, 1) \\
e \cdot h(m, 0) & t(m, 1) \\
\end{vmatrix}
+
\begin{vmatrix}
e \cdot h(n, 0) & s^2 \varsigma^+ (n, 0) + s^{-2} \varsigma^- (n, 0) - t(n, -1) \\
e \cdot h(m, 0) & s^2 \varsigma^+ (m, 0) + s^{-2} \varsigma^- (m, 0) - t(m, -1)
\end{vmatrix} \\
=& 
\begin{vmatrix}
e \cdot h(n, 0) & s^2 \varsigma^+ (n, 0) + s^{-2} \varsigma^- (n, 0) - 2 t(n, 1) \\
e \cdot h(m, 0) & s^2 \varsigma^+ (m, 0) + s^{-2} \varsigma^- (m, 0) - 2 t(m, 1)
\end{vmatrix} \\
=& 
\begin{vmatrix}
e \cdot h(n, 0) & (s^2 - 1) \varsigma^+ (n, 0) + (s^{-2} - 1) \varsigma^- (n, 0) \\
e \cdot h(m, 0) & (s^2 - 1) \varsigma^+ (m, 0) + (s^{-2} - 1) \varsigma^- (m, 0)
\end{vmatrix} 
\end{align*}
where we use $t(n, 1) = t(n, -1)$. This completes the proof. 
\end{proof}

\begin{remark} \label{rmk:detgymnastics6}
Using a parallel technique as in the proof of Lemma \ref{lemma:detgymnastics4}, one can show that
\begin{align*}
& \quad \,
\begin{vmatrix}
e \cdot h(n, k) & t(n, k+1) \\
e \cdot h(m, k) & t(m, k+1)
\end{vmatrix} \\
&=
\begin{vmatrix}
h(n, k) \cdot e & s^{-2} \varsigma^+ (n, k) + s^{2} \varsigma^- (n, k) - t(n, k-1) \\
h(m, k) \cdot e & s^{-2} \varsigma^+ (m, k) + s^{2} \varsigma^- (m, k) - t(m, k-1)
\end{vmatrix}\\
& \quad +
\sum_{l=1}^{k}
\begin{vmatrix}
h(n, k-l) \cdot e & (s^{-2} - 1) \varsigma^+ (n, k-l) + (s^{2} - 1) \varsigma^- (n, k-l) \\
h(m, k-l) \cdot e & (s^{-2} - 1) \varsigma^+ (m, k-l) + (s^{2} - 1) \varsigma^- (m, k-l)
\end{vmatrix}.
\end{align*}
\end{remark}

\begin{corollary}
\begin{align*}
\begin{vmatrix}
t(n, k) & t(n, k+1) \\
t(m, k) & t(m, k+1)
\end{vmatrix}
& =
\sum_{l=0}^{k}
\begin{vmatrix}
e \cdot h(n, k-l) & (s^2 - 1) \varsigma^+ (n, k-l) + (s^{-2} - 1) \varsigma^- (n, k-l) \\
e \cdot h(m, k-l) & (s^2 - 1) \varsigma^+ (m, k-l) + (s^{-2} - 1) \varsigma^- (m, k-l)
\end{vmatrix} \\
& = \sum_{l=0}^{k}
\begin{vmatrix}
h(n, k-l) \cdot e & (1 - s^{-2}) \varsigma^+ (n, k-l) + (1 - s^{2}) \varsigma^- (n, k-l) \\
h(m, k-l) \cdot e & (1 - s^{-2}) \varsigma^+ (m, k-l) + (1 - s^{2}) \varsigma^- (m, k-l)
\end{vmatrix}.
\end{align*}
\end{corollary}
\begin{proof}
We prove the first equality. The second follows by using Remark \ref{rmk:detgymnastics6}.
\begin{align*}
&\quad \begin{vmatrix}
t(n, k) & t(n, k+1) \\
t(m, k) & t(m, k+1)
\end{vmatrix} \\
&= - 
\begin{vmatrix}
e \cdot h(n, k) & t(n, k+1) \\
e \cdot h(m, k) & t(m, k+1)
\end{vmatrix} 
+
\begin{vmatrix}
h(n, k) \cdot e & t(n, k+1) \\
h(m, k) \cdot e & t(m, k+1)
\end{vmatrix}
\\
&= - 
\begin{vmatrix}
e \cdot h(n, k) & \varsigma^+ (n, k) + \varsigma^- (n, k) - t(n, k-1) \\
e \cdot h(m, k) & \varsigma^+ (m, k) + \varsigma^- (m, k) - t(m, k-1)
\end{vmatrix} \\
&\quad +
\begin{vmatrix}
e \cdot h(n, k) & s^2 \varsigma^+ (n, k) + s^{-2} \varsigma^- (n, k) - t(n, k-1) \\
e \cdot h(m, k) & s^2 \varsigma^+ (m, k) + s^{-2} \varsigma^- (m, k) - t(m, k-1)
\end{vmatrix}\\
&\quad +
\sum_{l=1}^{k}
\begin{vmatrix}
e \cdot h(n, k-l) & (s^2 - 1) \varsigma^+ (n, k-l) + (s^{-2} - 1) \varsigma^- (n, k-l) \\
e \cdot h(m, k-l) & (s^2 - 1) \varsigma^+ (m, k-l) + (s^{-2} - 1) \varsigma^- (m, k-l)
\end{vmatrix} \\
= & \sum_{l=0}^{k}
\begin{vmatrix}
e \cdot h(n, k-l) & (s^2 - 1) \varsigma^+ (n, k-l) + (s^{-2} - 1) \varsigma^- (n, k-l) \\
e \cdot h(m, k-l) & (s^2 - 1) \varsigma^+ (m, k-l) + (s^{-2} - 1) \varsigma^- (m, k-l)
\end{vmatrix} 
\end{align*}
\end{proof}

Let $\textrm{cl}_\ca : \ca \to \cc$ be the wiring which connects the two boundary points of $\ca$ by an arc above the annulus. 
\begin{figure}[H]
\centering
$\pic[7]{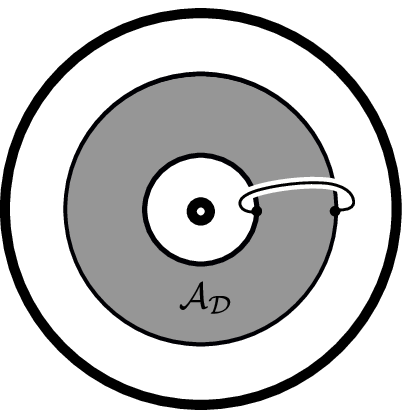}$
\caption{The wiring diagram which defines $\textrm{cl}_\ca : \ca \to \cc$.}
\end{figure}

This map is not unital nor multiplicative, but it is a right $\cc$-module homomorphism so that $\cl_\ca(e \cdot x) = \delta_\cd x$. Furthermore, the closure $\cl_\ca(x \cdot e)$ is the element $\phi(x)$ which is the meridian map $\phi$ applied to $x$. The following identity is a generalization of Lemma 8.3 in Lukac.

\begin{lemma} \label{lemma:tclosure}
Let 
\[
\alpha_n := \{ n \} \left( v^{-1} s^{n - 1} - v s^{1 - n} \right).
\] 
Then, 
\[
\textrm{cl}_\ca \left( \frac{1}{2} t(n,0) \right) = \textrm{cl}_\ca (t_n) = \alpha_n \tilde{h}_n
\]
and hence 
\[
\textrm{cl}_\ca \big( t(n, k) \big) = \alpha_{n+k} \tilde{h}_{n+k} + \alpha_{n-k} \tilde{h}_{n-k}.
\]
\end{lemma}
\begin{proof}
Recall Equation \eqref{eq:annfund}, which we restate here using different notation:
\[
t_n = \{ n \} a^{-1}W^*_{n - 1} - \{ n \} aW_{n - 1}.
\]
Apply the closure to both sides. Apply the framing relation to the diagrams on the right to pick up the framing parameter terms. Use the property that the BMW idempotents absorb crossings to a factor of $s$ for positive crossings and $s^{-1}$ for negative crossings. This completes the proof. 
\end{proof}

\begin{lemma} \label{lemma:skewclosures}
Define the constants
\begin{align*}
\omega_{n}^+ &:= - \{n\} s^{1-n} v, \\
\upsilon_{n}^+ &:= s^{n} [n+1] (\bar{\beta}_{n+1} - \beta_{n+1} ) \big( \delta_\cd + s^{n-1} [n] ( s v^{-1} + \beta_{n} ) \big), \\
\omega_{n}^- &:= \{n\} s^{n-1} v^{-1}, \\
\upsilon_{n}^- &:= -s^{-n} [n+1] ( \bar{\beta}_{n+1} - \beta_{n+1} ) \big( \delta_\cd + s^{n-1} [n] ( s v^{-1} + \beta_{n} ) \big).
\end{align*}
Then the following identities hold in $\cc$ for all $n \geq 0$ 
\begin{align*}
\textrm{cl}_\ca \left( \frac{1}{2} \varsigma^+(n,0) \right) &= \quad \textrm{cl}_\ca \big( s^{-1} \tilde{h}_n \cdot a - s a \cdot \tilde{h}_n \big) \quad = \omega_{n+1}^+ \tilde{h}_{n+1} + \upsilon_{n-1}^+ \tilde{h}_{n-1}, \\
\textrm{cl}_\ca \left( \frac{1}{2} \varsigma^-(n,0) \right) &=\textrm{cl}_\ca \big( s \tilde{h}_n \cdot a^{-1} - s^{-1} a^{-1} \cdot \tilde{h}_n \big) = \omega_{n+1}^- \tilde{h}_{n+1} + \upsilon_{n-1}^- \tilde{h}_{n-1}
\end{align*}
which implies that
\begin{align*}
\textrm{cl}_\ca \big( \varsigma^+ (n, k) \big) &= \omega_{n+k+1}^+ \tilde{h}_{n+k+1} + \upsilon_{n+k-1}^+ \tilde{h}_{n+k-1} + \omega_{n-k+1}^+ \tilde{h}_{n-k+1} + \upsilon_{n-k-1}^+ \tilde{h}_{n-k-1}, \\
\textrm{cl}_\ca \big( \varsigma^- (n, k) \big) &= \omega_{n+k+1}^- \tilde{h}_{n+k+1} + \upsilon_{n+k-1}^- \tilde{h}_{n+k-1} + \omega_{n-k+1}^- \tilde{h}_{n-k+1} + \upsilon_{n-k-1}^- \tilde{h}_{n-k-1}.
\end{align*}
\end{lemma}
\begin{proof}
Apply Equations \eqref{eq:r3} and \eqref{eq:r1} :
\begin{align*}
& \quad \, ( s^{-1} \tilde{h}_n \cdot e - s e \cdot \tilde{h}_n ) a \\
&=  \Big( s^{-1} \big( [n+1] W_n - [n] s a W_{n-1} - [n] s \bar{\beta}_n a^{-1} W^*_{n-1} \big) \\
&\quad \, - s \big( [n+1] W_n - [n] s^{-1} a W_{n-1} - [n] s^{-1} \beta_n a^{-1} W^*_{n-1} \big) \Big) a \\
&= - \{n+1\} a W_n - ( \bar{\beta}_n - \beta_n ) [n] W^*_{n-1}.
\end{align*}
The identities below hold in $\cc$. Use Lemma 21 from \cite{She16} to prove the second two.
\begin{align*}
\textrm{cl}_\ca ( a W_n ) &= s^{-n} v \tilde{h}_{n+1} \\
\textrm{cl}_\ca ( a^{-1} W^*_n ) &= s^n v^{-1} \tilde{h}_{n+1} \\
\textrm{cl}_\ca ( W_n ) &= -s^{-n} \big( \delta_\cd + s^{n-1} [n] ( s v^{-1} + \beta_n ) \big) \tilde{h}_n \\
&= [n+1]^{-1} \big( \delta_\cd + [n]s^{-1} ( s^{-(n-1)} v - \beta_n s^{n-1} v^{-1} ) \big) \tilde{h}_n \\
\textrm{cl}_\ca ( W^*_n ) &= s^n \big( \delta_\cd + s^{n-1} [n] ( s v^{-1} + \beta_n ) \big) \tilde{h}_n
\end{align*}
Then applying $\textrm{cl}_\ca$ to the above gives
\begin{align*} 
& \textrm{cl}_\ca \big( s^{-1} \tilde{h}_n \cdot a - s a \cdot \tilde{h}_n \big) \\
=& - \{n+1\} \textrm{cl}_\ca ( a W_n ) - ( \bar{\beta}_n - \beta_n ) [n] \textrm{cl}_\ca ( W^*_{n-1} ) \\
= & - \{n+1\} s^{-n} v \tilde{h}_{n+1} - ( \bar{\beta}_n - \beta_n ) [n] s^{n-1} \big( \delta_\cd + s^{n-2} [n-1] ( s v^{-1} \beta_{n-1} ) \big) \tilde{h}_{n-1} \\
= & \omega_{n+1}^+ \tilde{h}_{n+1} + \upsilon_{n-1}^+ \tilde{h}_{n-1}.
\end{align*}
The proof of the other statement is similar. 
\end{proof}

\begin{remark}
As a quick corollary of this lemma, apply $\textrm{cl}_\ca$ to Equation \eqref{eq:skewcommutator4} with $k=0$, so
\[
\textrm{cl}_\ca \big( 2 t(n, 1) \big) = \textrm{cl}_\ca \big( \varsigma^+ (n, 0) + \varsigma^- (n, 0) \big)
\]
and apply the previous two Lemmas and divide by $2$ to get
\[
\alpha_{n+1} \tilde{h}_{n+1} + \alpha_{n-1} \tilde{h}_{n-1} = (\omega^+_{n+1} + \omega^-_{n+1}) \tilde{h}_{n+1} + (\upsilon^+_{n-1} + \upsilon^-_{n-1}) \tilde{h}_{n-1}.
\]
Since the $\tilde{h}_n$ are linearly independent, it's true that 
\begin{equation}
\alpha_n = \omega_n^+ + \omega_n^- = \upsilon_n^+ + \upsilon_n^-
\end{equation}
for any $n$. 
\end{remark}

Consider the algebra homomorphism $\Lambda \to \cd(A)$ defined by the assignment $hb_i \mapsto \tilde{h}_i$. This map is well defined since $\Lambda$ is generated freely as a commutative ring by the set $\{hb_i\}_{i}$. Let $S_\lambda$ be the image of $sb_\lambda$ under this assignment. Now we are ready to state the main conjecture of this section.

\begin{conjecture} \label{conj:schureigenvalues}
Let $\phi: \cc \to \cc$ be the meridian map.  Then
\begin{equation}
\phi ( S_\lambda ) = c_\lambda S_\lambda
\end{equation}
where
\begin{equation}
c_\lambda = \delta_\cd + ( s - s^{-1} ) \Big( v^{-1} \sum_{x \in \lambda} s^{2 \textrm{cn}(x)} - v \sum_{x \in \lambda} s^{-2 \textrm{cn}(x)} \Big)
\end{equation}
and where $\textrm{cn}(x)$ is the content of the cell $x$ of the Young diagram of $\lambda$ in the $(i,j)^{\textrm{th}}$ position, defined by $\textrm{cn}(x) := j-i$.
\end{conjecture}

Conjecture \ref{conj:schureigenvalues} states that $S_\lambda$ and $\widetilde{Q}_\lambda$ are in the same eigenspace with respect to $\phi$. By \cite{LZ02}, each eigenspace is $1$-dimensional, and so $S_\lambda$ is a scalar multiple of $\widetilde{Q}_\lambda$. What's left is to show that this scalar is $1$. Let's state this as a corollary.

\begin{corollary}
For any partition $\lambda$, we have 
\[
S_\lambda = \widetilde{Q}_\lambda.
\]
\end{corollary}
\begin{proof}
This proof is a generalization of the arguement given in \cite{Luk05}. Firstly, the statement is true when $\lambda$ is either the empty partition or the unique partition $\square$ of $1$, simply by definition of $S_\lambda$. Koike and Terada provide the structure constants with respect to the $S_\lambda$ \cite{KT87}. This multiplication rule has as a special case the formula
\begin{equation}
S_\mu S_\square = \sum_{\mu'} S_{\mu'}
\end{equation}
where the sum runs over all $\mu'$ obtained by either adding or removing one square from $\mu$. The annular closure applied to the branching formula from \cite{BB01} shows the same formula
\begin{equation}
\widetilde{Q}_\mu \widetilde{Q}_\square = \sum_{\mu'} \widetilde{Q}_{\mu'}.
\end{equation}
Setting $\mu=\square$ implies the equations
\begin{align}
S_\square^n &= \sum_\lambda d_\lambda S_\lambda \label{eq:lukeq1} \\
\widetilde{Q}_\square^n &= \sum_\lambda d_\lambda \widetilde{Q}_\lambda . \label{eq:lukeq2}
\end{align}
where the scalars $d_\lambda$ are positive integers and the sums range over partitions $\lambda$ such that there exists an up-down tableau of length $n$ and shape $\lambda$. As stated above, Conjecture \ref{conj:schureigenvalues} states that $S_\lambda$ and $\widetilde{Q}_\lambda$ are contained in the same eigenspace of $\phi$. \cite{LZ02} shows that each eigenspace is $1$-dimensional, so
\begin{equation} \label{eq:lukeq3}
S_\lambda = k_\lambda \widetilde{Q}_\lambda
\end{equation}
for some scalar $k_\lambda$. Now, $S_\square = \widetilde{Q}_\square$ by definition of $S_\square$, so we may set Equations \eqref{eq:lukeq1} and \eqref{eq:lukeq2} equal to eachother. After expressing $S_\lambda$ in terms of $\widetilde{Q}_\lambda$ using Equation \eqref{eq:lukeq3}, we may write
\begin{equation}
\sum_\lambda d_\lambda (1-k_\lambda) \widetilde{Q}_\lambda = 0.
\end{equation}
Since the $\widetilde{Q}_\lambda$ are linearly independent \cite{LZ02}, it must be true that $k_\lambda=1$ for all $\lambda$ appearing in the sum. For each partition $\lambda$ there exist some up-down tableau of shape $\lambda$, so each $k_\lambda$ appears in the sum for some $n$. Therefore, $k_\lambda=1$ for all partitions $\lambda$. This completes the proof. 
\end{proof}

Using the machinery built in this section, we can prove a special case of Conjecture \ref{conj:schureigenvalues}.

\begin{proposition}
The conjecture holds when $\lambda$ has length at most $2$.
\end{proposition}
\begin{proof}
The statement is trivial if $\lambda$ has length $1$ by definition of the map. Assume $\lambda = (a, b)$. Apply Equation \eqref{eq:jacobitrudi}.
\[
S_\lambda = 
\begin{vmatrix}
\frac{1}{2}\tilde{h}_\lambda(0,0) & \tilde{h}_\lambda(0,1) \\
\frac{1}{2}\tilde{h}_\lambda(1,0) & \tilde{h}_\lambda(1,1)
\end{vmatrix}
\]
Use multilinearity of the determinant as we did above to decompose the commutator $S_\lambda \cdot e - e \cdot S_\lambda$.
\begin{equation*}
S_\lambda \cdot e - e \cdot S_\lambda = 
\begin{vmatrix}
\frac{1}{2} t_\lambda(0,0) & e \cdot \tilde{h}_\lambda(0,1) \\
\frac{1}{2} t_\lambda(1,0) & e \cdot \tilde{h}_\lambda(1,1)
\end{vmatrix}
+
\begin{vmatrix}
\frac{1}{2}\tilde{h}_\lambda(0,0) \cdot e & t_\lambda(0,1) \\
\frac{1}{2}\tilde{h}_\lambda(1,0) \cdot e & t_\lambda(1,1)
\end{vmatrix} \\
\end{equation*}
The entries in the second column of the second determinant need to be split using Equation \eqref{eq:skewcommutator2} into a form which is compatible with Lemma \ref{lemma:detgymnastics3}. The right-hand side becomes
\[
\begin{vmatrix}
\frac{1}{2} t_\lambda(0,0) & e \cdot \tilde{h}_\lambda(0,1) \\
\frac{1}{2} t_\lambda(1,0) & e \cdot \tilde{h}_\lambda(1,1)
\end{vmatrix}
+
\begin{vmatrix}
\frac{1}{2}\tilde{h}_\lambda(0,0) \cdot e & \frac{1}{2} \varsigma_\lambda^+(0,0) \\
\frac{1}{2}\tilde{h}_\lambda(1,0) \cdot e & \frac{1}{2} \varsigma_\lambda^+(1,0)
\end{vmatrix}
+
\begin{vmatrix}
\frac{1}{2}\tilde{h}_\lambda(0,0) \cdot e & \frac{1}{2} \varsigma_\lambda^-(0,0) \\
\frac{1}{2}\tilde{h}_\lambda(1,0) \cdot e & \frac{1}{2} \varsigma_\lambda^-(1,0)
\end{vmatrix}.
\]
Apply Lemma \ref{lemma:detgymnastics3} to the second two terms:
\[
\begin{vmatrix}
\frac{1}{2} t_\lambda(0,0) & e \cdot \tilde{h}_\lambda(0,1) \\
\frac{1}{2} t_\lambda(1,0) & e \cdot \tilde{h}_\lambda(1,1)
\end{vmatrix}
+ s^2
\begin{vmatrix}
e \cdot \frac{1}{2}\tilde{h}_\lambda(0,0) & \frac{1}{2} \varsigma_\lambda^+(0,0) \\
e \cdot \frac{1}{2}\tilde{h}_\lambda(1,0) & \frac{1}{2} \varsigma_\lambda^+(1,0)
\end{vmatrix}
+ s^{-2}
\begin{vmatrix}
e \cdot \frac{1}{2}\tilde{h}_\lambda(0,0) & \frac{1}{2} \varsigma_\lambda^-(0,0) \\
e \cdot \frac{1}{2}\tilde{h}_\lambda(1,0) & \frac{1}{2} \varsigma_\lambda^-(1,0)
\end{vmatrix}.
\]
Now apply $\cl_\ca$ and use the property that it is a right $\cc$-module homomorphism to remove $e$ from the $\tilde{h}_i$ terms. Apply Lemmas \ref{lemma:tclosure} and \ref{lemma:skewclosures} for the second equality.
\begin{align*}
&\quad \textrm{cl}_\ca ( S_\lambda \cdot e - e \cdot S_\lambda ) \\
&=
\begin{vmatrix}
\cl_\ca \big( \frac{1}{2} t_\lambda(0,0) \big) & \tilde{h}_\lambda(0,1) \\
\cl_\ca \big( \frac{1}{2} t_\lambda(1,0) \big) & \tilde{h}_\lambda(1,1)
\end{vmatrix}
+ s^2
\begin{vmatrix}
\frac{1}{2}\tilde{h}_\lambda(0,0) & \cl_\ca \big( \frac{1}{2} \varsigma_\lambda^+(0,0) \big) \\
\frac{1}{2}\tilde{h}_\lambda(1,0) & \cl_\ca \big( \frac{1}{2} \varsigma_\lambda^+(1,0) \big)
\end{vmatrix} \\
& \quad + s^{-2}
\begin{vmatrix}
\frac{1}{2}\tilde{h}_\lambda(0,0) & \cl_\ca \big( \frac{1}{2} \varsigma_\lambda^-(0,0) \big) \\
\frac{1}{2}\tilde{h}_\lambda(1,0) & \cl_\ca \big( \frac{1}{2} \varsigma_\lambda^-(1,0) \big)
\end{vmatrix} \\
&=
\begin{vmatrix}
\cl_\ca \big( \frac{1}{2} t(a,0) \big) & \tilde{h}(a,1) \\
\cl_\ca \big( \frac{1}{2} t(b-1,0) \big) & \tilde{h}(b-1,1)
\end{vmatrix}
+ s^2
\begin{vmatrix}
\frac{1}{2}\tilde{h}(a,0) & \cl_\ca \big( \frac{1}{2} \varsigma^+(a,0) \big) \\
\frac{1}{2}\tilde{h}(b-1,0) & \cl_\ca \big( \frac{1}{2} \varsigma^+(b-1,0) \big)
\end{vmatrix} \\
& \quad + s^{-2}
\begin{vmatrix}
\frac{1}{2}\tilde{h}(a,0) & \cl_\ca \big( \frac{1}{2} \varsigma^-(a,0) \big) \\
\frac{1}{2}\tilde{h}(b-1,0) & \cl_\ca \big( \frac{1}{2} \varsigma^-(b-1,0) \big)
\end{vmatrix} \\
&=
\begin{vmatrix}
\alpha_a \tilde{h}_a & \tilde{h}_{a+1} + \tilde{h}_{a-1} \\
\alpha_{b-1} \tilde{h}_{b-1} & \tilde{h}_{b} + \tilde{h}_{b-2}
\end{vmatrix}
+ s^2
\begin{vmatrix}
\tilde{h}_a & \omega_{a+1}^+ \tilde{h}_{a+1} + \upsilon^+_{a-1} \tilde{h}_{a-1} \\
\tilde{h}_{b-1} & \omega_{b}^+ \tilde{h}_{b} + \upsilon^+_{b-2} \tilde{h}_{b-2}
\end{vmatrix} \\
& \quad + s^{-2}
\begin{vmatrix}
\tilde{h}_a & \omega_{a+1}^- \tilde{h}_{a+1} + \upsilon^-_{a-1} \tilde{h}_{a-1} \\
\tilde{h}_{b-1} & \omega_{b}^- \tilde{h}_{b} + \upsilon^-_{b-2} \tilde{h}_{b-2}
\end{vmatrix} \\
&= \big( \alpha_a \tilde{h}_a (\tilde{h}_{b} + \tilde{h}_{b-2}) - \alpha_{b-1} \tilde{h}_{b-1} (\tilde{h}_{a+1} + \tilde{h}_{a-1}) \big) \\
& \quad + s^2 \big( \tilde{h}_a (\omega_{b}^+ \tilde{h}_{b} + \upsilon^+_{b-2} \tilde{h}_{b-2}) - \tilde{h}_{b-1}( \omega_{a+1}^+\tilde{h}_{a+1} + \upsilon^+_{a-1} \tilde{h}_{a-1}) \big) \\
& \quad + s^{-2} \big( \tilde{h}_a (\omega_{b}^- \tilde{h}_{b} + \upsilon^-_{b-2} \tilde{h}_{b-2}) - \tilde{h}_{b-1} (\omega_{a+1}^- \tilde{h}_{a+1} + \upsilon^-_{a-1} \tilde{h}_{a-1}) \big) \\
&= (\alpha_a + s^2 \omega_b^+ + s^{-2} \omega_b^-) \tilde{h}_a \tilde{h}_b + (\alpha_a + s^2 \upsilon_{b-2}^+ + s^{-2} \upsilon_{b-2}^-) \tilde{h}_a \tilde{h}_{b-2} \\
&\quad - (\alpha_{b-1} + s^{2} \omega_{a+1}^+ + s^{-2} \omega_{a+1}^-) \tilde{h}_{a+1} \tilde{h}_{b-1} - (\alpha_{b-1} + s^2 \upsilon_{a-1}^+ + s^{-2} \upsilon_{a-1}^-) \tilde{h}_{a-1} \tilde{h}_{b-1} 
\end{align*}
The conjecture claims that all of these coefficients are equal to the appropriate eigenvalue minus the value of the unknot $c_\lambda - \delta_\cd$. At first glance, it's not at all obvious that these coefficients are even equal to each other. Nevertheless, we will show it is true by looking at each coefficient one by one, going in the order in which they are written. Firstly, observe:
\begin{align*}
\alpha_a + s^2 \omega_b^+ + s^{-2}\omega_b^- &= \{a\} (v^{-1}s^{a-1} - v s^{1-a} ) - \{b\}vs^{3-b} + \{b\}v^{-1}s^{b-3} \\
&= (s^a - s^{-a}) (v^{-1}s^{a-1} - v s^{1-a} ) -  (s^b - s^{-b})(vs^{3-b} - v^{-1}s^{b-3}) \\
&= (v^{-1} s^{2a-1} - vs - v^{-1}s^{-1} + vs^{1-2a}) - \\
& \quad\,\, (vs^3 - v^{-1}s^{2b-3} - vs^{3-2b} + v^{-1}s^{-3}) \\
&= v^{-1} \big( (s^{2a-1} - s^{-1}) + (s^{2b-3} - s^{-3}) \big) - v \big( ( s - s^{1-2a}) + (s^3 - s^{3-2b}) \big) \\
&= v^{-1} \left( (s-s^{-1}) \sum_{i=0}^{a-1} s^{2a - 2 - 2i} + (s-s^{-1}) \sum_{i=0}^{b-1} s^{2b-4-2i} \right) \\
&\quad v \left( (s-s^{-1}) \sum_{i=0}^{a-1} s^{-2i} + (s-s^{-1}) \sum_{i=0}^{b-1} s^{2-2i} \right) \\
&= (s-s^{-1}) \left( v^{-1} \sum_{x \in \lambda} s^{2 \cn(x)} - v \sum_{x \in \lambda} s^{-2\cn(x)} \right) \\
&= c_\lambda - \delta_\cd.
\end{align*}
Therefore, the first coefficient is what we claimed. Now we can simply show the other coefficients are equal to the first. Let's show the second coefficient is equal to the first:
\begin{align*}
&\quad-\alpha_{b-1} - s^2 \omega_{a+1}^+ - s^{-2} \omega_{a+1}^- \\
&= -\{b-1\}(v^{-1}s^{b-2}-vs^{2-b}) + \{a+1\}vs^{2-a} - \{a+1\}v^{-1}s^{a-2} \\
&= (s^{1-b} - s^{b-1} ) (v^{-1}s^{b-2} - vs^{2-b}) + (s^{a+1}-s^{a+1})(vs^{2-a} - v^{-1}s^{a-2}) \\
&= (v^{-1}s^{-1} - v^{-1}s^{2b-3} - vs^{3-2b} + vs) - (vs^3 - vs^{1-2a} - v^{-1}s^{2a-1} + v^{-1}s) \\
&= v^{-1} \big( (s^{2a-1} - s^{-1}) + (s^{2b-3} - s^{-3}) \big) - v \big( ( s - s^{1-2a}) + (s^3 - s^{3-2b}) \big) \\
&= c_\lambda - \delta_\cd
\end{align*}

The next two are slightly different. We will avoid the complexity of $\upsilon_i^\pm$ by using the relation $\upsilon_i^+ + \upsilon_i^- = \alpha_i$. Observe:
\begin{align*}
& \quad \alpha_a + s^2 \upsilon_{b-2}^+ + s^{-2}\upsilon_{b-2}^- \\
&= \{a\}(v^{-1}s^{a-1} - vs^{1-a}) + (s^{b} - s^{-b}) [b-1] (\bar{\beta}_{b-1} - \beta_{b-1}) \big(\delta_\cd + s^{b-3}[b-2](sv^{-1}+\beta_{b-2})\big) \\
&= \{a\}(v^{-1}s^{a-1}-vs^{1-a}) \\
&\quad\,\, + (s^{b} - s^{-b})(s^{b-2}-s^{2-b})^{-1}(s^{b-2}-s^{2-b})[b-1] (\bar{\beta}_{b-1} - \beta_{b-1}) \big(\delta_\cd + s^{b-3}[b-2](sv^{-1}+\beta_{b-2})\big) \\
&= \{a\}(v^{-1}s^{a-1}-vs^{1-a}) + (s^{b} - s^{-b})(s^{b-2}-s^{2-b})^{-1}(\upsilon_{b-2}^+ + \upsilon_{b-2}^-) \\
&= \{a\}(v^{-1}s^{a-1}-vs^{1-a}) + (s^{b} - s^{-b})(s^{b-2}-s^{2-b})^{-1}\alpha_{b-2} \\
&= \{a\}(v^{-1}s^{a-1}-vs^{1-a}) - (s^{b} - s^{-b})(s^{b-2}-s^{2-b})^{-1}(s^{b-2}-s^{2-b})(vs^{3-b}-v^{-1}s^{b-3}) \\
&= (s^a - s^{-a})(v^{-1}s^{a-1}-vs^{1-a}) - (s^{b} - s^{-b})(vs^{3-b}-v^{-1}s^{b-3}) \\
&= c_\lambda - \delta_\cd
\end{align*}
and a similar computation holds for the final coefficient. This completes the proof.
\end{proof}

\section{Meridians and Jucys-Murphy Elements of $BMW_n$}

Consider the elements $M_{n, i} = \sigma_{i-1} \cdots \sigma_1 \sigma_1 \cdots \sigma_{i-1} \in BMW_n$ for $2 \leq i \leq n$ and set $M_{n,1} = \id_n$. The set of $M_{n, i}$ in $BMW_n$ are called the \textbf{BMW Jucys-Murphy elements}, which are invertible and which generate a commutative subalgebra (see \cite{IMO14} for more details). 
\begin{figure}[h]
\centering
$M_{n,i} = \pic[5]{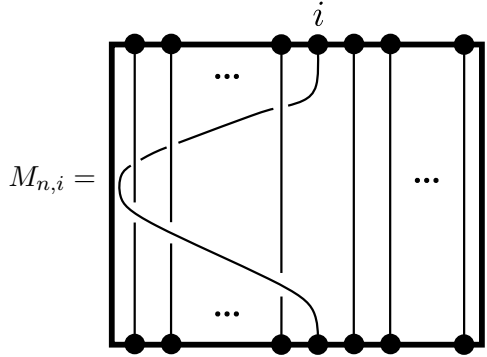}$
\caption{The Jucys-Murphy element wraps the $i^\textrm{th}$ strand around the first $i-1$ strands.}
\end{figure}

There is a homomorphism $\psi_n: \cd(A) \to BMW_n$ defined so that $\psi_n(x)$ is the element $x$ threaded by a meridian around the identity of $BMW_n$.
\begin{figure}
\centering
$\psi_n(x) = \pic[4]{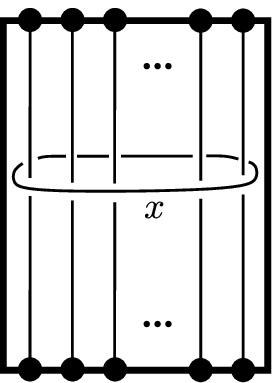}$
\caption{The image of $x$ under the homomorphism $\psi_n$.}
\end{figure}
The image of $\psi_n$ clearly lies in the center of $BMW_n$. In the next proposition, we give a simple application which expresses $\psi_n(\widetilde{P}_k)$ as a linear combination of the BMW Jucys-Murphy elements, their inverses, and the identity. One way to interpret the proposition is that we express $(s^k - s^{-k})^{-1}\big(\psi_n(\widetilde{P}_k) - \langle \widetilde{P}_k \rangle \id_n \big)$ as a power sum symmetric polynomial in the elements $\{ v^{\pm 1} M_{n,i}^{\pm 1} \}_i$. This is a Dubrovnik generalization of the HOMFLYPT counterpart found in \cite{Mor02b}.

\begin{proposition} \label{prop:murphy}
For any $k \geq 1$, and any $n \geq 1$, we may represent $\psi_n(\widetilde{P}_k)$ as
\begin{equation}
\psi_n(\widetilde{P}_k) = \langle \widetilde{P}_k \rangle \id_n + (s^k - s^{-k}) \sum_{i=2}^n v^{-k} M_{n, i}^k - v^k M_{n, i}^{-k}
\end{equation}
where $\langle \widetilde{P}_k \rangle$ is the evaluation of $\widetilde{P}_k$ in a $3$-ball. 
\end{proposition}
\begin{proof}
Induct on $n$. For $n=1$, simply apply Theorem \ref{thm:powersumcommutator} (via wirings coming from the skein functors described in Section \ref{sec:foundations}) and apply framing relations to get the desired expression (imagine removing the first $n-1$ strands from the diagrams which follow). Assume the induction hypothesis, that the statement is true for all $n<i$. Apply Theorem \ref{thm:powersumcommutator} to move $\widetilde{P}_k$ past the $n^{\mathrm{th}}$ strand and apply framing relations.
\begin{align*}
& \quad\,\, \pic[2.4]{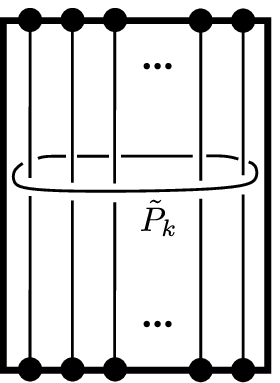} \\
&= \pic[2.4]{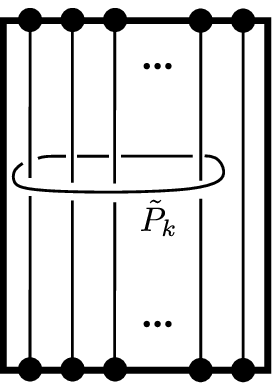} + (s^k - s^{-k}) \left( \pic[2.4]{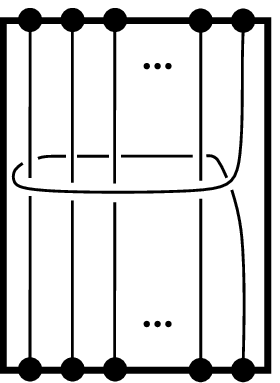}\right)^k - (s^{k}-s^{-k}) \left(\pic[2.4]{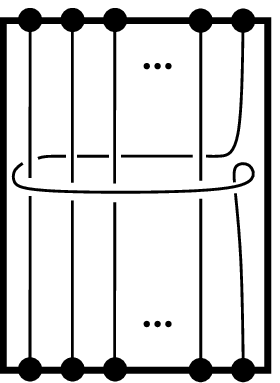}\right)^k \\
&= \pic[2.4]{meridianpkpass} + (s^k - s^{-k}) v^{-k} \left(\pic[2.4]{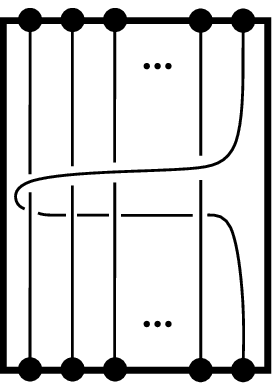}\right)^k - (s^{k}-s^{-k}) v^k \left(\pic[2.4]{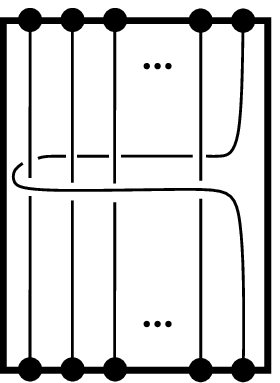}\right)^k \\
&= \psi_{n-1}(\widetilde{P}_k) \otimes \id_1  + (s^k - s^{-k}) (v^{-k}M_{n,n}^k - v^k M_{n,n}^{-k}) \\
&= \left( \langle \widetilde{P}_k \rangle \id_{n-1} + (s^k - s^{-k}) \sum_{i=2}^{n-1} v^{-k} M_{n-1,i}^k - v^k M_{n-1,i}^{-k} \right) \otimes \id_1 \\
&\quad\,\, + (s^k - s^{-k}) (v^{-k}M_{n,n}^k - v^k M_{n,n}^{-k}) \\
&= \langle \widetilde{P}_k \rangle \id_{n-1} \otimes \id_1 + (s^k - s^{-k}) \sum_{i=2}^{n-1} \left( v^{-k} M_{n-1,i}^k \otimes \id_1 - v^k M_{n-1,i}^{-k}\otimes \id_1 \right) \\
&\quad\,\, + (s^k - s^{-k}) (v^{-k}M_{n,n}^k - v^{-k} M_{n,n}^{-k}) \\
&= \langle \widetilde{P}_k \id_n + (s^k - s^{-k}) \sum_{i=2}^{n-1} \left( v^{-k} M_{n,i}^k - v^k M_{n,i}^{-k} \right) + (s^k - s^{-k}) (v^{-k}M_{n,n}^k - v^k M_{n,n}^{-k}) \\
&=  \langle \widetilde{P}_k \rangle \id_n + (s^k - s^{-k}) \sum_{i=2}^n v^{-k} M_{n, i}^k - v^k M_{n, i}^{-k}
\end{align*}
\end{proof}

For a minimal idempotent $\tilde{y}_\lambda \in BMW_n$, let $\widetilde{P}_k \cdot \tilde{y}_\lambda := \psi_n(\widetilde{P}_k) \tilde{y}_\lambda$, which is essentially a meridian threaded by $\widetilde{P}_k$ wrapped around $\tilde{y}_\lambda$. The next proposition shows that $\tilde{y}_\lambda$ is an eigenvector with respect to the assignment $\tilde{y}_\lambda \mapsto \widetilde{P}_k \cdot \tilde{y}_\lambda$. Note that all the eigenvalues are distinct from each other. The annular closure of these eigenvalue equations generalize the result found in \cite{LZ02}, which they show for $k=1$. 
\begin{proposition} \label{prop:zlgeneralization}
In $BMW_n$, we have that 
\[
\widetilde{P}_k \cdot \tilde{y}_\lambda = \left( \langle \widetilde{P}_k \rangle + (s^k - s^{-k})\sum_{\square \in \lambda} \left(  v^{-k} s^{2k \cn(\square)} - v^k s^{-2k \cn(\square)} \right) \right) \tilde{y}_\lambda
\]
where $\cn(\square)$ is the content of the cell $\square$ in the Young diagram of $\lambda$. 
\end{proposition}
\begin{proof}
We will prove the formula by inducting on the size of $\lambda$. For $|\lambda| = 1$, we have that $\tilde{y}_\lambda = \id_1$, so simply apply Proposition \ref{prop:murphy} and verify that the fomulas coincide. Now proceed assuming the induction hypothesis. Let $\mu$ be the partition obtained by removing a choice of cell $\square$ from the Young diagram of $\lambda$. First apply the branching formula (see Corollary 5.3 in \cite{BB01}). Then apply Theorem \ref{thm:powersumcommutator} to move $\widetilde{P}_k$ past the $n^{\mathrm{th}}$ strand and apply framing relations.
\begin{align*}
&\quad\,\, \pic[2.4]{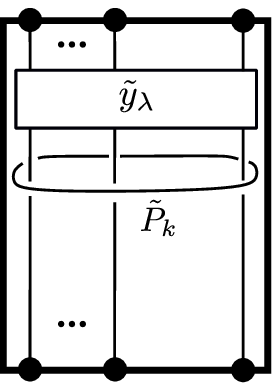} \\
&= \pic[2.4]{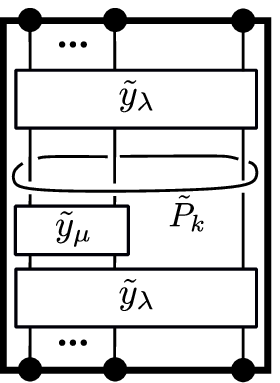} \\
&= \pic[2.4]{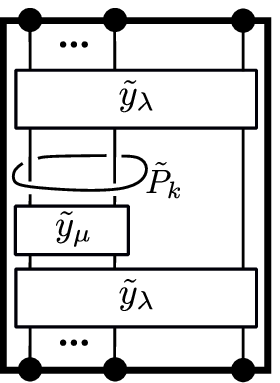} + (s^k - s^{-k}) \pic[2.4]{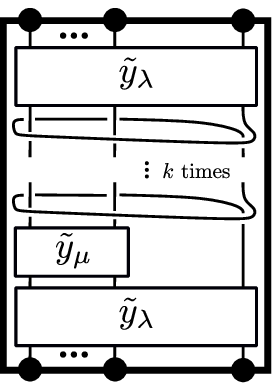} - (s^k - s^{-k}) \pic[2.4]{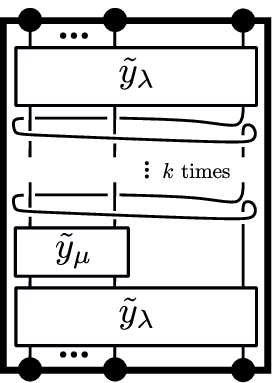} \\
&= \pic[2.4]{ylambdameridianpkbranchpass.eps} + (s^k - s^{-k}) v^{-k} \pic[2.4]{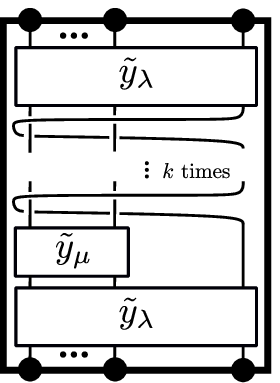} - (s^k - s^{-k}) v^k \pic[2.4]{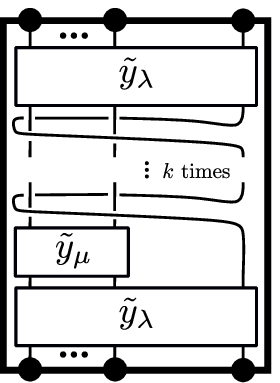}
\end{align*}
The first diagram is equal to $ \left( \langle \widetilde{P}_k \rangle+ (s^k - s^{-k})\sum_{\square \in \mu} \left(  v^{-k} s^{2k \cn(\square)} - v^k s^{-2k \cn(\square)} \right) \right) \tilde{y}_\lambda$ by the induction hypothesis. The second diagram is equal to $s^{2k \cn(\square)} \tilde{y}_\lambda$, which may be shown by repeating the main computation found in the proof of Theorem 5.5 of \cite{AM98} $k$ times. The third diagram is is mirror map applied to the second diagram, so it is equal to $s^{-2k \cn(\square)} \tilde{y}_\lambda$. This completes the proof.
\end{proof}


\renewcommand{\bibname}{References}

\addcontentsline{toc}{chapter}{\bibname}

\bibliographystyle{plain}
\bibliography{somerefs}

\end{document}